\documentclass[12pt]{amsart}

\usepackage[utf8]{inputenc}
\usepackage[T1]{fontenc}

\usepackage[backref=page,
  colorlinks=true,
  linkcolor=magenta,
  citecolor=cyan,
  hyperindex
]{hyperref}
\usepackage[alphabetic,lite,backrefs]{amsrefs}

\usepackage[table,dvipsnames]{xcolor}
\usepackage{amsmath, amsfonts, amsthm, amssymb}
\usepackage[top=1in, bottom=1in, left=1in, right=1in]{geometry}
\usepackage{tikz}
\usetikzlibrary{decorations.pathreplacing}
\usepackage{microtype}
\usepackage{changepage}
\usepackage{mathrsfs}
\usepackage{tikz-cd}
\usepackage{mathtools} % just for defining the symbol :=
\usepackage{bm} %Need this for bold math
\usepackage{comment} %Comments out large portions of text
\usepackage{float}
\usepackage{appendix}
\usepackage{enumitem} %% Need this to label individual items in enumerate
\usepackage[colorinlistoftodos,prependcaption,textsize=scriptsize, obeyFinal]{todonotes}

\makeatletter
\providecommand\@dotsep{5}
\renewcommand{\listoftodos}[1][\@todonotes@todolistname]{%
  \@starttoc{tdo}{#1}}
\makeatother

\theoremstyle{plain}
	\newtheorem{theorem}{Theorem}[section]
	\newtheorem{lemma}[theorem]{Lemma}
    \newtheorem{corollary}[theorem]{Corollary}
    \newtheorem{proposition}[theorem]{Proposition}

    \newtheorem{theoremalpha}{Theorem}
    
\theoremstyle{definition}
    \newtheorem{defn}[theorem]{Definition}
    \newtheorem*{defn*}{Definition}
    \newtheorem{example}{Example}[section]
    \newtheorem*{example*}{Example}
    
\theoremstyle{remark}
	\newtheorem{remark}[theorem]{Remark}
	\newtheorem*{remark*}{Remark}

\usepackage{mathtools}

\DeclarePairedDelimiterX\Set[1]\{\}{#1}

\AtBeginEnvironment{example}{%
  \pushQED{\qed}%
}
\AtEndEnvironment{example}{\popQED\endexample}
\AtBeginEnvironment{example*}{%
  \pushQED{\qed}%
}
\AtEndEnvironment{example*}{\popQED\endexample}

\newcommand{\mf}[1]{\mathfrak{#1}}

\newcommand{\mc}[1]{\mathcal{#1}}

\def\m{{\mathfrak{m}}}

\def\Z{{\mathbb{Z}}}

\def\N{{\mathbb{N}}}
\def\PP{{\mathbb{P}}}

\def\surj{\twoheadrightarrow}

\def\OO{{\mathcal O}}
\DeclareMathOperator{\Proj}{Proj}

\DeclareMathOperator{\rank}{rank}
\DeclareMathOperator{\reg}{reg}

\DeclareMathOperator{\Sym}{Sym}

\DeclareMathOperator{\Tor}{Tor}

\DeclareMathOperator{\ord}{ord}
\DeclareMathOperator{\gr}{gr}

\DeclareMathOperator{\Ap}{Ap}

\title{Weighted Syzygies of Pointed Curves}

\author{Maya Banks}
\address{Department of Mathematics, Statistics, and Computer Science, University of Illinois Chicago, Chicago, IL}
\email{mayadb@uic.edu}
\urladdr{\href{https://sites.google.com/view/mayabanks}{https://sites.google.com/view/mayabanks}}

\author{John Cobb}
\address{Department of Mathematics and Statistics, Auburn University, Auburn, AL}
\email{jdcobb3@gmail.com}
\urladdr{\href{https://johndcobb.github.io}{https://johndcobb.github.io}}

\author{Mahrud Sayrafi}
\address{Department of Mathematics and Statistics, McMaster University, Hamilton, ON}
\email{mahrud@mcmaster.ca}
\urladdr{\href{https://mahrud.github.io}{https://mahrud.github.io}}

\newcommand{\ww}{\mathbf w}
\newcommand{\vv}{\mathbf v}

\newcommand{\bd}{\cdot}
\newcommand{\sbd}{\cellcolor{red!15}\bd}
\newcommand{\BettiTable}[6]{%
  \begin{minipage}[t]{#1}
    \vspace{0pt}% anchor every table at the physical top of its minipage
    \quad{\tiny
      $\begin{aligned}
        d &= #2 \\[-0.5em]
        \ww &= (#3)\\[-0.5em]
        \operatorname{ord} &= (#4)
      \end{aligned}$\par\vspace{0.2em}}
         {\tiny
           \setlength{\arraycolsep}{1.8pt}%
           \renewcommand{\arraystretch}{0.82}%
           \(
           \begin{array}{r|#5}
             #6
           \end{array}
           \)
         }
  \end{minipage}%
}

\begin{document}
\begin{abstract}
  For a point $P$ on a smooth projective curve $C$ of genus $g$, the section ring $R_d = R(C,\OO_C(dP))$ can be minimally presented as a quotient $S_d/I_d$ where $S_d$ is a $\Z$-graded polynomial ring. Motivated by Green's $N_p$ properties for projective embeddings, we investigate the syzygies of $R_d$ over $S_d$ in low degrees $d$ when $R_d$ is \emph{not} generated in degree 1.
  We bound the degrees of the generators of $R_d$ and prove uniform column-by-column bounds on the support of the Betti table of $R_d$ over $S_d$. We compute the weighted regularity of $R_d$ and show that if $d$ is larger than the Frobenius number of $P$ then $R_d$ satisfies the weighted $N_{g-1-\binom{d-g}{2}}$ condition. 
  Finally, we give sufficient criteria for the Betti numbers to be determined explicitly and show that for ordinary points, the resolution of $R_{g+1}$ is pure.
\end{abstract}

%%%%%%%%%%%%%%%%%%%%%%%%%%%%%%%%%%%%%%%%%%%%%%%%%%%%%%%%%%%%%%%%%%%%%%%%%%%%%%%%

\vspace*{-0.3in}
\maketitle
\vspace{-0.25in}

\section{Introduction}

Let $C$ be a smooth projective curve of genus $g$ over an algebraically closed field $k$.\linebreak The section ring of a line bundle $L$ on $C$ is the graded ring
\begin{equation*}
    R(C,L) = \bigoplus_{m\geq 0} H^0(C,L^{\otimes m}).
\end{equation*}
When $L$ is sufficiently positive, $R(C,L)$ is generated in degree 1, and a complete linear series on $L$ induces a closed embedding of $C$ into $\PP H^0(C,L)^\vee$. An overarching principle is that the geometry of $L$ is reflected in the syzygies of the coordinate ring of $C$ under this embedding; this is borne out by a number of classical results on the ``geometry of syzygies'', highlighted by Green's theorem on linear syzygies which says that once $L$ has degree $2g+p+1$, the first $p$ syzygies of $C$ embedded by $L$ are linear. In general, as $L$ gets more positive, the syzygies of the corresponding embedding become simpler; e.g. see \cites{green1984koszul,lazarsfeld2004positivity1}.

\begin{tikzpicture}[>=stealth, font=\small]

    % Number line
    \draw[->, thick] (0,0) -- (13,0);

    % Tick marks
    \foreach \x in {0,1.2,5.8,7,10}
        \draw (\x,0.1) -- (\x,-0.1);

    % Left-hand label
    \node[below] at (-1,-0.32) {$\deg L =$};

    % Labels below the line
    \node[below] at (0,-.35) {$0$};
    \node[below] at (1.2, -.35) {$1$};
    \node[below] at (3.5,-.35) {$\cdots$};
    \node[below] at (5.8,-.35) {$2g+1$};
    \node[below] at (7,-.35) {$2g+2$};
    \node[below] at (8.5,-.35) {$\cdots$};
    \node[below] at (10,-.35) {$2g+p+1$};
    \node[below] at (11.5,-.35) {$\cdots$};

    % Angled text above the points
    %\node[above=6pt, rotate=30] at (0,0) {text};
    \node[above=6pt, rotate=30, anchor=north west] at (.9,.3) {ample};
    \node[above=6pt, rotate=30, anchor=north west] at (5.5,.3) {very ample};
    \node[above=6pt, rotate=30, anchor=north west, align=right] at (6.7,.3) {generated by\\quadrics};
    \node[above=6pt, rotate=30, anchor=north west] at (9.7,.3) {$N_p$};

    % Brace labeling the strictly ample range
    \draw[decorate, decoration={brace, amplitude=8pt, raise=4pt}] (1.2,.8) -- (5.5,.8)
        node[midway, above=10pt] {\emph{strictly ample}};

\end{tikzpicture}

Our goal is to understand the connection between positivity and syzygies on the left-hand side of the picture above, in the degree regime before $L$ is guaranteed to be very ample.
When $L$ is \emph{strictly ample}, that is, ample but not very ample, the section ring is not generated in degree 1. By picking a \emph{weighted series} $\{s_0, \ldots, s_r\}$ corresponding to a minimal algebra generating set for $R(C,L)$, we obtain a
closed embedding of $C=\Proj R(C,L)$ in the weighted projective space $\PP(w_0,\dots, w_r)$, where $s_i \in H^0(C, L^{\otimes w_i})$.
For the resulting family of embeddings obtained from the line bundles $\{L^{\otimes d}\}_{d\geq 1}$, we ask the following questions:

\begin{enumerate}
    \item What are the weights of the ambient $\PP(w_0,\dots,w_r)$? In other words, into which weighted projective space is the curve minimally embedded by $L^{\otimes d}$?
    \item What are the degrees of the generators of the defining ideal of $C\subset \PP(w_0,\dots,w_r)$? More generally, which are the nonzero entries of the Betti table of $C\subset\PP(w_0,\dots,w_r)$?
    \item How does the ``complexity'' of the syzygies of $C\subset\PP(w_0,\dots,w_r)$ change as $d$ grows?
\end{enumerate}
While these questions are well-studied when $d$ is sufficiently large, relatively little is known in the strictly ample range.

Weighted syzygies are often more delicate than their standard-graded counterparts: there are several different notions of linearity and regularity, and basic geometric questions for subvarieties of weighted projective spaces depend subtly on the numerics of the ambient weights. Recent work has developed several approaches to syzygies in the nonstandard and multigraded settings \cites{cobb2024syzygies, bruce2110characterizing, chardin2020multigraded, eisenbud2015tate, berkesch2017virtual,banks2026varieties}.
Brown and Erman developed weighted analogues of the $N_p$ conditions and Green's Linear Syzygy Theorem for curves embedded via certain weighted series. Their embeddings are constructed by taking an incomplete linear series on a very ample line bundle and completing using sections of a higher power of the line bundle \cite{brown2025linear}.
Our setting is different in that we explore the case of a strictly ample line bundle and the weighted embedding that arises naturally via the nonstandard grading of the section ring.

We address all three of the above questions for strictly ample line bundles of the form $L = \OO_C(dP)$ where $P$ is a point.
Fix a point $P\in C$ and an integer $d\geq 1$, and set
\begin{equation*}
    R_d \coloneqq R(C,\OO_C(dP)) = \bigoplus_{m\geq 0}H^0(C,\OO_C(mdP)).
\end{equation*}
For line bundles of this form, we are able to leverage additional combinatorial data associated to the point in order to study the section ring $R_d$ and its syzygies. That combinatorial data comes from a family of affine semigroups based off of the \emph{Weierstrass semigroup} of $P$, defined to be
\begin{equation*}
    H(P) = \{ - \ord_P(f) \mid f \in k(C)^\times \text{ is regular away from } P\}.
\end{equation*}
Weierstrass semigroups are classical invariants in the study of moduli of pointed curves, encoding rich geometric data. We use the Weierstrass semigroup to bound the behavior of the section rings $R_d$. Let $F(P)$, the \emph{Frobenius number} of $P$, be the maximum positive integer in $\Z_{\geq 0} \setminus H(P)$, and set
\[
q_d = \min\{q\geq 1 \mid qd \in H(P)\}.
\]
As a partial answer to Question (a) above, we obtain the following.

\begin{theoremalpha}[Corollary \ref{cor:minimal-generator-bounds}]\label{thm: intro-varDegrees}
Let $C$ be a smooth curve of genus $g\geq 1$ and $P$ a point of $C$. Let $C\hookrightarrow \PP(\ww)$ be the minimal weighted embedding induced by $\OO_C(dP)$. Then $\PP(\ww)$ has dimension $\leq dq_d$ and for every $0\leq i\leq \dim\PP(\ww)$,
    \[
    w_i\leq q_d + \left\lceil\frac{F(P)}{d}\right\rceil.
    \]
    Furthermore, when $d>F(P)$, $\PP(\ww)$ has the form $\PP(1^a,2^b)$ where $a = d+1-g$.
    % $\PP(\underbrace{1,\dots,1}_{a\text{ times}}, \underbrace{2,\dots,2}_{b\text{ times}})$
\end{theoremalpha}

Closely related to the bound on the weights, we also obtain a column-by-column vanishing criterion for the Betti numbers $\beta_{i,j}^{S_d}(R_d) = \dim_k \Tor_i^{S_d}(R_d,k)_j$.
Define $\vv = (v_1,\dots, v_s)$ to be the ordered weight vector obtained from $\ww$ by removing a single entry of 1 and a single entry of $q_d$ (we will see that these are always present as two distinct entries of $\ww$).  For a nondecreasing integer weight vector $\vv$ of length $s$ and $0\leq i\leq s$, set
\[
    \vv_i=v_1+\cdots+v_i
    \qquad\text{and}\qquad
    \vv^i=v_{s-i+1}+\cdots+v_s
\]
to be the sum of the $i$-lowest and $i$-highest weights of $\vv$, respectively, with $\vv_0=\vv^0=0$.

\begin{theoremalpha}\label{thm: intro-bettiSupport}
    Let $C$ be a smooth curve of genus $g \geq 1$ and $P$ a point of $C$. Let $R_d = S_d/I_d$ be the coordinate ring of the minimal weighted embedding $C\hookrightarrow \PP(\ww)$ induced by $\OO_C(dP)$. Then for $1\leq i \leq s$, the nonvanishing Betti numbers $\beta_{i,j}^{S_d}(R_d)\neq 0$ are concentrated in degrees
\begin{equation*}
        \mathbf{v}_i + v_1 \leq j \leq \mathbf{v}^i + q_d + \left\lceil \frac{F(P)}{d}\right\rceil.
    \end{equation*}
    Moreover, if $2v_1>q_d+\lceil F(P)/d\rceil$, then both inequalities are sharp for every $i$. The same conclusion holds when $P$ is ordinary and $d\mid g$ or $d\mid(g+1)$.
\end{theoremalpha}

We prove Theorem \ref{thm: intro-bettiSupport} via an Artinian reduction coming from the combinatorics of the Weierstrass semigroup and an examination of a certain associated graded algebra induced by a pole-order valuation on the section ring. In some cases, the Artinian reduction is the quotient by a monomial ideal whose resolution is explicitly known; in these cases, we understand the Betti numbers of $R_d$ completely. Proposition \ref{prop:square-zero-criteria} gives numerical criteria for the Artinian reduction to have this form, and Corollary \ref{cor:square-zero-sharpness} proves that both Betti bounds are then sharp. As a further result, Corollary \ref{cor: pure-resolution} shows that when $d=g+1$ and $P$ is ordinary, the minimal free resolution of $R_d$ is pure.

There are multiple ways to measure the ``complexity'' of the syzygies of $R_d$ as $d$ grows. The first that we will consider is via the \emph{(weighted) Castelnuovo--Mumford regularity} of $R_d$ (see Definition \ref{def: weighted reg}). 
Classically, the regularity measures precisely the height of the Betti table, though we note that in the weighted setting it measures the height of the Betti table after shifting all entries by a constant depending on the weights of the variables in the ambient polynomial ring \cite{symonds2011castelnuovo}.

Another way to measure the complexity of syzygies is via the $N_p$ conditions, which classically measure the linearity of the minimal free resolution. In the standard graded setting, the minimal free resolution of the coordinate ring of a curve embedded in $\PP^n$ by a line bundle $L$ remains linear for more steps as the degree of $L$ grows \cite{green1984koszul}. To analyze the weighted embeddings given by strictly ample line bundles, we use the following weighted analogue of the $N_p$ conditions developed by Brown--Erman in \cite{brown2025linear}.

\begin{defn}\label{def:weighted-np}
    Let $C\subset \PP(\ww) = \Proj S$ be a curve with weighted homogeneous coordinate ring $R = S/I$ and let $F_\bullet$ be the minimal free resolution of $R$ over $S$.
    \begin{enumerate}
        \item $R$ is \emph{normally generated} if $H^1_\m(R) = 0$, where $\m$ is the irrelevant ideal of $S$.
        \item $R$ satisfies the \emph{weighted $N_p$ condition} if $R$ is normally generated and $F_i$ is generated in degrees $\leq \ww^{i+1}$ for all $1\leq i\leq p$.
    \end{enumerate}
\end{defn}
Satisfying condition (b) of Definition \ref{def:weighted-np} is equivalent to saying that the growth of the syzygies of $R$ is bounded by the growth of the syzygies of the residue field for the first $p$ steps of the resolution. In the case of curves, satisfying the weighted $N_0$ condition is equivalent to the coordinate ring being Cohen--Macaulay.
Addressing Question (c), we prove the following.

\begin{theoremalpha}[Proposition \ref{prop: pointed section ring regularity}, Corollary \ref{cor:weighted Np}]\label{thm: intro-regNp}
    Let $C$ be a smooth curve of genus $g$ and $P$ a point of $C$. Let $R_d = S_d/I_d$ be the coordinate ring of the minimal weighted embedding $C\hookrightarrow \PP(\ww)$ induced by $\OO_C(dP)$. 
    \begin{enumerate}
        \item $R_d$ is Cohen--Macaulay with regularity
        \(
        \reg (R_d) = \left\lceil\frac{F(P)}{d}\right\rceil + 1.
        \)
        \item  If $d>F(P)$ and $\binom{d-g}{2}\leq g-1$, then $R_d$ satisfies the weighted $N_{g-1-\binom{d-g}{2}}$ condition.
    \end{enumerate}
\end{theoremalpha}

The regularity computation is a straightforward application of Riemann--Roch and Serre duality. The weighted $N_p$ result combines the Hilbert function of the Artinian reduction with Theorem \ref{thm: intro-bettiSupport}.
In Corollary \ref{cor:explicit-regularity}, we further show that $\reg(R_g)$ detects whether the point $P$ is ordinary.

Theorem \ref{thm: intro-regNp} highlights two competing features of the complexity of the syzygies of the family $R_d$ as $d$ changes. On the one hand, the regularity of $R_d$ is nonincreasing in $d$, giving a measure by which the syzygies grow ``simpler'' with $d$, even before the line bundle is very ample. On the other hand, within the range covered by part (b), the guaranteed weighted $N_p$ index
\[
    p=g-1-\binom{d-g}{2}
\]
increases as $d$ decreases toward $F(P)$. Thus the syzygies become simpler \emph{relative to the ambient weights} as $d$ becomes less positive.

\subsection{Examples}

\begin{example}\label{ex:betti tables of genus 4 hyperelliptic}
  Let $C$ be a hyperelliptic curve with genus $4$ and let $P\in C$ be the preimage of a branch point of the double cover $C\to\PP^1$. The point $P$ has Weierstrass semigroup $\{0,2,4,6,8,9,10,11,\dots\}$ with $F(P) = 7$.
  Let $L = \OO_C(dP)$ so that $\ww$ comes from a minimal algebra generating set of $R(C,\OO_C(dP))$. The following gives the Betti tables of $C\subset \PP(\ww)$ for $1\leq d\leq 9$. We display the weight vector $\ww$, where the two underlined entries are those removed from $\ww$ to obtain $\vv$.

  \begin{figure}[H]
% redefining to remove ord
\renewcommand{\BettiTable}[6]{%
  \begin{minipage}[t]{#1}
    \vspace{0pt}% anchor every table at the physical top of its minipage
    \quad{\tiny
      $\begin{aligned}
        d &= #2 \\[-0.5em]
        \ww &= (#3)
      \end{aligned}$\par\vspace{0.2em}}
         {\tiny
           \setlength{\arraycolsep}{1.8pt}%
           \renewcommand{\arraystretch}{0.82}%
           \(
           \begin{array}{r|#5}
             #6
           \end{array}
           \)
         }
  \end{minipage}%
}
\begin{adjustwidth}{-0.5in}{-0.3in}
\centering
%%%%%%%%%%%%%%%%%%%%%
% kk = ZZ/32003
% H  = {2, 9} -- genus 4
% R  = kk[x_0..x_5]/(x_3^2-x_2*x_4,x_2*x_3-x_1*x_4,x_1*x_3-x_0*x_4,x_2^2-x_0*x_4,x_1*x_2-x_0*x_3,x_1^2-x_0*x_2,x_4^3-x_3*x_5^2,x_3*x_4^2-x_2*x_5^2,x_2*x_4^2-x_1*x_5^2,x_1*x_4^2-x_0*x_5^2)
% pt = ideal(-x_0,-x_1,-x_2,-x_3,-x_4)
\BettiTable{0.15\linewidth}{1}{\underline{1},\underline{2},9}{0,2,9}{rr}{
        &0&1\\ \hline
        \text{total:}&1&1\\ \hline
        0:&1&\sbd\\
        1:&\sbd&\sbd\\
        2:&\sbd&\sbd\\
        3:&\sbd&\sbd\\
        4:&\sbd&\sbd\\
        5:&\sbd&\sbd\\
        6:&\sbd&\sbd\\
        7:&\sbd&\sbd\\
        8:&\sbd&\sbd\\
        9\text{--}16:&\sbd&\sbd\\
        %% 9:&\sbd&\sbd\\
        %% 10:&\sbd&\sbd\\
        %% 11:&\sbd&\sbd\\
        %% 12:&\sbd&\sbd\\
        %% 13:&\sbd&\sbd\\
        %% 14:&\sbd&\sbd\\
        %% 15:&\sbd&\sbd\\
        %% 16:&\sbd&\sbd\\
        17:&\sbd&1\\
        18:&\sbd&\sbd\\
}
\hspace{0.01\linewidth}%
\BettiTable{0.15\linewidth}{2}{\underline{1},\underline{1},5}{0,2,9}{rr}{
        &0&1\\ \hline
        \text{total:}&1&1\\ \hline
        0:&1&\sbd\\
        1:&\sbd&\sbd\\
        2:&\sbd&\sbd\\
        3:&\sbd&\sbd\\
        4:&\sbd&\sbd\\
        5:&\sbd&\sbd\\
        6:&\sbd&\sbd\\
        7:&\sbd&\sbd\\
        8:&\sbd&\sbd\\
        9:&\sbd&1\\
        10:&\sbd&\sbd\\
}
\hspace{0.01\linewidth}%
\BettiTable{0.15\linewidth}{3}{\underline{1},1,\underline{2},3}{0,2,6,9}{rrr}{
        &0&1&2\\ \hline
        \text{total:}&1&2&1\\ \hline
        0:&1&\sbd&\sbd\\
        1:&\sbd&\bd&\sbd\\
        2:&\sbd&1&\sbd\\
        3:&\sbd&\bd&\bd\\
        4:&\sbd&\bd&\bd\\
        5:&\sbd&1&\bd\\
        6:&\sbd&\bd&\bd\\
        7:&\sbd&\bd&1\\
        8:&\sbd&\sbd&\sbd\\
}
\hspace{0.01\linewidth}%
\BettiTable{0.15\linewidth}{4}{\underline{1},1,\underline{1},3,3}{0,2,4,9,11}{rrrr}{
        &0&1&2&3\\ \hline
        \text{total:}&1&6&8&3\\ \hline
        0:&1&\sbd&\sbd&\sbd\\
        1:&\sbd&1&\sbd&\sbd\\
        2:&\sbd&\bd&\sbd&\sbd\\
        3:&\sbd&2&2&\sbd\\
        4:&\sbd&\bd&\bd&\sbd\\
        5:&\sbd&3&4&1\\
        6:&\sbd&\sbd&\bd&\bd\\
        7:&\sbd&\sbd&2&2\\
        8:&\sbd&\sbd&\sbd&\sbd\\
}
\hspace{0.01\linewidth}%
\BettiTable{0.18\linewidth}{5}{\underline{1},1,1,2,\underline{2},3}{0,2,4,9,10,15}{rrrrr}{
        &0&1&2&3&4\\ \hline
        \text{total:}&1&9&17&12&3\\ \hline
        0:&1&\sbd&\sbd&\sbd&\sbd\\
        1:&\sbd&1&\bd&\sbd&\sbd\\
        2:&\sbd&2&2&\bd&\sbd\\
        3:&\sbd&4&4&1&\sbd\\
        4:&\sbd&1&6&3&\bd\\
        5:&\sbd&1&4&4&1\\
        6:&\sbd&\bd&1&4&1\\
        7:&\sbd&\sbd&\bd&\bd&1\\
        8:&\sbd&\sbd&\sbd&\sbd&\sbd\\
}
\hspace{0.01\linewidth}%
\par\vspace{0.1in}
\BettiTable{0.18\linewidth}{6}{\underline{1},1,1,\underline{1},2,2}{0,2,4,6,9,11}{rrrrr}{
        &0&1&2&3&4\\ \hline
        \text{total:}&1&9&16&9&1\\ \hline
        0:&1&\sbd&\sbd&\sbd&\sbd\\
        1:&\sbd&3&2&\sbd&\sbd\\
        2:&\sbd&3&6&3&\sbd\\
        3:&\sbd&3&6&3&\bd\\
        4:&\sbd&\bd&2&3&\bd\\
        5:&\sbd&\sbd&\bd&\bd&1\\
        6:&\sbd&\sbd&\sbd&\sbd&\sbd\\
}
\hspace{0.01\linewidth}%
\BettiTable{0.25\linewidth}{7}{\underline{1},1,1,1,2,2,2,\underline{2},3}{0,2,4,6,9,11,13,14,21}{rrrrrrrr}{
        &0&1&2&3&4&5&6&7\\ \hline
        \text{total:}&1&28&112&210&224&140&48&7\\ \hline
        0:&1&\sbd&\sbd&\sbd&\sbd&\sbd&\sbd&\sbd\\
        1:&\sbd&3&2&\bd&\sbd&\sbd&\sbd&\sbd\\
        2:&\sbd&9&18&9&\bd&\sbd&\sbd&\sbd\\
        3:&\sbd&12&39&42&15&\bd&\sbd&\sbd\\
        4:&\sbd&3&35&72&51&11&\bd&\sbd\\
        5:&\sbd&1&15&57&76&36&3&\sbd\\
        6:&\sbd&\sbd&3&27&60&51&15&\bd\\
        7:&\sbd&\sbd&\sbd&3&21&36&21&3\\
        8:&\sbd&\sbd&\sbd&\sbd&1&6&9&4\\
        9:&\sbd&\sbd&\sbd&\sbd&\sbd&\sbd&\sbd&\sbd\\
}
\hspace{0.01\linewidth}%
\BettiTable{0.25\linewidth}{8}{\underline{1},1,1,1,\underline{1},2,2,2,2}{0,2,4,6,8,9,11,13,15}{rrrrrrrr}{
        &0&1&2&3&4&5&6&7\\ \hline
        \text{total:}&1&28&112&210&224&140&48&7\\ \hline
        0:&1&\sbd&\sbd&\sbd&\sbd&\sbd&\sbd&\sbd\\
        1:&\sbd&6&8&3&\sbd&\sbd&\sbd&\sbd\\
        2:&\sbd&12&36&36&12&\sbd&\sbd&\sbd\\
        3:&\sbd&10&48&84&64&18&\sbd&\sbd\\
        4:&\sbd&\sbd&20&72&96&56&12&\sbd\\
        5:&\sbd&\sbd&\sbd&15&48&54&24&3\\
        6:&\sbd&\sbd&\sbd&\sbd&4&12&12&4\\
        7:&\sbd&\sbd&\sbd&\sbd&\sbd&\sbd&\sbd&\sbd\\
}
\hspace{0.01\linewidth}%
\BettiTable{0.18\linewidth}{9}{\underline{1},1,1,1,1,\underline{1}}{0,2,4,6,8,9}{rrrrr}{
        &0&1&2&3&4\\ \hline
        \text{total:}&1&10&20&15&4\\ \hline
        0:&1&\sbd&\sbd&\sbd&\sbd\\
        1:&\sbd&6&8&3&\bd\\
        2:&\sbd&4&12&12&4\\
        3:&\sbd&\sbd&\sbd&\sbd&\sbd\\
}
\hspace{0.01\linewidth}%
%%%%%%%%%%%%%%%%%%%%%
\end{adjustwidth}
\caption{Betti tables of $R(C,\OO_C(dP))$ for a $P$ a branch point on a genus 4 hyperelliptic curve. Dots denote zeroes and shaded entries are those whose vanishing is forced by Theorem \ref{thm: intro-bettiSupport}}\label{fig:betti-genus-4-hyperelliptic}
\end{figure}

\end{example}

\begin{example}\label{ex:betti tables of genus 4 space curve}
  Let $C=V(q,c)$ inside $\PP^3$ be the genus-4 complete intersection cut out by
  \begin{equation*}
    q=x_0x_1 + x_2x_3 \quad \text{and} \quad c=x_0^2x_2+x_0x_1^2 + x_0x_3^2 +x_1^3 + x_2^3 + x_3^3.
  \end{equation*}
  Let $P = (1:0:0:0)$. The point $P$ is ordinary and has Weierstrass semigroup $\{0, 5,6,7,8,9,\dots\}$ with $F(P) = 4$.
  The following are the Betti tables of the minimal embeddings of $C \subset \PP(\ww)$ by the line bundles $\OO_C(dP)$ for $1\leq d\leq 9$.
  \begin{figure}[H]
% redefining to remove ord
\renewcommand{\BettiTable}[6]{%
  \begin{minipage}[t]{#1}
    \vspace{0pt}% anchor every table at the physical top of its minipage
    \quad{\tiny
      $\begin{aligned}
        d &= #2 \\[-0.5em]
        \ww &= (#3)
      \end{aligned}$\par\vspace{0.2em}}
         {\tiny
           \setlength{\arraycolsep}{1.8pt}%
           \renewcommand{\arraystretch}{0.82}%
           \(
           \begin{array}{r|#5}
             #6
           \end{array}
           \)
         }
  \end{minipage}%
}
\begin{adjustwidth}{-0.5in}{-0.3in}
\centering
%%%%%%%%%%%%%%%%%%%%%
% kk = ZZ/32003
% H  = {5, 6, 7, 8, 9} -- genus 4
% R  = kk[x_0..x_3]/(x_0*x_1+x_2*x_3,x_0*x_1^2+x_1^3+x_0^2*x_2+x_2^3+x_0*x_3^2+x_3^3)
% pt = ideal(-x_1,-x_2,-x_3)
\BettiTable{0.18\linewidth}{1}{\underline{1},\underline{5},6,7,8,9}{0,5,6,7,8,9}{rrrrr}{
        &0&1&2&3&4\\ \hline
        \text{total:}&1&10&20&15&4\\ \hline
        0:&1&\sbd&\sbd&\sbd&\sbd\\
        1\text{--}10:&\sbd&\sbd&\sbd&\sbd&\sbd\\
        %% 1:&\sbd&\sbd&\sbd&\sbd&\sbd\\
        %% 2:&\sbd&\sbd&\sbd&\sbd&\sbd\\
        %% 3:&\sbd&\sbd&\sbd&\sbd&\sbd\\
        %% 4:&\sbd&\sbd&\sbd&\sbd&\sbd\\
        %% 5:&\sbd&\sbd&\sbd&\sbd&\sbd\\
        %% 6:&\sbd&\sbd&\sbd&\sbd&\sbd\\
        %% 7:&\sbd&\sbd&\sbd&\sbd&\sbd\\
        %% 8:&\sbd&\sbd&\sbd&\sbd&\sbd\\
        %% 9:&\sbd&\sbd&\sbd&\sbd&\sbd\\
        %% 10:&\sbd&\sbd&\sbd&\sbd&\sbd\\
        11:&\sbd&1&\sbd&\sbd&\sbd\\
        12:&\sbd&1&\sbd&\sbd&\sbd\\
        13:&\sbd&2&\sbd&\sbd&\sbd\\
        14:&\sbd&2&\sbd&\sbd&\sbd\\
        15:&\sbd&2&\sbd&\sbd&\sbd\\
        16:&\sbd&1&\sbd&\sbd&\sbd\\
        17:&\sbd&1&1&\sbd&\sbd\\
        18:&\sbd&\sbd&2&\sbd&\sbd\\
        19:&\sbd&\sbd&3&\sbd&\sbd\\
        20:&\sbd&\sbd&4&\sbd&\sbd\\
        21:&\sbd&\sbd&4&\sbd&\sbd\\
        22:&\sbd&\sbd&3&\sbd&\sbd\\
        23:&\sbd&\sbd&2&\sbd&\sbd\\
        24:&\sbd&\sbd&1&1&\sbd\\
        25:&\sbd&\sbd&\sbd&2&\sbd\\
        26:&\sbd&\sbd&\sbd&3&\sbd\\
        27:&\sbd&\sbd&\sbd&3&\sbd\\
        28:&\sbd&\sbd&\sbd&3&\sbd\\
        29:&\sbd&\sbd&\sbd&2&\sbd\\
        30:&\sbd&\sbd&\sbd&1&\sbd\\
        31:&\sbd&\sbd&\sbd&\sbd&\sbd\\
        32:&\sbd&\sbd&\sbd&\sbd&1\\
        33:&\sbd&\sbd&\sbd&\sbd&1\\
        34:&\sbd&\sbd&\sbd&\sbd&1\\
        35:&\sbd&\sbd&\sbd&\sbd&1\\
        36:&\sbd&\sbd&\sbd&\sbd&\sbd\\
}
\hspace{0.01\linewidth}%
\BettiTable{0.20\linewidth}{2}{\underline{1},3,\underline{3},4,4,5,5}{0,5,6,7,8,9,10}{rrrrrr}{
        &0&1&2&3&4&5\\ \hline
        \text{total:}&1&15&40&45&24&5\\ \hline
        0:&1&\sbd&\sbd&\sbd&\sbd&\sbd\\
        1:&\sbd&\sbd&\sbd&\sbd&\sbd&\sbd\\
        2:&\sbd&\sbd&\sbd&\sbd&\sbd&\sbd\\
        3:&\sbd&\sbd&\sbd&\sbd&\sbd&\sbd\\
        4:&\sbd&\sbd&\sbd&\sbd&\sbd&\sbd\\
        5:&\sbd&1&\sbd&\sbd&\sbd&\sbd\\
        6:&\sbd&2&\sbd&\sbd&\sbd&\sbd\\
        7:&\sbd&5&\sbd&\sbd&\sbd&\sbd\\
        8:&\sbd&4&2&\sbd&\sbd&\sbd\\
        9:&\sbd&3&6&\sbd&\sbd&\sbd\\
        10:&\sbd&\sbd&10&\sbd&\sbd&\sbd\\
        11:&\sbd&\sbd&12&1&\sbd&\sbd\\
        12:&\sbd&\sbd&8&6&\sbd&\sbd\\
        13:&\sbd&\sbd&2&11&\sbd&\sbd\\
        14:&\sbd&\sbd&\sbd&14&\sbd&\sbd\\
        15:&\sbd&\sbd&\sbd&9&2&\sbd\\
        16:&\sbd&\sbd&\sbd&4&6&\sbd\\
        17:&\sbd&\sbd&\sbd&\sbd&8&\sbd\\
        18:&\sbd&\sbd&\sbd&\sbd&6&\sbd\\
        19:&\sbd&\sbd&\sbd&\sbd&2&1\\
        20:&\sbd&\sbd&\sbd&\sbd&\sbd&2\\
        21:&\sbd&\sbd&\sbd&\sbd&\sbd&2\\
        22:&\sbd&\sbd&\sbd&\sbd&\sbd&\sbd\\
}
\hspace{0.01\linewidth}%
\BettiTable{0.18\linewidth}{3}{\underline{1},2,\underline{2},3,3,3}{0,5,6,7,8,9}{rrrrr}{
        &0&1&2&3&4\\ \hline
        \text{total:}&1&10&20&15&4\\ \hline
        0:&1&\sbd&\sbd&\sbd&\sbd\\
        1:&\sbd&\sbd&\sbd&\sbd&\sbd\\
        2:&\sbd&\sbd&\sbd&\sbd&\sbd\\
        3:&\sbd&\bd&\sbd&\sbd&\sbd\\
        4:&\sbd&3&\sbd&\sbd&\sbd\\
        5:&\sbd&7&\bd&\sbd&\sbd\\
        6:&\sbd&\bd&9&\sbd&\sbd\\
        7:&\sbd&\sbd&11&\bd&\sbd\\
        8:&\sbd&\sbd&\bd&9&\sbd\\
        9:&\sbd&\sbd&\sbd&6&\bd\\
        10:&\sbd&\sbd&\sbd&\bd&3\\
        11:&\sbd&\sbd&\sbd&\sbd&1\\
        12:&\sbd&\sbd&\sbd&\sbd&\sbd\\
}
\hspace{0.01\linewidth}%
\BettiTable{0.25\linewidth}{4}{\underline{1},2,2,2,\underline{2},3,3,3,3}{0,5,6,7,8,9,10,11,12}{rrrrrrrr}{
        &0&1&2&3&4&5&6&7\\ \hline
        \text{total:}&1&28&112&210&224&140&48&7\\ \hline
        0:&1&\sbd&\sbd&\sbd&\sbd&\sbd&\sbd&\sbd\\
        1:&\sbd&\sbd&\sbd&\sbd&\sbd&\sbd&\sbd&\sbd\\
        2:&\sbd&\sbd&\sbd&\sbd&\sbd&\sbd&\sbd&\sbd\\
        3:&\sbd&6&\sbd&\sbd&\sbd&\sbd&\sbd&\sbd\\
        4:&\sbd&12&8&\sbd&\sbd&\sbd&\sbd&\sbd\\
        5:&\sbd&10&36&3&\sbd&\sbd&\sbd&\sbd\\
        6:&\sbd&\sbd&48&36&\sbd&\sbd&\sbd&\sbd\\
        7:&\sbd&\sbd&20&84&12&\sbd&\sbd&\sbd\\
        8:&\sbd&\sbd&\sbd&72&64&\sbd&\sbd&\sbd\\
        9:&\sbd&\sbd&\sbd&15&96&18&\sbd&\sbd\\
        10:&\sbd&\sbd&\sbd&\sbd&48&56&\sbd&\sbd\\
        11:&\sbd&\sbd&\sbd&\sbd&4&54&12&\sbd\\
        12:&\sbd&\sbd&\sbd&\sbd&\sbd&12&24&\sbd\\
        13:&\sbd&\sbd&\sbd&\sbd&\sbd&\sbd&12&3\\
        14:&\sbd&\sbd&\sbd&\sbd&\sbd&\sbd&\sbd&4\\
        15:&\sbd&\sbd&\sbd&\sbd&\sbd&\sbd&\sbd&\sbd\\
}
\hspace{0.01\linewidth}%
\par\vspace{0.1in}
\BettiTable{0.18\linewidth}{5}{\underline{1},\underline{1},2,2,2,2}{0,5,7,8,9,10}{rrrrr}{
        &0&1&2&3&4\\ \hline
        \text{total:}&1&10&20&15&4\\ \hline
        0:&1&\sbd&\sbd&\sbd&\sbd\\
        1:&\sbd&\sbd&\sbd&\sbd&\sbd\\
        2:&\sbd&\sbd&\sbd&\sbd&\sbd\\
        3:&\sbd&10&\sbd&\sbd&\sbd\\
        4:&\sbd&\sbd&20&\sbd&\sbd\\
        5:&\sbd&\sbd&\sbd&15&\sbd\\
        6:&\sbd&\sbd&\sbd&\sbd&4\\
        7:&\sbd&\sbd&\sbd&\sbd&\sbd\\
}
\hspace{0.01\linewidth}%
\BettiTable{0.18\linewidth}{6}{\underline{1},1,\underline{1},2,2,2}{0,5,6,10,11,12}{rrrrr}{
        &0&1&2&3&4\\ \hline
        \text{total:}&1&10&20&15&4\\ \hline
        0:&1&\sbd&\sbd&\sbd&\sbd\\
        1:&\sbd&\bd&\sbd&\sbd&\sbd\\
        2:&\sbd&4&\bd&\sbd&\sbd\\
        3:&\sbd&6&12&\bd&\sbd\\
        4:&\sbd&\sbd&8&12&\bd\\
        5:&\sbd&\sbd&\sbd&3&4\\
        6:&\sbd&\sbd&\sbd&\sbd&\sbd\\
}
\hspace{0.01\linewidth}%
\BettiTable{0.15\linewidth}{7}{\underline{1},1,1,\underline{1},2}{0,5,6,7,14}{rrrr}{
        &0&1&2&3\\ \hline
        \text{total:}&1&7&10&4\\ \hline
        0:&1&\sbd&\sbd&\sbd\\
        1:&\sbd&\bd&\bd&\sbd\\
        2:&\sbd&6&4&\bd\\
        3:&\sbd&1&6&4\\
        4:&\sbd&\sbd&\sbd&\sbd\\
}
\hspace{0.01\linewidth}%
\BettiTable{0.15\linewidth}{8}{\underline{1},1,1,1,\underline{1}}{0,5,6,7,8}{rrrr}{
        &0&1&2&3\\ \hline
        \text{total:}&1&6&9&4\\ \hline
        0:&1&\sbd&\sbd&\sbd\\
        1:&\sbd&2&\bd&\bd\\
        2:&\sbd&4&9&4\\
        3:&\sbd&\sbd&\sbd&\sbd\\
}
\hspace{0.01\linewidth}%
\BettiTable{0.18\linewidth}{9}{\underline{1},1,1,1,1,\underline{1}}{0,5,6,7,8,9}{rrrrr}{
        &0&1&2&3&4\\ \hline
        \text{total:}&1&6&13&12&4\\ \hline
        0:&1&\sbd&\sbd&\sbd&\sbd\\
        1:&\sbd&6&4&\bd&\bd\\
        2:&\sbd&\bd&9&12&4\\
        3:&\sbd&\sbd&\sbd&\sbd&\sbd\\
}
\hspace{0.01\linewidth}%
%%%%%%%%%%%%%%%%%%%%%
\end{adjustwidth}
\caption{Betti tables of $R(C,\OO_C(dP))$ for an ordinary point on a genus 4 complete intersection in $\PP^3$.} \label{fig:betti-genus-4-CI}
\end{figure}

  For the ordinary genus 4 case displayed in this example, Proposition \ref{prop:square-zero-criteria} identifies precisely the tables with $d=1,2,4,5$ as those for which the bound of Theorem \ref{thm: intro-bettiSupport} are guaranteed to be sharp; Corollary \ref{cor: pure-resolution} shows that the resolution for $d=g+1$ (in this example $d=5$) is always pure. Furthermore, part (b) of Theorem \ref{thm: intro-regNp} identifies that weighted $N_3$ is satisfied in the $d=5$ table and that weighted $N_2$ is satisfied for $d=6$. Note that weighted $N_3$ is also satisfied for $d=4$, but this lies outside of the range covered by the theorem. We can also observe that while the total Betti numbers for several of the tables are the same as those of an Eagon--Northcott complex, the resolution does not in general have this form. For instance, when $d=3$ one can show purely by degree reasons that the relations cannot be given by $2\times 2$ minors of a $2\times 5$ matrix.
\end{example}
  
\begin{example}
Let $C$ be a general genus 4 curve and $P$ a point of $C$ with Weierstrass semigroup $\{0,4,6,7,8,9,\ldots\}$, which has $F(P) = 5$. The following are the Betti tables of the minimal embeddings given by $\OO_C(dP)$ for $1\leq d\leq 9$

  \begin{figure}[H]
% redefining to remove ord
\renewcommand{\BettiTable}[6]{%
  \begin{minipage}[t]{#1}
    \vspace{0pt}% anchor every table at the physical top of its minipage
    \quad{\tiny
      $\begin{aligned}
        d &= #2 \\[-0.5em]
        \ww &= (#3)
      \end{aligned}$\par\vspace{0.2em}}
         {\tiny
           \setlength{\arraycolsep}{1.8pt}%
           \renewcommand{\arraystretch}{0.82}%
           \(
           \begin{array}{r|#5}
             #6
           \end{array}
           \)
         }
  \end{minipage}%
}
\begin{adjustwidth}{-0.5in}{-0.3in}
\centering
%%%%%%%%%%%%%%%%%%%%%
% kk = ZZ/3
% H  = {4, 6, 7, 9} -- genus 4
% R  = kk[x_0..x_5]/(x_1*x_3-x_0*x_4-x_1*x_4-x_3*x_4-x_0*x_5+x_2*x_5,x_2^2-x_2*x_3+x_3^2-x_0*x_4-x_3*x_4+x_4^2-x_0*x_5+x_1*x_5+x_2*x_5-x_3*x_5,x_1*x_2-x_2*x_3+x_2*x_4+x_4^2+x_0*x_5+x_1*x_5-x_3*x_5,x_0*x_2+x_2*x_3+x_3^2-x_0*x_4-x_1*x_4+x_2*x_4-x_4^2+x_0*x_5-x_1*x_5+x_3*x_5,x_1^2-x_0*x_4,x_0*x_1-x_0*x_3+x_2*x_3-x_0*x_4+x_3*x_4+x_0*x_5-x_1*x_5-x_2*x_5,x_3^2*x_4+x_1*x_4^2-x_2*x_4^2+x_0^2*x_5-x_0*x_3*x_5-x_2*x_3*x_5+x_3^2*x_5+x_0*x_4*x_5+x_2*x_4*x_5-x_3*x_4*x_5-x_4^2*x_5+x_1*x_5^2+x_2*x_5^2+x_3*x_5^2,x_2*x_3*x_4-x_1*x_4^2-x_0*x_3*x_5-x_2*x_3*x_5-x_0*x_4*x_5-x_1*x_4*x_5-x_2*x_4*x_5+x_3*x_4*x_5-x_4^2*x_5+x_1*x_5^2-x_2*x_5^2+x_3*x_5^2,x_3^3-x_0*x_3*x_4-x_0*x_4^2+x_1*x_4^2+x_2*x_4^2-x_4^3+x_2*x_3*x_5-x_3^2*x_5+x_0*x_4*x_5+x_1*x_4*x_5+x_2*x_4*x_5-x_3*x_4*x_5-x_0*x_5^2-x_2*x_5^2,x_2*x_3^2-x_0*x_4^2-x_1*x_4^2-x_3*x_4^2-x_0^2*x_5+x_0*x_3*x_5+x_2*x_3*x_5+x_0*x_4*x_5-x_1*x_4*x_5+x_2*x_4*x_5+x_0*x_5^2-x_1*x_5^2,x_0*x_3^2-x_0^2*x_4-x_0*x_3*x_4+x_0*x_4^2+x_1*x_4^2-x_2*x_4^2-x_0^2*x_5+x_2*x_3*x_5+x_1*x_4*x_5+x_2*x_4*x_5-x_3*x_4*x_5+x_4^2*x_5+x_0*x_5^2+x_2*x_5^2-x_3*x_5^2,x_0*x_3*x_4^2-x_0*x_4^3-x_1*x_4^3-x_2*x_4^3+x_3*x_4^3+x_4^4+x_0^2*x_3*x_5-x_0^2*x_4*x_5-x_1*x_4^2*x_5+x_2*x_4^2*x_5+x_0^2*x_5^2+x_2*x_3*x_5^2-x_3^2*x_5^2+x_0*x_4*x_5^2+x_0*x_5^3,x_0^2*x_4^2+x_0*x_4^3-x_1*x_4^3-x_2*x_4^3+x_0^3*x_5+x_0^2*x_4*x_5+x_0*x_3*x_4*x_5-x_0*x_4^2*x_5-x_1*x_4^2*x_5-x_3*x_4^2*x_5-x_4^3*x_5-x_0^2*x_5^2-x_3^2*x_5^2-x_0*x_4*x_5^2-x_1*x_4*x_5^2+x_4^2*x_5^2-x_1*x_5^3+x_2*x_5^3-x_3*x_5^3)
% pt = ideal(-x_0,-x_1,-x_2,-x_3,-x_4)
\BettiTable{0.15\linewidth}{1}{\underline{1},\underline{4},6,7,9}{0,4,6,7,9}{rrrr}{
        &0&1&2&3\\ \hline
        \text{total:}&1&6&8&3\\ \hline
        0:&1&\sbd&\sbd&\sbd\\
        1\text{--}10:&\sbd&\sbd&\sbd&\sbd\\
        % 1:&\sbd&\sbd&\sbd&\sbd\\
        % 2:&\sbd&\sbd&\sbd&\sbd\\
        % 3:&\sbd&\sbd&\sbd&\sbd\\
        % 4:&\sbd&\sbd&\sbd&\sbd\\
        % 5:&\sbd&\sbd&\sbd&\sbd\\
        % 6:&\sbd&\sbd&\sbd&\sbd\\
        % 7:&\sbd&\sbd&\sbd&\sbd\\
        % 8:&\sbd&\sbd&\sbd&\sbd\\
        % 9:&\sbd&\sbd&\sbd&\sbd\\
        % 10:&\sbd&\sbd&\sbd&\sbd\\
        11:&\sbd&1&\sbd&\sbd\\
        12:&\sbd&1&\sbd&\sbd\\
        13:&\sbd&1&\sbd&\sbd\\
        14:&\sbd&1&\sbd&\sbd\\
        15:&\sbd&1&\sbd&\sbd\\
        16:&\sbd&\bd&\sbd&\sbd\\
        17:&\sbd&1&1&\sbd\\
        18:&\sbd&\sbd&1&\sbd\\
        19:&\sbd&\sbd&1&\sbd\\
        20:&\sbd&\sbd&2&\sbd\\
        21:&\sbd&\sbd&1&\sbd\\
        22:&\sbd&\sbd&1&\sbd\\
        23:&\sbd&\sbd&1&\sbd\\
        24:&\sbd&\sbd&\sbd&\sbd\\
        25:&\sbd&\sbd&\sbd&1\\
        26:&\sbd&\sbd&\sbd&1\\
        27:&\sbd&\sbd&\sbd&\bd\\
        28:&\sbd&\sbd&\sbd&1\\
        29:&\sbd&\sbd&\sbd&\sbd\\
}
\hspace{0.01\linewidth}%
\BettiTable{0.15\linewidth}{2}{\underline{1},\underline{2},3,4,5}{0,4,6,7,9}{rrrr}{
        &0&1&2&3\\ \hline
        \text{total:}&1&6&8&3\\ \hline
        0:&1&\sbd&\sbd&\sbd\\
        1:&\sbd&\sbd&\sbd&\sbd\\
        2:&\sbd&\sbd&\sbd&\sbd\\
        3:&\sbd&\sbd&\sbd&\sbd\\
        4:&\sbd&\sbd&\sbd&\sbd\\
        5:&\sbd&1&\sbd&\sbd\\
        6:&\sbd&1&\sbd&\sbd\\
        7:&\sbd&2&\sbd&\sbd\\
        8:&\sbd&1&1&\sbd\\
        9:&\sbd&1&2&\sbd\\
        10:&\sbd&\sbd&2&\sbd\\
        11:&\sbd&\sbd&2&\sbd\\
        12:&\sbd&\sbd&1&1\\
        13:&\sbd&\sbd&\sbd&1\\
        14:&\sbd&\sbd&\sbd&1\\
        15:&\sbd&\sbd&\sbd&\sbd\\
}
\hspace{0.01\linewidth}%
\BettiTable{0.20\linewidth}{3}{\underline{1},2,\underline{2},3,3,3,4}{0,4,6,7,8,9,11}{rrrrrr}{
        &0&1&2&3&4&5\\ \hline
        \text{total:}&1&15&40&45&24&5\\ \hline
        0:&1&\sbd&\sbd&\sbd&\sbd&\sbd\\
        1:&\sbd&\sbd&\sbd&\sbd&\sbd&\sbd\\
        2:&\sbd&\sbd&\sbd&\sbd&\sbd&\sbd\\
        3:&\sbd&1&\sbd&\sbd&\sbd&\sbd\\
        4:&\sbd&3&\sbd&\sbd&\sbd&\sbd\\
        5:&\sbd&7&3&\sbd&\sbd&\sbd\\
        6:&\sbd&3&10&\sbd&\sbd&\sbd\\
        7:&\sbd&1&14&3&\sbd&\sbd\\
        8:&\sbd&\sbd&10&12&\sbd&\sbd\\
        9:&\sbd&\sbd&3&15&1&\sbd\\
        10:&\sbd&\sbd&\sbd&12&6&\sbd\\
        11:&\sbd&\sbd&\sbd&3&10&\sbd\\
        12:&\sbd&\sbd&\sbd&\sbd&6&1\\
        13:&\sbd&\sbd&\sbd&\sbd&1&3\\
        14:&\sbd&\sbd&\sbd&\sbd&\sbd&1\\
        15:&\sbd&\sbd&\sbd&\sbd&\sbd&\sbd\\
}
\hspace{0.01\linewidth}%
\BettiTable{0.15\linewidth}{4}{\underline{1},\underline{1},2,2,3}{0,4,6,7,9}{rrrr}{
        &0&1&2&3\\ \hline
        \text{total:}&1&6&8&3\\ \hline
        0:&1&\sbd&\sbd&\sbd\\
        1:&\sbd&\sbd&\sbd&\sbd\\
        2:&\sbd&\sbd&\sbd&\sbd\\
        3:&\sbd&3&\sbd&\sbd\\
        4:&\sbd&2&2&\sbd\\
        5:&\sbd&1&4&\sbd\\
        6:&\sbd&\sbd&2&2\\
        7:&\sbd&\sbd&\sbd&1\\
        8:&\sbd&\sbd&\sbd&\sbd\\
}
\hspace{0.01\linewidth}%
\BettiTable{0.20\linewidth}{5}{\underline{1},1,2,2,2,\underline{2},3}{0,4,6,7,9,10,15}{rrrrrr}{
        &0&1&2&3&4&5\\ \hline
        \text{total:}&1&15&40&45&24&5\\ \hline
        0:&1&\sbd&\sbd&\sbd&\sbd&\sbd\\
        1:&\sbd&\bd&\sbd&\sbd&\sbd&\sbd\\
        2:&\sbd&1&\bd&\sbd&\sbd&\sbd\\
        3:&\sbd&10&3&\bd&\sbd&\sbd\\
        4:&\sbd&3&21&3&\bd&\sbd\\
        5:&\sbd&1&13&18&1&\sbd\\
        6:&\sbd&\sbd&3&21&7&\bd\\
        7:&\sbd&\sbd&\sbd&3&15&1\\
        8:&\sbd&\sbd&\sbd&\sbd&1&4\\
        9:&\sbd&\sbd&\sbd&\sbd&\sbd&\sbd\\
}
\hspace{0.01\linewidth}%
\BettiTable{0.18\linewidth}{6}{\underline{1},1,\underline{1},2,2,2}{0,4,6,7,9,11}{rrrrr}{
        &0&1&2&3&4\\ \hline
        \text{total:}&1&10&20&15&4\\ \hline
        0:&1&\sbd&\sbd&\sbd&\sbd\\
        1:&\sbd&\bd&\sbd&\sbd&\sbd\\
        2:&\sbd&4&\bd&\sbd&\sbd\\
        3:&\sbd&6&12&\bd&\sbd\\
        4:&\sbd&\sbd&8&12&\bd\\
        5:&\sbd&\sbd&\sbd&3&4\\
        6:&\sbd&\sbd&\sbd&\sbd&\sbd\\
}
\hspace{0.01\linewidth}%
\BettiTable{0.15\linewidth}{7}{\underline{1},1,1,\underline{1},2}{0,4,6,7,9}{rrrr}{
        &0&1&2&3\\ \hline
        \text{total:}&1&7&10&4\\ \hline
        0:&1&\sbd&\sbd&\sbd\\
        1:&\sbd&\bd&\bd&\sbd\\
        2:&\sbd&6&4&\bd\\
        3:&\sbd&1&6&4\\
        4:&\sbd&\sbd&\sbd&\sbd\\
}
\hspace{0.01\linewidth}%
\BettiTable{0.15\linewidth}{8}{\underline{1},1,1,1,\underline{1}}{0,4,6,7,8}{rrrr}{
        &0&1&2&3\\ \hline
        \text{total:}&1&6&9&4\\ \hline
        0:&1&\sbd&\sbd&\sbd\\
        1:&\sbd&2&\bd&\bd\\
        2:&\sbd&4&9&4\\
        3:&\sbd&\sbd&\sbd&\sbd\\
}
\hspace{0.01\linewidth}%
\BettiTable{0.18\linewidth}{9}{\underline{1},1,1,1,1,\underline{1}}{0,4,6,7,8,9}{rrrrr}{
        &0&1&2&3&4\\ \hline
        \text{total:}&1&6&13&12&4\\ \hline
        0:&1&\sbd&\sbd&\sbd&\sbd\\
        1:&\sbd&6&4&\bd&\bd\\
        2:&\sbd&\bd&9&12&4\\
        3:&\sbd&\sbd&\sbd&\sbd&\sbd\\
}
\hspace{0.01\linewidth}%
%%%%%%%%%%%%%%%%%%%%%
\end{adjustwidth}
\caption{Betti tables of $R(C,\OO_C(dP))$ for $P$ on a general genus 4 curve such that the Weierstrass semigroup has gap sequence $\{1,\dots,g-1,g+1\}$.}\label{fig:betti-genus-4-new}
\end{figure}

\end{example}

As predicted by Theorem \ref{thm: intro-regNp}, we can see that $\OO(6P)$ satisfies the weighted $N_2$ condition but not the weighted $N_3$ condition. Furthermore, none of the others satisfy any $N_p$ conditions for $p\geq 1$ except for the very ample $\OO(9P)$.

%% \begin{example}
%%   Here are two example pointed genus 5 curves with Weierstrass semigroups $\{4,6,7,\dots\}$ and $\{5,6,7,8\}$, where, far before the $2g+1=11$ bound, $7P$ and $8P$ induce very ample embeddings in $\PP^3$ and $\PP^4$, respectively, but $9P$ and $10P$ are strictly ample. Notice that the embedding by $d = 2g-2 = 8$ is Gorenstein.
%%   \input{figs/fig-genus-5-betti-tables-4-6-7.tex}
%%   \input{figs/fig-genus-5-betti-tables-5-6-7-8.tex}
%% \end{example}

%\begin{corollary}\label{cor: pure-resolution}
%Let $C$ be a smooth projective curve of genus $g$, and let $P\in C$ be an ordinary point. The minimal free resolution of $R_{g+1}$ is pure with Betti numbers
%\begin{equation*}
%    \beta_{i,2i+2}^{S_{g+1}}(R_{g+1}) = i\binom{g+1}{i+1} \quad 1\leq i \leq g.
%\end{equation*}
%\end{corollary}

\subsection*{Organization}
Section \ref{sec-weighted-presentations} collects preliminaries on section rings and Weierstrass semigroups. Section \ref{sec:artinian-reduction} establishes the Cohen--Macaulay and regularity assertions in Theorem \ref{thm: intro-regNp}, constructs the Artinian reduction, and proves Theorem \ref{thm: intro-varDegrees}, the support assertion in Theorem \ref{thm: intro-bettiSupport}, and the weighted $N_p$ assertion in Theorem \ref{thm: intro-regNp}. Finally, Section \ref{sec:square-zeroreduction} proves the sharpness assertion in Theorem \ref{thm: intro-bettiSupport} and the pure-resolution statement of Corollary \ref{cor: pure-resolution}.

\subsection*{Acknowledgments}
We thank Daniel Erman, David Eisenbud, Mike Stillman, Hal Schenck, and Michael Brown for helpful conversations and comments. This project began at the Fields Institute in Toronto during the 2025 thematic program on Commutative Algebra and Applications. Cobb was supported by National Science Foundation Grant DMS-2402199. Banks was partially supported by National Science Foundation grant DMS-2037569. The software system Macaulay2 \cite{M2} and computing resources provided by the Digital Research Alliance of Canada were helpful in the development of this paper.

%%%%%%%%%%%%%%%%%%%%%%%%%%%%%%%%%%%%%%%%%%%%%%%%%%%%%%%%%%%%%%%%%%%%%%%%%%%%%%%%

\section{Preliminaries}\label{sec-weighted-presentations}

\subsection{Pointed section rings} Let $R=\bigoplus_{m\geq 0}R_m$ be a finitely generated $\N$-graded $k$-algebra with $R_0=k$. Write
\begin{equation*}
    R_+ = \bigoplus_{m>0}R_m.
\end{equation*}
By lifting a basis for the graded vector space $R_+/(R_+)^2$, we may choose a homogeneous basis
\begin{equation*}
     f_0,\ldots,f_r \in R.
\end{equation*}
Writing $w_i=\deg(f_i)$, we obtain a weighted presentation
\begin{equation*}
    S=k[x_0,\ldots,x_r]\surj R,\qquad x_i\mapsto f_i,\qquad \deg(x_i)=w_i,
\end{equation*}
minimal in the sense that $f_0,\dots, f_r$ are a minimal homogeneous $k$-algebra generating set. Note that the multiset of degrees $w_i$ is independent of the choice of basis.
We regard $S$ as a weighted polynomial ring and write $\PP(w_0,\dots, w_r)=\Proj S$ for the corresponding weighted projective space.

For a line bundle $L$ on a smooth projective curve $C$, the section ring is
\begin{equation*}
    R(C,L)=\bigoplus_{m\geq 0}H^0(C,L^{\otimes m}).
\end{equation*}
When $L$ has positive degree, $L$ is ample and $\Proj R(C,L)\cong C$. Thus a minimal weighted presentation of $S \surj R(C,L)$ realizes $C$ as a closed subvariety of weighted projective space. If $I$ is the kernel of this presentation, then the homogeneous coordinate ring of this weighted embedding is $S/I \cong R(C,L)$. In this sense, throughout the paper, the section ring and the homogeneous coordinate ring of the corresponding minimal weighted embedding are the same object. In our setting, we apply this construction to $R_d=R(C,\OO_C(dP))$, and refer to the corresponding embedding as the \emph{minimal weighted embedding} induced by $\OO_C(dP)$. We denote by $S_d = k[x_0,\ldots,x_n]$ the minimal weighted polynomial ring surjecting onto $R_d$, and we use $\ww = (w_0, \ldots, w_n)$ to denote the degrees of the $x_i$ (equivalently the weights of $\Proj S_d$), written in nondecreasing order. When we refer to the Betti numbers of $R_d$, we will always mean as a module over $S_d$. For $0\leq a\leq n+1$, let
\begin{equation*}
    \ww_a=w_0+\cdots+w_{a-1} \qquad \text{and} \qquad  \ww^a=w_{n-a+1}+\cdots+w_n,
\end{equation*}
the sums of the $a$ smallest and $a$ largest weights, respectively, with $\ww_0=\ww^0=0$.

\subsection{Weierstrass semigroups}
The behavior of the Betti numbers of $R_d = R(C,\OO_C(dP))$ is heavily dependent on the \emph{Weierstrass semigroup} of $P$.

\begin{defn}\label{def: weierstrass semigroup}
    Let $P$ be a point on a curve $C$. The \emph{Weierstrass semigroup} of $P$ is the set
    \begin{equation*}
    H(P) =\{ -\ord_P(f) \, | \, f\in k(C)^\times, \, f \text{ is regular away from } P \}.
\end{equation*}
This has the structure of a numerical semigroup since for $f,g\in k(C)$, $\ord_P(fg) = \ord_P(f)+\ord_P(g)$. The complement of $H(P)$ in $\Z_{\geq 0}$ is called the \emph{gap sequence}, denoted $G(P)$. The cardinality of $G(P)$ is exactly $g$ and, by Riemann--Roch, $G(P)$ is bounded above by $2g-1$.
\end{defn} 

Weierstrass semigroups are classically well-studied objects, particularly for their connection to the study of $\mc M_{g,1}$. The Weierstrass semigroups stratify $\mc M_{g,1}$ into locally closed subvarieties whose dimension can be bounded using data of the gap sequence \cites{Pinkham1974Deformations, pflueger2018nonprimitive}. When $\operatorname{char} k = 0$, a dense open subset of $\mc M_{g,1}$ belongs to the stratum with gap sequence $G = \{1,2, \ldots, g\}$, and for a general smooth curve the only other gap sequence occurring is $G = \{1,2,\ldots, g-1, g+1\}$. For arbitrary curves, the question of which strata are nonempty remains open and is a subject of active inquiry \cite{eisenbud2026minimalnonweierstrasssemigroups}.

The Weierstrass semigroup of $P$ and the Hilbert function of $\OO_C(P)$ are computable from one another via the following fact:
\begin{equation*}
    k \in H(P) \iff h^0(\OO_C(kP)) = h^0(\OO_C((k-1)P))+1.
\end{equation*}
By Riemann--Roch and Serre duality, the vanishing orders at $P$ of nonzero canonical sections are precisely $\ell-1$ for $\ell\in G(P)$.

\begin{defn}\label{def: frobenius number}
    The \emph{Frobenius number} of $H(P)$, which we will also refer to as the Frobenius number of the point $P$, is the largest element of $G(P)$; we denote the Frobenius number by $F(P)$.
We call $P$ \emph{ordinary} if $G(P) =\{1,\dots, g\}$. In characteristic $0$, a general point of $C$ is ordinary. For $c\in H(P) \setminus \{0\}$, define the \emph{Ap\'{e}ry set}
\begin{equation*}
    \Ap(H(P),c) = \{ n \in H(P) \mid n-c \notin H(P)\}.
\end{equation*}
\end{defn}
The following standard lemma realizes the Ap\'{e}ry set as a sort of ``fundamental domain'' for $H(P)$ modulo $c$.

\begin{lemma}\label{lem:apery-facts}
    The set $\Ap(H(P),c)$ contains a unique element in each residue class modulo $c$. Consequently, $|\Ap(H(P),c)| = c$ and every $b\in H(P)$ has a unique expression
    \begin{equation*}
        b=n+\ell c, \qquad n\in\Ap(H(P),c), \quad \ell\in\N.
    \end{equation*}
    Moreover, $\max\Ap(H(P),c)=F(P)+c.$
\end{lemma}
\begin{proof}
    Since $H(P)$ is cofinite, each residue class modulo $c$ contains an element of $H(P)$, and the least such element lies in $\Ap(H(P),c)$. Repeatedly subtracting $c$ gives the displayed expression, and minimality in each residue class gives uniqueness. Finally, $F(P)+c$ belongs to $H(P)$ while $F(P)+c-c=F(P)$ does not, so $F(P)+c\in\Ap(H(P),c)$. If $n>F(P)+c$, then $n-c>F(P)$ and hence $n-c\in H(P)$, so no such $n$ belongs to the Ap\'{e}ry set.
\end{proof}

For an integer $d\geq1$, set
\begin{equation*}
    q_d=\min\{q\geq1\mid qd\in H(P)\}, \qquad e_d=q_d d.
\end{equation*}
Thus $e_d$ is the smallest positive element of $H(P)$ divisible by $d$. In particular, $q_d=1$ if and only if $d\in H(P)$, and hence $q_d=1$ whenever $d>F(P)$.

Define the affine semigroup
\begin{equation*}
    \Gamma_d(P)=\{(a,b)\in\N^2\mid b\in H(P),\ b\leq ad\}.
\end{equation*}
The rational cone generated by $\Gamma_d(P)$ has extremal rays containing the semigroup points
\begin{equation*}
    \alpha=(1,0), \qquad \beta=(q_d,e_d).
\end{equation*}
These are the first semigroup points on the lower and upper extremal rays, respectively (see Figure \ref{fig: semigroups}).

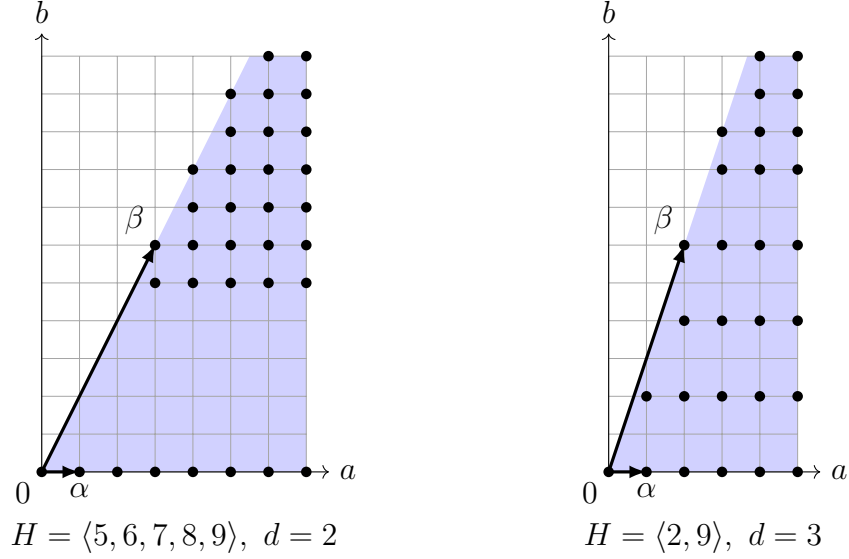
\begin{figure}[h]
\begin{tikzpicture}[
    scale=.5,
    semipt/.style={circle,fill,inner sep=1.4pt},
    cone/.style={fill=blue!18},
    ray/.style={very thick},
    gridline/.style={black!35, line width=.2pt}
]

%==================================================
% Example 1: H = {0,5,6,7,8,...}, d = 2
%==================================================
\begin{scope}

    % Window size
    \def\xmax{7}
    \def\ymax{11}

    % Shaded cone: 0 <= b <= 2a
    \fill[cone] (0,0) -- (\xmax,0) -- (\xmax,\ymax) -- ({\ymax/2},\ymax) -- cycle;

    % Faint integer lattice
    \draw[gridline] (0,0) grid (\xmax,\ymax);

    % Axes
    \draw[->] (0,0) -- (\xmax+0.6,0) node[right] {$a$};
    \draw[->] (0,0) -- (0,\ymax+0.6) node[above] {$b$};

    % Semigroup points: b=0 or b>=5, and b <= 2a
    \foreach \a in {0,...,\xmax}{
        \foreach \b in {0,...,\ymax}{
            \pgfmathtruncatemacro{\below}{\b <= 2*\a}
            \pgfmathtruncatemacro{\inH}{(\b == 0) || (\b >= 5)}
            \ifnum\below=1
                \ifnum\inH=1
                    \node[semipt] at (\a,\b) {};
                \fi
            \fi
        }
    }

    % Extremal rays drawn to first semigroup points on those rays
    \draw[ray, -latex] (0,0) -- (1,0);
    \draw[ray,-latex] (0,0) -- (3,6);

    % Labels
    \node[below left] at (0,0) {$0$};
    \node[below] at (1,0) {$\alpha$};
    \node[above left] at (3,6) {$\beta$};
    \node[below] at ({\xmax/2},-1.0) {$H=\langle 5,6,7,8,9\rangle,\ d=2$};

\end{scope}

%==================================================
% Example 2: H = <2,9>, d = 3
%==================================================
\begin{scope}[xshift=15cm]

    % Window size
    \def\xmax{5}
    \def\ymax{11}

    % Shaded cone: 0 <= b <= 3a
    \fill[cone] (0,0) -- (\xmax,0) -- (\xmax,\ymax) -- ({\ymax/3},\ymax) -- cycle;

    % Faint integer lattice
    \draw[gridline] (0,0) grid (\xmax,\ymax);

    % Axes
    \draw[->] (0,0) -- (\xmax+0.6,0) node[right] {$a$};
    \draw[->] (0,0) -- (0,\ymax+0.6) node[above] {$b$};

    % H = <2,9> = {0,2,4,6,8,9,10,...}
    % Semigroup points with b <= 3a
    \foreach \a in {0,...,\xmax}{
        \foreach \b in {0,...,\ymax}{
            \pgfmathtruncatemacro{\below}{\b <= 3*\a}
            \pgfmathtruncatemacro{\inH}{
                (\b == 0) || (\b == 2) || (\b == 4) || (\b == 6) ||(\b==8) ||(\b >= 9)
            }
            \ifnum\below=1
                \ifnum\inH=1
                    \node[semipt] at (\a,\b) {};
                \fi
            \fi
        }
    }

    % Extremal rays drawn to first semigroup points on those rays
    \draw[ray, -latex] (0,0) -- (1,0);
    \draw[ray, -latex] (0,0) -- (2,6);

    % Labels
    \node[below left] at (0,0) {$0$};
    \node[below] at (1,0) {$\alpha$};
    \node[above left] at (2,6) {$\beta$};
    \node[below] at ({\xmax/2},-1.0) {$H=\langle 2,9\rangle,\ d=3$};

\end{scope}
\end{tikzpicture}

\caption{The semigroups $\Gamma_d(P)$ for $H(P) = \langle 5,6,7,8,9\rangle, \, d=2$ (left) and $H(P) = \langle 2,9\rangle, \, d=3$ (right). The rational cone is highlighted and the extremal rays containing $\alpha$ and $\beta$ as their first $\Gamma_d$-points are labeled.}
\label{fig: semigroups}
\end{figure}

%\begin{center}
%\begin{tikzpicture}[scale=0.8]

% Axes
%\draw[->] (-0.5,0) -- (6,0) node[right] {$x$};
%\draw[->] (0,-0.5) -- (0,12) node[above] {$y$};

% Grid (optional)
%\draw[step=1cm,gray!30,very thin] (0,0) grid (6.5,12.5);

% Generators
%\coordinate (g1) at (1,0);
%\coordinate (g2) at (3,5);
%\coordinate (g3) at (3,6);
%\coordinate (g4) at (4,7);
%\coordinate (g5) at (4,8);
%\coordinate (g6) at (5,9);
%\coordinate (g7) at (5,10);

% Draw cone (spanned by extreme rays (2,0) and (0,2))
%\fill[blue!15] (0,0) -- (6,0) -- (6,12) -- cycle;

% Draw rays
%\draw[thick,blue] (0,0) -- (6,0);
%\draw[thick,blue] (0,0) -- (6,12);

% Draw generators as arrows
%\draw[->,red,thick] (0,0) -- (g1);
%\draw[->,red,thick] (0,0) -- (g2);
%\draw[->,red,thick] (0,0) -- (g3);

% Label generators
%\node[below] at (g1) {$(1,0)$};
%\node[left] at (g2) {$(3,5)$};
%\node[left] at (g3) {$(3,6)$};
%\node[left] at (g4) {$(4,7)$};
%\node[left] at (g5) {$(4,8)$};
%\node[left] at (g6) {$(5,9)$};
%\node[left] at (g7) {$(5,10)$};

% Plot semigroup elements (example points)
%\foreach \a in {0,...,2}
%{
%    \foreach \b in {0,...,2}
%    {
%        \foreach \c in {0,...,2}
%        {
%            \pgfmathtruncatemacro\x{2*\a + \b}
%            \pgfmathtruncatemacro\y{\b + 2*\c}
%            \ifnum\x<7
%                \ifnum\y<12
%                    \fill (\x,\y) circle (2pt);
%                \fi
%            \fi
%        }
 %   }
%}

%\node at (3,5.5) {$S \subset \mathbb{Z}^2$};
%\node at (4,2) {$\mathbb{R}_{\ge 0}S$};

%\end{tikzpicture}
%\end{center}

We denote by $k[\Gamma_d(P)]$ the \emph{semigroup algebra} of $\Gamma_d(P)$. As a subalgebra of the polynomial ring $k[s,t]$, we have $k[\Gamma_d(P)] = k[s^at^b : (a,b)\in\Gamma_d(P)]$.

\begin{proposition}\label{prop:apery labeling}
    There is a decomposition
    \begin{equation*}
        \Gamma_d(P) = \bigsqcup_{n\in \Ap(H(P), e_d)} ((\lceil n/d \rceil,n) + \N\alpha + \N\beta).
    \end{equation*}
    Consequently, $k[\Gamma_d(P)]$ is a free module over the polynomial subalgebra $k[\chi^\alpha, \chi^\beta]$ with basis
    \begin{equation*}
        \{ \chi^{(\lceil n/d \rceil, n)} \mid n\in \Ap(H(P),e_d) \}.
    \end{equation*}
    In particular, $\chi^{\alpha}, \chi^{\beta}$ is a regular sequence and the quotient $k[\Gamma_d(P)]/(\chi^\alpha, \chi^\beta)$ has a monomial basis indexed by $\Ap(H(P),e_d)$.
\end{proposition}
\begin{proof}
    Every $b\in H(P)$ has a unique expression
    \begin{equation*}
        b = n+\ell e_d, \qquad n\in \Ap(H(P), e_d), \quad \ell\in \N.
    \end{equation*}
    Since $e_d = d q_d$, one has $\lceil b/d \rceil = \ell q_d + \lceil n/d \rceil$. Thus, for $(a,b) \in \Gamma_d(P)$,
    \begin{equation*}
        (a,b) = (\lceil n/d\rceil,n) + \left( a- \lceil b/d\rceil \right)\alpha + \ell \beta.
    \end{equation*}
    This proves the decomposition of $k[\Gamma_d(P)]$ and the free-module decomposition. The final statements follow by reducing this free module modulo $\chi^\alpha$ and $\chi^\beta$.
\end{proof}

\section{Regularity, generators, and syzygies}\label{sec:artinian-reduction}

This section develops the structural results underlying the main theorems. We first compute the regularity and Cohen--Macaulayness of $R_d$. We then use a pole-order degeneration and an Artinian reduction by two extremal sections to bound the minimal algebra generators and the support of the Betti table.

\subsection{Regularity and Cohen--Macaulayness}
Several notions of regularity have been developed for nonstandard graded rings. We use \emph{weighted regularity} as introduced by Benson \cite{benson2004dickson}. Let $S$ be a $\Z$-graded polynomial ring with homogeneous maximal ideal $\mf{m}$.

\begin{defn}\label{def: weighted reg}
    For a graded $S$-module $M$, define
    \[
    a_i(M)=\sup\{j\in \Z\mid H^i_{\mathfrak m}(M)_j\neq 0\}.
    \]
    The \emph{(weighted) regularity} of $M$ is
    \[
    \reg(M)=\max_i\{a_i(M)+i\}.
    \]
\end{defn}

\begin{proposition}\label{prop: pointed section ring regularity}
    Assume $g\geq 1$. The section ring $R_d$ is Cohen--Macaulay of dimension $2$, and $\reg(R_d) = \lceil F(P)/d\rceil +1$.
\end{proposition}
\begin{proof}
    Let $L=\OO_C(dP)$. By the Serre--Grothendieck correspondence, we have an exact sequence
\begin{equation*}
    0 \to H^0_\mf{m}(R_d)_m \to (R_d)_m \to H^0(C,L^{\otimes m}) \to H^1_\mf{m}(R_d)_m \to 0
\end{equation*}
for every $m\in \Z$ and isomorphisms $H^i_\mf{m}(R_d)_m \cong H^{i-1}(C,L^{\otimes m})$ for $i\geq 2$; see, for example, \cite{bruns1998cohen}. The middle map is an isomorphism when $m \geq 0$ by definition of the section ring. For $m<0$ both terms vanish because $L^{\otimes m}$ has negative degree. Hence $H^0_\mf{m}(R_d) = H^1_\mf{m}(R_d)=0$. Since $\dim R_d = 2$, this proves $R_d$ is Cohen--Macaulay.

It remains to compute the top cohomology:
\begin{equation*}
    H^2_{\mf{m}}(R_d)_m \cong H^1(C,\OO_C(mdP)) \cong H^0(C,K_C-mdP)^\vee
\end{equation*}
by Serre duality.
As noted in Section \ref{sec-weighted-presentations}, the vanishing orders at $P$ of canonical sections are precisely $\ell-1$ for $\ell\in G(P)$. Thus $H^0(C,K_C-mdP) \neq 0$ if and only if
\begin{equation*}
    md \leq F(P) -1.
\end{equation*}
Therefore we have
\begin{equation*}
    a_2(R_d) = \max\{m\in \Z \mid md \leq F(P)-1\} = \left\lfloor\frac{F(P)-1}{d}\right\rfloor = \left\lceil \frac{F(P)}{d} \right\rceil -1.
\end{equation*}
Since $R_d$ is Cohen--Macaulay of dimension 2, its weighted regularity is $a_2(R_d)+2$, giving the stated formula.
\end{proof}

We obtain some immediate corollaries.

\begin{corollary}\label{cor:explicit-regularity}
    The regularity of $R_1$ satisfies
    \begin{equation*}
        g+1 \leq \reg R_1 \leq 2g.
    \end{equation*}
    Equality on the left holds if and only if $P$ is ordinary, while equality on the right holds if and only if $P$ is subcanonical, i.e. $K_C \sim (2g-2)P$. Moreover, when $g\geq 2$,
    \begin{equation*}
        \reg(R_g) = \begin{cases}
            2 & \text{if $P$ ordinary},\\
            3 & \text{otherwise}.
        \end{cases}
    \end{equation*}
\end{corollary}
\begin{proof}
    Since $G(P)$ consists of $g$ positive integers and is contained in $\{1,\dots, 2g-1\}$, one has $g\leq F(P) \leq 2g-1$. The equality $F(P)=g$ holds precisely when $G(P)=\{1,\dots, g\}$, i.e. when $P$ is ordinary. At the other extreme, $F(P)=2g-1$ if and only if some canonical section vanishes to order $2g-2$ at $P$, equivalently $K_C \sim (2g-2)P$. The remaining assertions now follow from Proposition \ref{prop: pointed section ring regularity}.
\end{proof}

\begin{corollary}\label{cor:gorenstein}
    Assume $g\geq 2$. Then $R_d$ is Gorenstein if and only if $P$ is subcanonical and $d\mid (F(P)-1)$.
\end{corollary}
\begin{proof}
    Since $R_d$ is Cohen--Macaulay, \cite[(5.1.9)]{goto1978graded} shows that $R_d$ is Gorenstein if and only if $K_C \sim a\OO_C(dP)$ for some $a\in \Z$.
    We see from our previous argument that $P$ must be subcanonical and $d$ must divide $2g-2$. Conversely, since $2g-2=F(P)-1$ for a subcanonical point $P$, we are done.
\end{proof}

\begin{remark}
    The regularity computation alone gives the coarse support bound of Symonds \cite{symonds2011castelnuovo}; the local cohomology computation itself also yields a column-by-column support bound via the \emph{Koszul regularity}~\cite{brown2024positivity}. We compare with Theorem \ref{thm: intro-bettiSupport} in Remark \ref{rem:symonds-comparison} below.
\end{remark}

We now turn to the finer columnwise information given by the pole-order degeneration.

\subsection{The semigroup degeneration}
With $\Gamma_d(P)$ as in Section \ref{sec-weighted-presentations}, define a multiplicative pole-order filtration $F_\bullet$ on $R_d$ by, for $a,b\geq 0$,
\begin{equation*}
    F_b(R_d)_a= \{0 \} \cup \{0\neq f \in H^0(C,\OO_C(adP)) \mid -\ord_P(f) \leq b \},
\end{equation*}
with $F_{-1}(R_d)_a = 0$. 

\begin{lemma}\label{lem:pole-order-semigroup}
    There is a bigraded algebra isomorphism
    \begin{equation*}
        \gr_F(R_d) \cong k[\Gamma_d(P)].
    \end{equation*}
    In particular, the bigraded Hilbert function of $\gr_F(R_d)$ is $1$ on $\Gamma_d(P)$ and $0$ elsewhere.
\end{lemma}
\begin{proof}
    The $(a,b)$-graded piece of $\gr_F(R_d)$ is $F_b(R_d)_a/F_{b-1}(R_d)_a$. It is one-dimensional precisely when there is a rational function regular away from $P$ with pole order exactly $b$ and $b\leq ad$, equivalently when $(a,b)\in \Gamma_d(P)$, and zero-dimensional otherwise.

    Let $\chi^{(a,b)}$ be the basis element of $k[\Gamma_d(P)]$ corresponding to $(a,b)$. Choose a uniformizer $t$ at $P$. If $[f]\neq 0$ in the $(a,b)$-graded piece, write
    \begin{equation*}
        f = c\, t^{-b} + \text{higher-order terms}, \qquad c\in k^\times.
    \end{equation*}
    The map sending $[f]$ to $c\chi^{(a,b)}$ is well defined and gives an isomorphism in each bidegree. Since leading coefficients multiply under products, this gives the required bigraded algebra isomorphism.
\end{proof}

In fact, the Rees algebra of $F$ gives a flat degeneration $R_d \rightsquigarrow k[\Gamma_d(P)]$. This construction is a coarser analogue of Ap\'{e}ry specialization for numerical semigroup algebras; cf. \cite{braun2025minimal}. We will sometimes refer to the grading on $R_d$ coming from its decomposition $R_d = \bigoplus_{m\geq 0}H^0(C,\OO_C(mdP))$ as the \emph{section grading} and the filtration degree coming from $F_\bullet(R_d)_a$ as the \emph{pole grading} on $k[\Gamma_d(P)]$.

\begin{remark}\label{rem:pole-order-degeneration}
    Lemma \ref{lem:pole-order-semigroup} above can similarly be obtained using the framework of \cites{kaveh2012newton,kaveh2019khovanskii}.
    The pole-order filtration is induced by the valuation $\nu_P=-\ord_P$ which has one-dimensional leaves. Recording section degree together with pole order gives the value semigroup $\Gamma_d(P)$, and the one-dimensional leaves of $\nu_P$ yield the semigroup degeneration above. They further compute an associated Newton--Okounkov body which in our case only records the degree of $\OO_C(dP)$ and is too coarse for syzygetic information.
\end{remark}

Combining this degeneration with Proposition \ref{prop:apery labeling}, we see that $\gr_F(R_d)$ is a finite free module over $k[\chi^\alpha,\chi^\beta]$. The two extremal semigroup elements $\alpha$ and $\beta$ now determine the regular sequence used in the Artinian reduction.

\subsection{Artinian reduction}
In every degree $m$, the pole orders of sections of $\OO_C(mdP)$ lie between $0$ and $md$. The semigroup points $\alpha$ and $\beta$ lift to sections $x$ and $z$: denote by $x$ the image of 1 under the natural inclusion $k = H^0(C,\OO_C)\subset H^0(C,\OO_C(dP))$, and let $z\in(R_d)_{q_d}$ correspond to a rational function in $H^0(C,\OO_C(e_dP))$ with pole order exactly $e_d$ at $P$.

\begin{proposition}\label{prop: artinian-reduction}
    One can always choose a minimal weighted presentation
    \begin{equation*}
    S_d = k[\tilde{x},\tilde{z},y_1,\dots, y_s] \surj R_d, \qquad \deg \tilde{x} = 1, \quad \deg \tilde{z}=q_d, \quad \deg y_a=v_a,
\end{equation*}
sending $\tilde{x}$ to $x$ and $\tilde{z}$ to $z$, where $v_1\leq \dots \leq v_s$. Furthermore, setting $T = k[y_1,\dots, y_s]$ and $A_d = R_d/(x,z)$,
\begin{enumerate}
    \item The elements $x,z$ form an $R_d$-regular sequence, $A_d$ has finite length, and
    \begin{equation*}
        \beta_{i,j}^{S_d}(R_d) = \beta_{i,j}^{T}(A_d)
    \end{equation*}
    for all $i,j$.
    \item If $\mf n=(R_d)_+$ and $\mf a=(A_d)_+$, then
    \begin{equation*}
        \frac{\mf a}{\mf a^2}
        \cong
        \frac{\mf n}{\mf n^2+(x,z)}.
    \end{equation*}
    In particular, the images of $y_1,\dots,y_s$ form a homogeneous $k$-basis of $\mf a/\mf a^2$.
\end{enumerate}
\end{proposition}
\begin{proof}
    We first show that we can always include $x$ and $z$ in a minimal generating set for $R_d$. Let $\mf{n} = (R_d)_+$. By graded Nakayama's lemma, a collection of homogeneous elements can be included in a minimal homogeneous generating set for $R_d$ if and only if their images in $\mf{n}/\mf{n}^2$ are linearly independent. It therefore suffices to check this for $x$ and $z$.

    Set $q=q_d$. Since $(\mf{n}^2)_1=0$, the element $x$ has nonzero image in $\mf{n}/\mf{n}^2$. If $q=1$, then $x,z\in (R_d)_1$, and their images are linearly independent because they have distinct pole orders $0$ and $d$.
    Suppose now that $q>1$. Every element of $(\mf{n}^2)_q$ is a $k$-linear combination of products $g_1g_2$ where $g_i$ has positive degree $a_i<q$ and $a_1+a_2=q$. By the minimality of $q$, one has $a_id \notin H(P)$, so every nonzero element of $(R_d)_{a_i}$ has pole order strictly less than $a_id$. Every element of $(\mf{n}^2)_q$ therefore has pole order strictly less than $qd=e_d$, whereas $z$ has pole order $e_d$. Thus the image of $z$ is nonzero. It lies in a different graded component from the image of $x$, so the two images are linearly independent.

    \textit{(a)} The section $x$ vanishes exactly at $P$ with order $d$, while $z$ has pole order exactly $e_d$ at $P$ and hence does not vanish there as a section of $\OO_C(e_dP)$. Thus $x$ and $z$ have no common zero on $C\cong \Proj R_d$. It follows that $R_d/(x,z)$ has finite length. Since $R_d$ is Cohen--Macaulay of dimension $2$ by Proposition \ref{prop: pointed section ring regularity}, the pair $x,z$ is a homogeneous system of parameters and hence a regular sequence.

    The Betti number comparison follows by tensoring a minimal $S_d$-resolution $G_\bullet$ of $R_d$ with $T=S_d/(\tilde{x},\tilde{z})$. The groups $\Tor_p^{S_d}(R_d,T)$ are the Koszul homology groups of $x,z$ with coefficients in $R_d$, so they vanish for $p>0$. Therefore $G_\bullet \otimes_{S_d} T$ resolves $A_d$ over $T$. It remains minimal because its differentials have entries in $(y_1,\dots, y_s)$, proving the equality of Betti numbers.

    \textit{(b)} Since $(x,z)\subseteq\mf{n}$, one has
    \begin{equation*}
        \mf{a}=\frac{\mf{n}}{(x,z)} \qquad\text{and}\qquad \mf{a}^2=\frac{\mf{n}^2+(x,z)}{(x,z)},
    \end{equation*}
    which gives the displayed isomorphism. The classes of the images of $\tilde{x},\tilde{z},y_1,\dots,y_s$ form a homogeneous basis of $\mf{n}/\mf{n}^2$ because the presentation $S_d\surj R_d$ is minimal. Quotienting by the classes of $x$ and $z$ proves the final assertion.
\end{proof}

\subsection{Bounds on minimal algebra generators}
The pole-order filtration on $R_d$ induces a filtration on the quotient $A_d=R_d/(x,z)$. Its associated graded algebra is identified in Proposition \ref{prop:apery-degeneration} with the corresponding quotient of the semigroup algebra.

\begin{proposition}\label{prop:apery-degeneration}
    Let $A_d$ carry the filtration $F_\bullet$ induced by the pole-order filtration on $R_d$. Then
    \begin{equation*}
        \gr_F(A_d) \cong \frac{k[\Gamma_d(P)]}{(\chi^{\alpha},\chi^{\beta})}.
    \end{equation*}
    A monomial basis is indexed by $\{(\lceil n/d\rceil,n) \mid n\in \Ap(H(P),e_d)\}$. This basis has exactly $e_d$ elements and
    \begin{equation*}
        H_{A_d}(t) = \sum_{n\in \Ap(H(P),e_d)} t^{\lceil n/d\rceil}.
    \end{equation*}
    Consequently, $\operatorname{length}(A_d) = e_d$ and the top nonzero degree of $A_d$ is $q_d + \lceil F(P)/d \rceil$.
\end{proposition}
\begin{proof}
    By Lemma \ref{lem:pole-order-semigroup}, the initial forms of $x,z$ are $\chi^{(1,0)}$ and $\chi^{(q_d,e_d)}$. Hence there is a natural section-graded surjection from $Q_d \coloneqq k[\Gamma_d(P)]/(\chi^{\alpha},\chi^{\beta})$ to $\gr_F(A_d)$. The source $Q_d$ has the displayed basis by Proposition \ref{prop:apery labeling}. It remains to prove that the surjection is an isomorphism.
    Since $x,z$ form an $R_d$-regular sequence of degrees $1,q_d$, and passing to the associated graded preserves the dimension of each section-graded piece, one has
    \begin{equation*}
        H_{\gr_F(A_d)}(t)=H_{A_d}(t)=(1-t)(1-t^{q_d})H_{R_d}(t).
    \end{equation*}
    Using the fact that $\chi^\alpha,\chi^\beta$ is a regular sequence on $k[\Gamma_d(P)]$ and that $R_d$ and $k[\Gamma_d(P)]$ have the same Hilbert series, it follows that
    \begin{equation*}
        H_{Q_d}(t) = (1-t)(1-t^{q_d})H_{k[\Gamma_d(P)]}(t) = (1-t)(1-t^{q_d})H_{R_d}(t) = H_{\gr_F(A_d)}(t).
    \end{equation*}
    Thus the surjection is an isomorphism. Since the basis element indexed by $n$ has section degree $\lceil n/d\rceil$, the formula for $H_{A_d}(t)$ follows. By Lemma \ref{lem:apery-facts}, $|\Ap(H(P),e_d)|=e_d$, so evaluating at $t=1$ gives $\operatorname{length}(A_d)=e_d$. The same lemma gives $\max\Ap(H(P),e_d)=F(P)+e_d$, so the largest section degree is
    \begin{equation*}
        \left \lceil \frac{F(P)+e_d}{d} \right \rceil = q_d + \left \lceil \frac{F(P)}{d} \right\rceil. \qedhere
    \end{equation*}
\end{proof}

The two degenerations mentioned in Lemma~\ref{lem:pole-order-semigroup} and Proposition~\ref{prop:apery-degeneration} are related by the diagram
\begin{center}
\usetikzlibrary{decorations.pathmorphing}
    \begin{tikzcd}
        R_d \ar[d, "\bmod {(x,z)}"'] \ar[r, squiggly] & \gr_F(R_d)\cong k[\Gamma_d(P)] \ar[d,"\bmod {(\chi^\alpha,\chi^\beta)}"] \\
        A_d \ar[r, squiggly] & \gr_F(A_d)\cong  \frac{k[\Gamma_d(P)]}{(\chi^{\alpha},\chi^{\beta})}.
    \end{tikzcd}
\end{center}
The horizontal arrows are flat degenerations, while the vertical maps are quotients. Proposition~\ref{prop:apery-degeneration} gathers the information needed to bound both the number and the degrees of the minimal algebra generators.

\begin{corollary}\label{cor:minimal-generator-bounds}
    Let $S_d=k[x_0,\dots, x_r] \surj R_d$ be a minimal weighted presentation, with $\deg x_i = w_i$. Then:
    \begin{enumerate}
        \item The presentation has at most $e_d+1$ variables; equivalently, $r\leq e_d$.
        \item Every weight satisfies
        \begin{equation*}
            w_i \leq q_d + \left\lceil \frac{F(P)}{d} \right\rceil.
        \end{equation*}
        \item If $d>F(P)$, then $q_d=1$ and
        \begin{equation*}
            H_{A_d}(t)=1+(d-g-1)t+gt^2.
        \end{equation*}
        In this case, every minimal presentation has exactly $d+1-g$ variables of weight $1$ and no variables of weight greater than $2$.
    \end{enumerate}
\end{corollary}
\begin{proof}
    The multiset of weights in a minimal presentation is the multiset of degrees of a homogeneous basis of $(R_d)_+/(R_d)_+^2$, so it is independent of the chosen presentation. By Proposition \ref{prop: artinian-reduction}, we may therefore work with a minimal weighted presentation
    \begin{equation*}
        k[x,z,y_1,\dots,y_s]\surj R_d,
        \qquad \deg(x)=1,\quad \deg(z)=q_d,\quad \deg(y_i)=v_i,
    \end{equation*}
    in which the images of $y_1,\dots,y_s$ form a basis of $\mf a/\mf a^2$ for $\mf a=(A_d)_+$. Proposition \ref{prop:apery-degeneration} gives $\dim_k\mf a=e_d-1$, and hence
    \begin{equation*}
        s=\dim_k(\mf a/\mf a^2)\leq \dim_k\mf a=e_d-1.
    \end{equation*}
    Thus the presentation has $s+2\leq e_d+1$ variables, proving (a).

    The class of each $y_i$ is nonzero in $\mf a/\mf a^2$, so $(A_d)_{v_i}\neq0$. Since the top nonzero degree of $A_d$ is $q_d+\lceil F(P)/d\rceil$, one has
    \begin{equation*}
        v_i\leq q_d+\left\lceil\frac{F(P)}d\right\rceil.
    \end{equation*}
    The variables $x$ and $z$ have weights $1$ and $q_d$, which satisfy the same bound. This proves (b).

    Since $d>F(P)$, one has $d\in H(P)$ and hence $q_d=1$. The Ap\'{e}ry set with respect to $d$ consists of $0$, the positive nongaps less than $d$, and the elements $d+\ell$ for $\ell \in G(P)$. There are $d-g-1$ positive nongaps less than $d$ and $g$ gaps, so Proposition \ref{prop:apery-degeneration} gives the displayed Hilbert series. The degree-one part of $A_d$ consists entirely of minimal generators, which together with $x,z$ account for $d+1-g$ variables of weight $1$. By (b), all remaining variables have weight $2$, proving (c).
\end{proof}

\subsection{Betti support} We are now ready to prove our bounds on the support of the Betti table of $R_d$.

\begin{proof}[Proof of Theorem \ref{thm: intro-bettiSupport}]
By Proposition \ref{prop: artinian-reduction}(a), $\beta_{i,j}^{S_d}(R_d) = \beta_{i,j}^T(A_d)$, so it suffices to bound the support of $\Tor_i^T(A_d,k)$. This Tor group is computed by $A_d\otimes_T K_\bullet(y_1,\dots,y_s)$, whose $i$-th term is
\begin{equation*}
    \bigoplus_{1\leq a_1 < \cdots < a_i \leq s} A_d(-v_{a_1} - \cdots - v_{a_i}).
\end{equation*}
In particular, $\Tor_i^T(A_d,k)=0$ for $i>s$ and
\begin{equation*}
    \Tor_0^T(A_d,k) = A_d/(y_1,\dots, y_s)A_d \cong k.
\end{equation*}

By Proposition \ref{prop:apery-degeneration}, the algebra $A_d$ is supported only in degrees
\begin{equation*}
    0 \leq m \leq q_d + \lceil F(P)/d \rceil.
\end{equation*}
Therefore the $i$-th term of the Koszul complex is supported only in internal degrees $j$ for which
\begin{equation*}
    v_{a_1} + \cdots + v_{a_i} \leq j \leq v_{a_1} + \cdots + v_{a_i} +q_d + \lceil F(P)/d\rceil
\end{equation*}
for some subset $\{a_1,\dots, a_i\} \subseteq [s]$. The smallest possible sum of $i$ distinct entries of $\bf{v}$ is $\mathbf{v}_i$, and the largest possible sum is $\mathbf{v}^i$. Hence the $i$-th Koszul term (and therefore its homology) is supported in degrees
\begin{equation*}
    \mathbf{v}_i \leq j \leq \mathbf{v}^i + q_d + \lceil F(P)/d\rceil.
\end{equation*}
To improve the lower bound to $\mathbf{v}_i+v_1$, Proposition \ref{prop: artinian-reduction}(b) shows that the images of $y_1,\dots, y_s$ are linearly independent in $A_d$. If $j<\mathbf{v}_i + v_1$, then every element of internal degree $j$ in the $i$-th Koszul term has coefficients in $(A_d)_0 = k$: every $i$-fold wedge has degree at least $\mathbf{v}_i$, while every positive degree in $A_d$ is at least $v_1$. No nonzero such element is a cycle, as the linear independence of $y_1,\dots, y_s$ implies that the Koszul differential out of this term is injective.

Finally, if $2v_1 > q_d + \lceil F(P)/d \rceil$, or if $P$ is ordinary and either $d\mid g$ or $d\mid(g+1)$, then Proposition \ref{prop:square-zero-criteria} gives $\mf a^2=0$. Corollary \ref{cor:square-zero-sharpness} shows that both endpoints of the displayed interval are attained in every homological degree, proving the sharpness assertion.
\end{proof}

\begin{remark}\label{rem:symonds-comparison}
    Applied to $A_d$ over $T$, the regularity bound of Symonds \cite{symonds2011castelnuovo} gives $\beta_{i,j} = 0$ for
    \begin{equation*}
        j > q_d + \left \lceil \frac{F(P)}{d} \right\rceil + \sum_{a=1}^s v_a - (s-i).
    \end{equation*}
    The upper bound in Theorem \ref{thm: intro-bettiSupport} improves this by exactly $\sum_{a=1}^{s-i}(v_a-1)$. On the other hand, the Koszul regularity bound of Brown--Erman \cite{brown2024positivity} gives $\beta_{i,j} = 0$ for
    \[
    j\geq \ww^{i+2}+\left\lceil\frac{F(P)}{d}\right\rceil,
    \]
    which we improve by ${\bf w}^{i+2}-({\bf v}^i + q_d + 1)$.
\end{remark}

\begin{corollary}\label{cor:tight-in-last-column}
    The upper support bound of Theorem \ref{thm: intro-bettiSupport} is attained in the final column of the Betti table. That is, $\beta_{s,\mathbf{v}^s + q_d + \lceil F(P)/d\rceil}^{S_d}(R_d) \neq 0$.
\end{corollary}
\begin{proof}
Let $D = q_d + \lceil F(P)/d \rceil$. Choose a nonzero $a\in (A_d)_D$ and view it in the last term of the Koszul complex $A_d \otimes_T K_s(y_1,\dots, y_s) \cong A_d(-\mathbf{v}^s)$. For every $i$, the element $y_ia \in (A_d)_{D+v_i}$ is zero because $D$ is the top nonzero degree of $A_d$ by Proposition \ref{prop:apery-degeneration}. Thus $a$ is a cycle. There is no incoming differential to the last Koszul term, so it gives a nonzero class in $\Tor^T_s(A_d,k)_{\mathbf{v}^s+D}$. This in turn gives the nonvanishing over $S_d$ by Proposition \ref{prop: artinian-reduction}(a).
\end{proof}

The preceding results show that the variable weights and the degrees of the syzygies are closely related: both are controlled by the quantity $q_d + \lceil F(P)/d\rceil$ coming from the extremal ray of the semigroup $\Gamma_d(P)$. To obtain a weighted $N_p$-type result, we need to bound the degrees of the syzygies entirely in terms of the degrees of the variables.
When $d>F(P)$, the Hilbert function of $A_d$ allows us to do exactly that.

\begin{corollary}\label{cor:weighted Np}
    Let $C$ be a smooth curve of genus $g\geq 1$ and $P\in C$.
    If $d>F(P)$ and $g-1-\binom{d-g}{2}\geq 0$, then $R_d$ satisfies the weighted $N_{g-1-\binom{d-g}{2}}$ condition.
\end{corollary}

\begin{proof}
Set $r = d-g$ and $p=g-1-\binom{r}{2}$; since $F(P)\geq g$ for any $P$, $r\geq 1$. Since $d>F(P)$, Corollary \ref{cor:minimal-generator-bounds} gives
\[
    H_{A_d}(t)=1+(r-1)t+gt^2.
\]
Let $\mf a=(A_d)_+$, and let $b$ be the number of weight-$2$ variables in a minimal presentation of $R_d$. By Proposition \ref{prop: artinian-reduction}(b), $b=\dim_k(\mf a/\mf a^2)_2$.
Multiplication in $A_d$ gives a map
\[
    \mu\colon \Sym^2\bigl((A_d)_1\bigr)\longrightarrow (A_d)_2,
\]
and hence
\[
\begin{aligned}
    b
      &=\dim_k(A_d)_2-\rank(\mu)\\
      &\geq g-\dim_k\Sym^2\bigl((A_d)_1\bigr)\\
      &=g-\binom{r}{2}
       =p+1.
\end{aligned}
\]

Corollary \ref{cor:minimal-generator-bounds} also shows that the full weight vector and the weight vector remaining after the two distinguished weight-$1$ variables are removed have the forms
\[
    \ww=(1^{r+1},2^b)
    \qquad\text{and}\qquad
    \vv=(1^{r-1},2^b).
\]
For every $1\leq i\leq p$, the inequality $b\geq p+1$ gives
\[
    \vv^i=2i
    \qquad\text{and}\qquad
    \ww^{i+1}=2(i+1)=\vv^i+2.
\]
Because $d>F(P)$, one has $q_d=1$ and $\lceil F(P)/d\rceil=1$. Theorem \ref{thm: intro-bettiSupport} now yields
\[
    \beta^{S_d}_{i,j}(R_d)=0
    \qquad\text{for }j>\vv^i+2=\ww^{i+1}
\]
for every $1\leq i\leq p$. Finally, $R_d$ is Cohen--Macaulay by Proposition \ref{prop: pointed section ring regularity}, so it is normally generated. This is precisely weighted $N_p$.
\end{proof}

Corollary \ref{cor:weighted Np} says that in the strictly ample regime, the section ring satisfies \emph{higher} weighted $N_p$ conditions as the degree of the line bundle \emph{decreases} towards $F(P)$, inverting the picture of the classical theory. Note that when $d>F(P)$, we have $\ww = (1^{r+1},2^b)$ for some $b$, and this corollary says that in particular $R_d$ satisfies weighted $N_{b-1}$.

\begin{example}\label{ex:weighted-Np-genus-4}
Continuing Example \ref{ex:betti tables of genus 4 space curve}, Corollary~\ref{cor:weighted Np} says that $R_5$ satisfies weighted $N_3$ and that $R_6$ satisfies weighted $N_2$. In this case these conclusions are optimal, in that $R_5$ does not satisfy weighted $N_4$ and $R_6$ does not satisfy weighted $N_3$. This can be seen in Figure \ref{fig:betti-genus-4-CI} since $\beta_{4,10}^{S_5}(R_5) \neq 0$, while $\ww^5 = 9$; similarly, $\beta_{3,8}^{S_6}(R_6) \neq 0$, while $\ww^4=7$.
\end{example}

\section{Square-zero reductions}\label{sec:square-zeroreduction}
The degenerations in Section \ref{sec:artinian-reduction} do not in general determine the Betti numbers.
In this section, we give sufficient conditions to guarantee that they do.
We keep the notation of the previous section:
\begin{equation*}
    S_d=k[x,z,y_1,\dots, y_s] \surj R_d, \quad T = k[y_1,\dots, y_s], \quad A_d = R_d/(x,z), \quad\mf{a} = (A_d)_+
\end{equation*}
where
\begin{equation*}
    \deg x = 1, \qquad \deg z = q_d, \qquad \deg y_a = v_a, \qquad v_1\leq \cdots \leq v_s.
\end{equation*}

Since $A_d$ has finite length, $\mf{a}^m = 0$ for some $m$. When that $m$ is equal to $2$, in other words the length of $A_d$ is the smallest possible, $A_d$ has a particularly simple description and thus the Betti numbers of $R_d$ can be described exactly.

\begin{proposition}\label{prop:square-zero-Ad}
    If the maximal ideal $\mf{a}$ of $A_d$ satisfies $\mf{a}^2=0$, then $A_d \cong T/(y_1,\dots, y_s)^2$ and thus for all $i,j$,
    \begin{equation*}
        \beta^{S_d}_{i,j}(R_d) = \beta^T_{i,j}(T/(y_1,\dots, y_s)^2).
    \end{equation*}

\end{proposition}
\begin{proof}
    Since $T \surj A_d$ is obtained from a minimal weighted presentation of $R_d$ after quotienting by $x,z$, the images of $y_1,\dots, y_s$ form a $k$-basis of $\mf{a}/\mf{a}^2$. If $\mf{a}^2=0$ then they form a $k$-basis for all of $\mf{a}$. Thus $A_d$ has basis $1,\overline{y}_1,\dots, \overline{y}_s$ and all products $\overline{y}_a\overline{y}_b$ vanish. Therefore
    \begin{equation*}
        A_d\cong T/(y_1,\dots, y_s)^2,
    \end{equation*}
    and the Betti equality follows directly from Proposition \ref{prop: artinian-reduction}(a).
\end{proof}

We can also give several criteria for when this phenomenon occurs.
\begin{proposition}\label{prop:square-zero-criteria}
    The square-zero condition can be detected in the following cases:
    \begin{enumerate}
        \item $\mf a^2 = 0$ if and only if $s=e_d-1$, that is, $\vv$ has maximal possible length.
        \item If $2v_1>q_d + \lceil F(P)/d\rceil$, then $\mf a^2 = 0$.
        \item If $P$ is ordinary, then $\mf a^2 = 0$ if and only if $d~|~g$ or $d~|~(g+1)$.
    \end{enumerate}
\end{proposition}
\begin{proof}
    \begin{enumerate}[leftmargin=*]
    \item Proposition \ref{prop:apery-degeneration} gives $\dim_k \mf a = e_d -1$, while Proposition \ref{prop: artinian-reduction}(b) gives $\dim_k(\mf a/\mf a^2) = s$. Hence $\dim_k \mf a^2 = e_d - 1 -s$, proving the equivalence.
    \item The smallest degree of any element of $\mf{a}^2$ is at least $2v_1$, so if $2v_1 > q_d + \lceil F(P)/d\rceil$, then every element of $\mf{a}^2$ has degree strictly larger than $q_d + \lceil F(P)/d\rceil$ and is thus zero by Proposition \ref{prop:apery-degeneration}.
    \item Now, suppose that $P$ is ordinary. Then
    \begin{equation*}
        H(P) = \{0,g+1,g+2,\dots\}, \quad q_d=\left\lceil \frac{g+1}{d} \right\rceil, \quad e_d = dq_d,
    \end{equation*}
    and a direct calculation gives
    \begin{equation*}
        \Ap(H(P),e_d) = \{0\} \cup \{g+1,\dots, e_d-1\} \cup \{e_d+1,\dots, e_d+g\}.
    \end{equation*}
    If $d \mid (g+1)$, write $g+1 = d\cdot b$; since $g+1$ is the smallest positive element of $H(P)$, $b=q_d$ and hence $e_d = g+1$. In this case,
    \begin{equation*}
        \Ap(H(P),e_d) =\{0,e_d+1,\dots,e_d+g\}.
    \end{equation*}
    The smallest and largest positive basis elements of $A_d$ are therefore indexed
    by $e_d+1$ and $e_d+g$ and have degree
    \begin{equation*}
        \left\lceil\frac{e_d+1}{d}\right\rceil=q_d+1 \qquad \text{and} \qquad \left\lceil\frac{e_d+g}{d}\right\rceil=q_d+\left\lceil\frac{g}{d}\right\rceil=2q_d.
    \end{equation*}
    Consequently, every product of two positive-degree elements has degree at least $2q_d+2$, above the top degree of $A_d$, so $\mf a^2=0$. A similar calculation shows that $d\mid g$ implies $\mf a^2 = 0$.

    Conversely, suppose that $d$ divides neither $g$ nor $g+1$. Write $g=ad+r$ with $1\leq r \leq d-2$. Then $q_d=a+1$ and $e_d=(a+1)d$, and inside $Q_d = k[\Gamma_d(P)]/(\chi^\alpha,\chi^\beta)$ we have the nonzero product
    \begin{equation*}
        \chi^{(q_d,g+1)}\chi^{(q_d,e_d-1)} = \chi^{(2q_d,e_d+g)}.
    \end{equation*}
    Since $Q_d \cong \gr_F(A_d)$ by Proposition \ref{prop:apery-degeneration}, we must have $\mf{a}^2\neq 0$.
    \end{enumerate}
\end{proof}

\begin{corollary}\label{cor:square-zero-sharpness}
    Suppose that $\mf a^2=0$. Then, for every $1\leq i\leq s$,
    \begin{equation*}
        \beta^{S_d}_{i,\vv_i+v_1}(R_d)\neq 0
        \qquad\text{and}\qquad
        \beta^{S_d}_{i,\vv^i+q_d+\lceil F(P)/d\rceil}(R_d)\neq 0.
    \end{equation*}
    In particular, both bounds in Theorem \ref{thm: intro-bettiSupport} are sharp in every homological degree.
\end{corollary}
\begin{proof}
    There is nothing to prove if $s=0$, so assume $s\geq 1$.
    By Proposition \ref{prop:square-zero-Ad}, it suffices to compute the Koszul homology of
    \begin{equation*}
        A_d\cong T/(y_1,\dots,y_s)^2.
    \end{equation*}
    Let $e_1,\dots,e_s$ be the standard basis of the Koszul complex on $y_1,\dots,y_s$, with $\deg(e_a)=v_a$. Since $\mf a^2=0$, the elements
    \begin{equation*}
        y_1e_1\wedge\cdots\wedge e_i
        \qquad\text{and}\qquad
        y_se_{s-i+1}\wedge\cdots\wedge e_s
    \end{equation*}
    are cycles. They are not boundaries: after writing each coefficient as a scalar plus an element of $\mf a$, the $\mf a$-part has zero differential, and the $e_I$-coefficient of the remaining boundary lies in the span of the $y_a$ with $a\notin I$. In each displayed cycle, by contrast, the coefficient is one of the $y_a$ with $a\in I$. Thus they give nonzero Tor classes in degrees $\vv_i+v_1$ and $\vv^i+v_s$, respectively. Finally, Proposition \ref{prop:apery-degeneration} identifies $q_d+\lceil F(P)/d\rceil$ as the top nonzero degree of $A_d$, while $A_d=k\oplus\langle y_1,\dots,y_s\rangle_k$ in the square-zero case. Hence this top degree is $v_s$, giving the asserted upper endpoint.
\end{proof}

\begin{corollary}\label{cor: pure-resolution}
    In the situation of Proposition \ref{prop:square-zero-Ad}, suppose moreover that $v_1=\cdots = v_s=v$. Then the minimal free resolution of $R_d$ over $S_d$ is pure with Betti numbers
    \begin{equation*}
        \beta_{i,(i+1)v}^{S_d}(R_d) = i \binom{s+1}{i+1}, \quad 1\leq i\leq s.
    \end{equation*}
    For $P$ an ordinary point, the minimal free resolution of $R_{g+1}$ is always pure.
\end{corollary}
\begin{proof}
    By Proposition \ref{prop:square-zero-Ad}, the Betti numbers of $R_d$ over $S_d$ match $T/(y_1,\dots,y_s)^2$. When all variables $y_i$ have degree $v=v_1=\dots=v_s$, the square of the homogeneous maximal ideal of $T$ is resolved by the Eagon--Northcott complex on a $2\times (s+1)$ matrix whose nonzero entries are variables. The $i$-th syzygy module is generated in degree $(i+1)v$ of rank $i\binom{s+1}{i+1}$.

    If $P$ is ordinary,
    \begin{equation*}
        H(P) = \{0,g+1,g+2,\dots\}, \quad F(P)=g.
    \end{equation*}
    For $d=g+1$, Proposition \ref{prop:square-zero-criteria}(c) gives $\mf a^2 = 0$, while Corollary \ref{cor:minimal-generator-bounds}(c) gives
    \begin{equation*}
        H_{A_{g+1}}(t) = 1 + gt^2.
    \end{equation*}
    Thus $s=g$ and $v_1=\cdots = v_s =2$.
\end{proof}

\newpage
\begin{appendix}
  \section{Betti tables of weighted embeddings with $g = 3$}
  Here we give examples of the family of Betti tables of section rings $R_d = R(C,\OO(dP)$ for points exhibiting all possible Weierstrass semigroups for curves of genus 3. For each, we display the weight vector of the ambient $\PP({\bf w})$ with the two weights corresponding to the extremal rays of $\Gamma_d$ underlined. We also display a vector recording, for each section in the minimal generating generating set of $R_d$, its pole order at the point $P$. The region shading in red shows the entries of the Betti table whose vanishing is forced by Theorem \ref{thm: intro-bettiSupport}.
  \begin{figure}[H]
\begin{adjustwidth}{-0.5in}{-0.3in}
\centering

$H(P) = \langle 4,5,6,7\rangle$

%%%%%%%%%%%%%%%%%%%%%
% kk = ZZ/3
% H  = {4, 5, 6, 7} -- genus 3
% R  = kk[x_0..x_4]/(x_2^2-x_0*x_3-x_1*x_3,x_0*x_2+x_3^2-x_2*x_4,x_0^2+x_0*x_1+x_2*x_3-x_0*x_4-x_1*x_4,x_1^2*x_2+x_0*x_1*x_3+x_2*x_3^2+x_0*x_3*x_4-x_1*x_3*x_4,x_1^3-x_1^2*x_3-x_1*x_2*x_3+x_0*x_3^2-x_3^3+x_0*x_1*x_4+x_1^2*x_4-x_2*x_3*x_4-x_3^2*x_4-x_1*x_4^2+x_2*x_4^2,x_0*x_1^2+x_1^2*x_3+x_1*x_2*x_3+x_3^3-x_0*x_1*x_4-x_1^2*x_4+x_2*x_3*x_4+x_1*x_4^2)
% pt = ideal(-x_0,-x_1,-x_2,-x_3)
\BettiTable{0.15\linewidth}{1}{\underline{1},\underline{4},5,6,7}{0,4,5,6,7}{rrrr}{
        &0&1&2&3\\ \hline
        \text{total:}&1&6&8&3\\ \hline
        0:&1&\sbd&\sbd&\sbd\\
        1:&\sbd&\sbd&\sbd&\sbd\\
        %% 1\text{--}8:&\sbd&\sbd&\sbd&\sbd\\
        %% 2:&\sbd&\sbd&\sbd&\sbd\\
        %% 3:&\sbd&\sbd&\sbd&\sbd\\
        %% 4:&\sbd&\sbd&\sbd&\sbd\\
        %% 5:&\sbd&\sbd&\sbd&\sbd\\
        %% 6:&\sbd&\sbd&\sbd&\sbd\\
        %% 7:&\sbd&\sbd&\sbd&\sbd\\
        %% 8:&\sbd&\sbd&\sbd&\sbd\\
        9:&\sbd&1&\sbd&\sbd\\
        10:&\sbd&1&\sbd&\sbd\\
        11:&\sbd&2&\sbd&\sbd\\
        12:&\sbd&1&\sbd&\sbd\\
        13:&\sbd&1&\sbd&\sbd\\
        14:&\sbd&\sbd&1&\sbd\\
        15:&\sbd&\sbd&2&\sbd\\
        16:&\sbd&\sbd&2&\sbd\\
        17:&\sbd&\sbd&2&\sbd\\
        18:&\sbd&\sbd&1&\sbd\\
        19:&\sbd&\sbd&\sbd&\sbd\\
        20:&\sbd&\sbd&\sbd&1\\
        21:&\sbd&\sbd&\sbd&1\\
        22:&\sbd&\sbd&\sbd&1\\
        23:&\sbd&\sbd&\sbd&\sbd\\
}
\hspace{0.01\linewidth}%
\BettiTable{0.15\linewidth}{2}{\underline{1},\underline{2},3,3,4}{0,4,5,6,7}{rrrr}{
        &0&1&2&3\\ \hline
        \text{total:}&1&6&8&3\\ \hline
        0:&1&\sbd&\sbd&\sbd\\
        1:&\sbd&\sbd&\sbd&\sbd\\
        2:&\sbd&\sbd&\sbd&\sbd\\
        3:&\sbd&\sbd&\sbd&\sbd\\
        4:&\sbd&\sbd&\sbd&\sbd\\
        5:&\sbd&3&\sbd&\sbd\\
        6:&\sbd&2&\sbd&\sbd\\
        7:&\sbd&1&2&\sbd\\
        8:&\sbd&\sbd&4&\sbd\\
        9:&\sbd&\sbd&2&\sbd\\
        10:&\sbd&\sbd&\sbd&2\\
        11:&\sbd&\sbd&\sbd&1\\
        12:&\sbd&\sbd&\sbd&\sbd\\
}
\hspace{0.01\linewidth}%
\BettiTable{0.20\linewidth}{3}{\underline{1},2,2,\underline{2},3,3,3}{0,4,5,6,7,8,9}{rrrrrr}{
        &0&1&2&3&4&5\\ \hline
        \text{total:}&1&15&40&45&24&5\\ \hline
        0:&1&\sbd&\sbd&\sbd&\sbd&\sbd\\
        1:&\sbd&\sbd&\sbd&\sbd&\sbd&\sbd\\
        2:&\sbd&\sbd&\sbd&\sbd&\sbd&\sbd\\
        3:&\sbd&3&\sbd&\sbd&\sbd&\sbd\\
        4:&\sbd&6&2&\sbd&\sbd&\sbd\\
        5:&\sbd&6&12&\sbd&\sbd&\sbd\\
        6:&\sbd&\sbd&18&6&\sbd&\sbd\\
        7:&\sbd&\sbd&8&18&\sbd&\sbd\\
        8:&\sbd&\sbd&\sbd&18&6&\sbd\\
        9:&\sbd&\sbd&\sbd&3&12&\sbd\\
        10:&\sbd&\sbd&\sbd&\sbd&6&2\\
        11:&\sbd&\sbd&\sbd&\sbd&\sbd&3\\
        12:&\sbd&\sbd&\sbd&\sbd&\sbd&\sbd\\
}
\hspace{0.01\linewidth}%
\BettiTable{0.15\linewidth}{4}{\underline{1},\underline{1},2,2,2}{0,4,5,6,7}{rrrr}{
        &0&1&2&3\\ \hline
        \text{total:}&1&6&8&3\\ \hline
        0:&1&\sbd&\sbd&\sbd\\
        1:&\sbd&\sbd&\sbd&\sbd\\
        2:&\sbd&\sbd&\sbd&\sbd\\
        3:&\sbd&6&\sbd&\sbd\\
        4:&\sbd&\sbd&8&\sbd\\
        5:&\sbd&\sbd&\sbd&3\\
        6:&\sbd&\sbd&\sbd&\sbd\\
}
\hspace{0.01\linewidth}%
\par
\BettiTable{0.15\linewidth}{5}{\underline{1},1,\underline{1},2,2}{0,4,5,6,7}{rrrr}{
        &0&1&2&3\\ \hline
        \text{total:}&1&6&8&3\\ \hline
        0:&1&\sbd&\sbd&\sbd\\
        1:&\sbd&\bd&\sbd&\sbd\\
        2:&\sbd&3&\bd&\sbd\\
        3:&\sbd&3&6&\bd\\
        4:&\sbd&\sbd&2&3\\
        5:&\sbd&\sbd&\sbd&\sbd\\
}
\hspace{0.01\linewidth}%
\BettiTable{0.15\linewidth}{6}{\underline{1},1,1,\underline{1},2}{0,4,5,6,7}{rrrr}{
        &0&1&2&3\\ \hline
        \text{total:}&1&6&8&3\\ \hline
        0:&1&\sbd&\sbd&\sbd\\
        1:&\sbd&1&\bd&\sbd\\
        2:&\sbd&4&4&\bd\\
        3:&\sbd&1&4&3\\
        4:&\sbd&\sbd&\sbd&\sbd\\
}
\hspace{0.01\linewidth}%
\BettiTable{0.15\linewidth}{7}{\underline{1},1,1,1,\underline{1}}{0,4,5,6,7}{rrrr}{
        &0&1&2&3\\ \hline
        \text{total:}&1&6&8&3\\ \hline
        0:&1&\sbd&\sbd&\sbd\\
        1:&\sbd&3&2&\bd\\
        2:&\sbd&3&6&3\\
        3:&\sbd&\sbd&\sbd&\sbd\\
}
\hspace{0.01\linewidth}%
%%%%%%%%%%%%%%%%%%%%%
\end{adjustwidth}
\caption{$C = \Proj \mathbb F_3[x_0..x_4]/(x_2^2-x_0x_3-x_1x_3,x_0x_2+x_3^2-x_2x_4,x_0^2+x_0x_1+x_2x_3-x_0x_4-x_1x_4,x_1^2x_2+x_0x_1x_3+x_2x_3^2+x_0x_3x_4-x_1x_3x_4,x_1^3-x_1^2x_3-x_1x_2x_3+x_0x_3^2-x_3^3+x_0x_1x_4+x_1^2x_4-x_2x_3x_4-x_3^2x_4-x_1x_4^2+x_2x_4^2,x_0x_1^2+x_1^2x_3+x_1x_2x_3+x_3^3-x_0x_1x_4-x_1^2x_4+x_2x_3x_4+x_1x_4^2)$}
\label{fig:betti-genus-3-semigroup-4-5-6-7}
\end{figure}

\newpage

\begin{figure}[H]
\begin{adjustwidth}{-0.5in}{-0.3in}
\centering

$H(P) = \langle 3,5,7\rangle$

%%%%%%%%%%%%%%%%%%%%%
% kk = ZZ/3
% H  = {3, 5, 7} -- genus 3
% R  = kk[x_0..x_4]/(x_0*x_2+x_1*x_2-x_2^2-x_0*x_3+x_1*x_4,x_1^2-x_0*x_3,x_0*x_1+x_1*x_2-x_0*x_3+x_1*x_3-x_3^2+x_2*x_4,x_0^2*x_3+x_2^2*x_3+x_0*x_3^2-x_3^3+x_1*x_2*x_4-x_1*x_3*x_4+x_2*x_3*x_4,x_1*x_2^2+x_1*x_2*x_3-x_2^2*x_3+x_0*x_3^2+x_1*x_3^2+x_2*x_3^2-x_3^3-x_2^2*x_4+x_0*x_3*x_4+x_1*x_3*x_4+x_2*x_3*x_4,x_0^3-x_2^3-x_2^2*x_3+x_1*x_3^2-x_3^3-x_0*x_3*x_4+x_2*x_3*x_4+x_3^2*x_4-x_0*x_4^2-x_2*x_4^2,x_2^3*x_3-x_2*x_3^3-x_3^4-x_0*x_3^2*x_4-x_2*x_3^2*x_4-x_2^2*x_4^2+x_0*x_3*x_4^2+x_1*x_3*x_4^2)
% pt = ideal(-x_0,-x_1,-x_2,-x_3)
\BettiTable{0.15\linewidth}{1}{\underline{1},\underline{3},5,7}{0,3,5,7}{rrr}{
        &0&1&2\\ \hline
        \text{total:}&1&3&2\\ \hline
        0:&1&\sbd&\sbd\\
        1:&\sbd&\sbd&\sbd\\
        2:&\sbd&\sbd&\sbd\\
        3:&\sbd&\sbd&\sbd\\
        4:&\sbd&\sbd&\sbd\\
        5:&\sbd&\sbd&\sbd\\
        6:&\sbd&\sbd&\sbd\\
        7:&\sbd&\sbd&\sbd\\
        8:&\sbd&\sbd&\sbd\\
        9:&\sbd&1&\sbd\\
        10:&\sbd&\bd&\sbd\\
        11:&\sbd&1&\sbd\\
        12:&\sbd&\bd&\sbd\\
        13:&\sbd&1&\sbd\\
        14:&\sbd&\sbd&\sbd\\
        15:&\sbd&\sbd&1\\
        16:&\sbd&\sbd&\bd\\
        17:&\sbd&\sbd&1\\
        18:&\sbd&\sbd&\sbd\\
}
\hspace{0.01\linewidth}%
\BettiTable{0.20\linewidth}{2}{\underline{1},2,3,\underline{3},4,4,5}{0,3,5,6,7,8,10}{rrrrrr}{
        &0&1&2&3&4&5\\ \hline
        \text{total:}&1&15&40&45&24&5\\ \hline
        0:&1&\sbd&\sbd&\sbd&\sbd&\sbd\\
        1:&\sbd&\sbd&\sbd&\sbd&\sbd&\sbd\\
        2:&\sbd&\sbd&\sbd&\sbd&\sbd&\sbd\\
        3:&\sbd&1&\sbd&\sbd&\sbd&\sbd\\
        4:&\sbd&1&\sbd&\sbd&\sbd&\sbd\\
        5:&\sbd&3&1&\sbd&\sbd&\sbd\\
        6:&\sbd&3&3&\sbd&\sbd&\sbd\\
        7:&\sbd&4&5&\sbd&\sbd&\sbd\\
        8:&\sbd&2&8&2&\sbd&\sbd\\
        9:&\sbd&1&9&4&\sbd&\sbd\\
        10:&\sbd&\sbd&7&8&\sbd&\sbd\\
        11:&\sbd&\sbd&5&9&1&\sbd\\
        12:&\sbd&\sbd&2&10&3&\sbd\\
        13:&\sbd&\sbd&\sbd&7&5&\sbd\\
        14:&\sbd&\sbd&\sbd&4&6&\sbd\\
        15:&\sbd&\sbd&\sbd&1&5&1\\
        16:&\sbd&\sbd&\sbd&\sbd&3&1\\
        17:&\sbd&\sbd&\sbd&\sbd&1&2\\
        18:&\sbd&\sbd&\sbd&\sbd&\sbd&1\\
        19:&\sbd&\sbd&\sbd&\sbd&\sbd&\sbd\\
}
\hspace{0.01\linewidth}%
\BettiTable{0.15\linewidth}{3}{\underline{1},\underline{1},2,3}{0,3,5,7}{rrr}{
        &0&1&2\\ \hline
        \text{total:}&1&3&2\\ \hline
        0:&1&\sbd&\sbd\\
        1:&\sbd&\sbd&\sbd\\
        2:&\sbd&\sbd&\sbd\\
        3:&\sbd&1&\sbd\\
        4:&\sbd&1&\sbd\\
        5:&\sbd&1&1\\
        6:&\sbd&\sbd&1\\
        7:&\sbd&\sbd&\sbd\\
}
\hspace{0.01\linewidth}%
\BettiTable{0.18\linewidth}{4}{\underline{1},1,2,2,\underline{2},3}{0,3,5,7,8,12}{rrrrr}{
        &0&1&2&3&4\\ \hline
        \text{total:}&1&10&20&15&4\\ \hline
        0:&1&\sbd&\sbd&\sbd&\sbd\\
        1:&\sbd&\bd&\sbd&\sbd&\sbd\\
        2:&\sbd&1&\bd&\sbd&\sbd\\
        3:&\sbd&6&2&\bd&\sbd\\
        4:&\sbd&2&9&1&\sbd\\
        5:&\sbd&1&7&5&\bd\\
        6:&\sbd&\sbd&2&8&1\\
        7:&\sbd&\sbd&\sbd&1&3\\
        8:&\sbd&\sbd&\sbd&\sbd&\sbd\\
}
\hspace{0.01\linewidth}%
\par
\BettiTable{0.15\linewidth}{5}{\underline{1},1,\underline{1},2,2}{0,3,5,7,9}{rrrr}{
        &0&1&2&3\\ \hline
        \text{total:}&1&6&8&3\\ \hline
        0:&1&\sbd&\sbd&\sbd\\
        1:&\sbd&\bd&\sbd&\sbd\\
        2:&\sbd&3&\bd&\sbd\\
        3:&\sbd&3&6&\bd\\
        4:&\sbd&\sbd&2&3\\
        5:&\sbd&\sbd&\sbd&\sbd\\
}
\hspace{0.01\linewidth}%
\BettiTable{0.15\linewidth}{6}{\underline{1},1,1,\underline{1},2}{0,3,5,6,7}{rrrr}{
        &0&1&2&3\\ \hline
        \text{total:}&1&6&8&3\\ \hline
        0:&1&\sbd&\sbd&\sbd\\
        1:&\sbd&1&\bd&\sbd\\
        2:&\sbd&4&4&\bd\\
        3:&\sbd&1&4&3\\
        4:&\sbd&\sbd&\sbd&\sbd\\
}
\hspace{0.01\linewidth}%
\BettiTable{0.15\linewidth}{7}{\underline{1},1,1,1,\underline{1}}{0,3,5,6,7}{rrrr}{
        &0&1&2&3\\ \hline
        \text{total:}&1&4&6&3\\ \hline
        0:&1&\sbd&\sbd&\sbd\\
        1:&\sbd&3&\bd&\bd\\
        2:&\sbd&1&6&3\\
        3:&\sbd&\sbd&\sbd&\sbd\\
}
\hspace{0.01\linewidth}%
%%%%%%%%%%%%%%%%%%%%%
\end{adjustwidth}
\caption{$C = \Proj \mathbb F_3[x_0..x_4]/(x_0x_2+x_1x_2-x_2^2-x_0x_3+x_1x_4,x_1^2-x_0x_3,x_0x_1+x_1x_2-x_0x_3+x_1x_3-x_3^2+x_2x_4,x_0^2x_3+x_2^2x_3+x_0x_3^2-x_3^3+x_1x_2x_4-x_1x_3x_4+x_2x_3x_4,x_1x_2^2+x_1x_2x_3-x_2^2x_3+x_0x_3^2+x_1x_3^2+x_2x_3^2-x_3^3-x_2^2x_4+x_0x_3x_4+x_1x_3x_4+x_2x_3x_4,x_0^3-x_2^3-x_2^2x_3+x_1x_3^2-x_3^3-x_0x_3x_4+x_2x_3x_4+x_3^2x_4-x_0x_4^2-x_2x_4^2,x_2^3x_3-x_2x_3^3-x_3^4-x_0x_3^2x_4-x_2x_3^2x_4-x_2^2x_4^2+x_0x_3x_4^2+x_1x_3x_4^2)$}
\label{fig:betti-genus-3-semigroup-3-5-7}
\end{figure}

\newpage

\begin{figure}[H]
\begin{adjustwidth}{-0.5in}{-0.3in}
\centering

$H(P) = \langle 3,4\rangle$

%%%%%%%%%%%%%%%%%%%%%
% kk = ZZ/3
% H  = {3, 4} -- genus 3
% R  = kk[x_0..x_4]/(x_2*x_3-x_1*x_4,x_1*x_2-x_0*x_4,x_1^2-x_0*x_3,x_0^2*x_3-x_0*x_1*x_3-x_0*x_3^2+x_3^3+x_0*x_1*x_4-x_0*x_3*x_4+x_1*x_3*x_4+x_0*x_4^2-x_2*x_4^2,x_0^2*x_1+x_0*x_1*x_3-x_0*x_3^2+x_1*x_3^2+x_3^3+x_0^2*x_4+x_0*x_2*x_4-x_2^2*x_4+x_1*x_3*x_4+x_0*x_4^2-x_2*x_4^2,x_0^3+x_0^2*x_2+x_0*x_2^2-x_2^3-x_0*x_3^2+x_1*x_3^2-x_3^3-x_0*x_1*x_4+x_0*x_2*x_4-x_2^2*x_4-x_0*x_3*x_4-x_1*x_3*x_4-x_0*x_4^2+x_2*x_4^2)
% pt = ideal(-x_0,-x_1,-x_2,-x_3)
\BettiTable{0.15\linewidth}{1}{\underline{1},\underline{3},4}{0,3,4}{rr}{
        &0&1\\ \hline
        \text{total:}&1&1\\ \hline
        0:&1&\sbd\\
        1:&\sbd&\sbd\\
        2:&\sbd&\sbd\\
        3:&\sbd&\sbd\\
        4:&\sbd&\sbd\\
        5:&\sbd&\sbd\\
        6:&\sbd&\sbd\\
        7:&\sbd&\bd\\
        8:&\sbd&\bd\\
        9:&\sbd&\bd\\
        10:&\sbd&\bd\\
        11:&\sbd&1\\
        12:&\sbd&\sbd\\
}
\hspace{0.01\linewidth}%
\BettiTable{0.15\linewidth}{2}{\underline{1},2,\underline{2},3}{0,3,4,6}{rrr}{
        &0&1&2\\ \hline
        \text{total:}&1&2&1\\ \hline
        0:&1&\sbd&\sbd\\
        1:&\sbd&\sbd&\sbd\\
        2:&\sbd&\sbd&\sbd\\
        3:&\sbd&1&\sbd\\
        4:&\sbd&\bd&\sbd\\
        5:&\sbd&1&\bd\\
        6:&\sbd&\bd&\bd\\
        7:&\sbd&\bd&\bd\\
        8:&\sbd&\sbd&1\\
        9:&\sbd&\sbd&\sbd\\
}
\hspace{0.01\linewidth}%
\BettiTable{0.15\linewidth}{3}{\underline{1},\underline{1},2,3}{0,3,4,8}{rrr}{
        &0&1&2\\ \hline
        \text{total:}&1&3&2\\ \hline
        0:&1&\sbd&\sbd\\
        1:&\sbd&\sbd&\sbd\\
        2:&\sbd&\sbd&\sbd\\
        3:&\sbd&1&\sbd\\
        4:&\sbd&1&\sbd\\
        5:&\sbd&1&1\\
        6:&\sbd&\sbd&1\\
        7:&\sbd&\sbd&\sbd\\
}
\hspace{0.01\linewidth}%
\BettiTable{0.15\linewidth}{4}{\underline{1},1,\underline{1}}{0,3,4}{rr}{
        &0&1\\ \hline
        \text{total:}&1&1\\ \hline
        0:&1&\sbd\\
        1:&\sbd&\bd\\
        2:&\sbd&\bd\\
        3:&\sbd&1\\
        4:&\sbd&\sbd\\
}
\hspace{0.01\linewidth}%
\par
\BettiTable{0.18\linewidth}{5}{\underline{1},1,1,2,\underline{2},3}{0,3,4,9,10,15}{rrrrr}{
        &0&1&2&3&4\\ \hline
        \text{total:}&1&10&20&15&4\\ \hline
        0:&1&\sbd&\sbd&\sbd&\sbd\\
        1:&\sbd&\bd&\bd&\sbd&\sbd\\
        2:&\sbd&4&1&\bd&\sbd\\
        3:&\sbd&4&7&1&\sbd\\
        4:&\sbd&1&7&4&\bd\\
        5:&\sbd&1&4&7&1\\
        6:&\sbd&\sbd&1&3&3\\
        7:&\sbd&\sbd&\sbd&\sbd&\sbd\\
}
\hspace{0.01\linewidth}%
\BettiTable{0.15\linewidth}{6}{\underline{1},1,1,\underline{1},2}{0,3,4,6,11}{rrrr}{
        &0&1&2&3\\ \hline
        \text{total:}&1&6&8&3\\ \hline
        0:&1&\sbd&\sbd&\sbd\\
        1:&\sbd&1&\bd&\sbd\\
        2:&\sbd&4&4&\bd\\
        3:&\sbd&1&4&3\\
        4:&\sbd&\sbd&\sbd&\sbd\\
}
\hspace{0.01\linewidth}%
\BettiTable{0.15\linewidth}{7}{\underline{1},1,1,1,\underline{1}}{0,3,4,6,7}{rrrr}{
        &0&1&2&3\\ \hline
        \text{total:}&1&6&8&3\\ \hline
        0:&1&\sbd&\sbd&\sbd\\
        1:&\sbd&3&2&\bd\\
        2:&\sbd&3&6&3\\
        3:&\sbd&\sbd&\sbd&\sbd\\
}
\hspace{0.01\linewidth}%
%%%%%%%%%%%%%%%%%%%%%
\end{adjustwidth}
\caption{$C = \Proj \mathbb F_3[x_0..x_4]/(x_2x_3-x_1x_4,x_1x_2-x_0x_4,x_1^2-x_0x_3,x_0^2x_3-x_0x_1x_3-x_0x_3^2+x_3^3+x_0x_1x_4-x_0x_3x_4+x_1x_3x_4+x_0x_4^2-x_2x_4^2,x_0^2x_1+x_0x_1x_3-x_0x_3^2+x_1x_3^2+x_3^3+x_0^2x_4+x_0x_2x_4-x_2^2x_4+x_1x_3x_4+x_0x_4^2-x_2x_4^2,x_0^3+x_0^2x_2+x_0x_2^2-x_2^3-x_0x_3^2+x_1x_3^2-x_3^3-x_0x_1x_4+x_0x_2x_4-x_2^2x_4-x_0x_3x_4-x_1x_3x_4-x_0x_4^2+x_2x_4^2)$}
\label{fig:betti-genus-3-semigroup-3-4}
\end{figure}

\newpage

\begin{figure}[H]
\begin{adjustwidth}{-0.5in}{-0.3in}
\centering

$H(P) = \langle 2,7\rangle$

%%%%%%%%%%%%%%%%%%%%%
% kk = ZZ/3
% H  = {2, 7} -- genus 3
% R  = kk[x_0..x_4]/(x_2^2-x_1*x_3,x_1*x_2-x_0*x_3,x_1^2-x_0*x_2,x_0^2*x_3+x_0*x_1*x_3-x_0*x_3^2-x_1*x_3^2+x_3^3-x_2*x_4^2,x_0^2*x_2-x_0*x_1*x_3-x_0*x_2*x_3+x_1*x_3^2+x_2*x_3^2-x_3^3-x_1*x_4^2+x_2*x_4^2,x_0^2*x_1-x_2*x_3^2+x_3^3-x_0*x_4^2+x_1*x_4^2-x_2*x_4^2)
% pt = ideal(-x_0,-x_1,-x_2,-x_3)
\BettiTable{0.15\linewidth}{1}{\underline{1},\underline{2},7}{0,2,7}{rr}{
        &0&1\\ \hline
        \text{total:}&1&1\\ \hline
        0:&1&\sbd\\
        1:&\sbd&\sbd\\
        2:&\sbd&\sbd\\
        3:&\sbd&\sbd\\
        4:&\sbd&\sbd\\
        5:&\sbd&\sbd\\
        6:&\sbd&\sbd\\
        7:&\sbd&\sbd\\
        8:&\sbd&\sbd\\
        9:&\sbd&\sbd\\
        10:&\sbd&\sbd\\
        11:&\sbd&\sbd\\
        12:&\sbd&\sbd\\
        13:&\sbd&1\\
        14:&\sbd&\sbd\\
}
\hspace{0.01\linewidth}%
\BettiTable{0.15\linewidth}{2}{\underline{1},\underline{1},4}{0,2,7}{rr}{
        &0&1\\ \hline
        \text{total:}&1&1\\ \hline
        0:&1&\sbd\\
        1:&\sbd&\sbd\\
        2:&\sbd&\sbd\\
        3:&\sbd&\sbd\\
        4:&\sbd&\sbd\\
        5:&\sbd&\sbd\\
        6:&\sbd&\sbd\\
        7:&\sbd&1\\
        8:&\sbd&\sbd\\
}
\hspace{0.01\linewidth}%
\BettiTable{0.15\linewidth}{3}{\underline{1},1,\underline{2},3,3}{0,2,6,7,9}{rrrr}{
        &0&1&2&3\\ \hline
        \text{total:}&1&6&8&3\\ \hline
        0:&1&\sbd&\sbd&\sbd\\
        1:&\sbd&\bd&\sbd&\sbd\\
        2:&\sbd&1&\sbd&\sbd\\
        3:&\sbd&1&\bd&\sbd\\
        4:&\sbd&1&2&\sbd\\
        5:&\sbd&3&2&\bd\\
        6:&\sbd&\bd&2&1\\
        7:&\sbd&\sbd&2&1\\
        8:&\sbd&\sbd&\bd&1\\
        9:&\sbd&\sbd&\sbd&\sbd\\
}
\hspace{0.01\linewidth}%
\BettiTable{0.15\linewidth}{4}{\underline{1},1,\underline{1},2}{0,2,4,7}{rrr}{
        &0&1&2\\ \hline
        \text{total:}&1&2&1\\ \hline
        0:&1&\sbd&\sbd\\
        1:&\sbd&1&\sbd\\
        2:&\sbd&\bd&\bd\\
        3:&\sbd&1&\bd\\
        4:&\sbd&\bd&1\\
        5:&\sbd&\sbd&\sbd\\
}
\hspace{0.01\linewidth}%
\par
\BettiTable{0.20\linewidth}{5}{\underline{1},1,1,2,2,\underline{2},3}{0,2,4,7,9,10,15}{rrrrrr}{
        &0&1&2&3&4&5\\ \hline
        \text{total:}&1&15&40&45&24&5\\ \hline
        0:&1&\sbd&\sbd&\sbd&\sbd&\sbd\\
        1:&\sbd&1&\bd&\sbd&\sbd&\sbd\\
        2:&\sbd&4&4&\bd&\sbd&\sbd\\
        3:&\sbd&7&12&5&\bd&\sbd\\
        4:&\sbd&2&14&14&2&\sbd\\
        5:&\sbd&1&8&15&8&\bd\\
        6:&\sbd&\sbd&2&10&10&2\\
        7:&\sbd&\sbd&\sbd&1&4&3\\
        8:&\sbd&\sbd&\sbd&\sbd&\sbd&\sbd\\
}
\hspace{0.01\linewidth}%
\BettiTable{0.20\linewidth}{6}{\underline{1},1,1,\underline{1},2,2,2}{0,2,4,6,7,9,11}{rrrrrr}{
        &0&1&2&3&4&5\\ \hline
        \text{total:}&1&15&40&45&24&5\\ \hline
        0:&1&\sbd&\sbd&\sbd&\sbd&\sbd\\
        1:&\sbd&3&2&\sbd&\sbd&\sbd\\
        2:&\sbd&6&12&6&\sbd&\sbd\\
        3:&\sbd&6&18&18&6&\sbd\\
        4:&\sbd&\sbd&8&18&12&2\\
        5:&\sbd&\sbd&\sbd&3&6&3\\
        6:&\sbd&\sbd&\sbd&\sbd&\sbd&\sbd\\
}
\hspace{0.01\linewidth}%
\BettiTable{0.15\linewidth}{7}{\underline{1},1,1,1,\underline{1}}{0,2,4,6,7}{rrrr}{
        &0&1&2&3\\ \hline
        \text{total:}&1&6&8&3\\ \hline
        0:&1&\sbd&\sbd&\sbd\\
        1:&\sbd&3&2&\bd\\
        2:&\sbd&3&6&3\\
        3:&\sbd&\sbd&\sbd&\sbd\\
}
\hspace{0.01\linewidth}%
%%%%%%%%%%%%%%%%%%%%%
\end{adjustwidth}
\caption{$C = \Proj \mathbb F_3[x_0..x_4]/(x_2^2-x_1x_3,x_1x_2-x_0x_3,x_1^2-x_0x_2,x_0^2x_3+x_0x_1x_3-x_0x_3^2-x_1x_3^2+x_3^3-x_2x_4^2,x_0^2x_2-x_0x_1x_3-x_0x_2x_3+x_1x_3^2+x_2x_3^2-x_3^3-x_1x_4^2+x_2x_4^2,x_0^2x_1-x_2x_3^2+x_3^3-x_0x_4^2+x_1x_4^2-x_2x_4^2)$}
\label{fig:betti-genus-3-semigroup-2-7}
\end{figure}

  \newpage
  \section{Betti tables of weighted embeddings with $g = 4$}
   Here we give examples of the family of Betti tables of section rings $R_d = R(C,\OO(dP)$ for points exhibiting all possible Weierstrass semigroups for curves of genus 4.
   \vspace{-0.10in}
  
\begin{figure}[H]
\begin{adjustwidth}{-0.5in}{-0.3in}
\centering

$H(P) = \langle 4,5,6 \rangle$

%%%%%%%%%%%%%%%%%%%%%
% kk = ZZ/3
% H  = {4, 5, 6} -- genus 4
% R  = kk[x_0..x_5]/(x_2*x_4-x_1*x_5,x_1*x_4+x_3*x_4-x_2*x_5,x_2^2-x_1*x_3-x_0*x_4,x_1*x_2-x_0*x_5,x_1^2-x_0*x_4,x_0^2-x_0*x_1-x_0*x_2-x_3^2-x_3*x_4+x_2*x_5,x_3^2*x_4-x_0*x_4^2-x_2*x_3*x_5+x_0*x_5^2,x_0*x_3*x_4+x_0*x_4^2-x_3*x_4^2-x_4^3-x_0*x_2*x_5+x_0*x_4*x_5+x_3*x_5^2)
% pt = ideal(-x_0,-x_1,-x_2,-x_3,-x_4)
\BettiTable{0.15\linewidth}{1}{\underline{1},\underline{4},5,6}{0,4,5,6}{rrr}{
        &0&1&2\\ \hline
        \text{total:}&1&2&1\\ \hline
        0:&1&\sbd&\sbd\\
        1\text{--}8:&\sbd&\sbd&\sbd\\
        % 1:&\sbd&\sbd&\sbd\\
        % 2:&\sbd&\sbd&\sbd\\
        % 3:&\sbd&\sbd&\sbd\\
        % 4:&\sbd&\sbd&\sbd\\
        % 5:&\sbd&\sbd&\sbd\\
        % 6:&\sbd&\sbd&\sbd\\
        % 7:&\sbd&\sbd&\sbd\\
        % 8:&\sbd&\sbd&\sbd\\
        9:&\sbd&1&\sbd\\
        10:&\sbd&\bd&\sbd\\
        11:&\sbd&1&\sbd\\
        12:&\sbd&\bd&\sbd\\
        13:&\sbd&\bd&\sbd\\
        14:&\sbd&\bd&\bd\\
        15:&\sbd&\bd&\bd\\
        16:&\sbd&\bd&\bd\\
        17:&\sbd&\sbd&\bd\\
        18:&\sbd&\sbd&\bd\\
        19:&\sbd&\sbd&\bd\\
        20:&\sbd&\sbd&1\\
        21:&\sbd&\sbd&\sbd\\
}
\hspace{0.01\linewidth}%
\BettiTable{0.15\linewidth}{2}{\underline{1},\underline{2},3,3}{0,4,5,6}{rrr}{
        &0&1&2\\ \hline
        \text{total:}&1&2&1\\ \hline
        0:&1&\sbd&\sbd\\
        1:&\sbd&\sbd&\sbd\\
        2:&\sbd&\sbd&\sbd\\
        3:&\sbd&\sbd&\sbd\\
        4:&\sbd&\sbd&\sbd\\
        5:&\sbd&2&\sbd\\
        6:&\sbd&\bd&\sbd\\
        7:&\sbd&\bd&\bd\\
        8:&\sbd&\bd&\bd\\
        9:&\sbd&\sbd&\bd\\
        10:&\sbd&\sbd&1\\
        11:&\sbd&\sbd&\sbd\\
}
\hspace{0.01\linewidth}%
\BettiTable{0.18\linewidth}{3}{\underline{1},2,2,\underline{2},3,3}{0,4,5,6,8,9}{rrrrr}{
        &0&1&2&3&4\\ \hline
        \text{total:}&1&9&16&9&1\\ \hline
        0:&1&\sbd&\sbd&\sbd&\sbd\\
        1:&\sbd&\sbd&\sbd&\sbd&\sbd\\
        2:&\sbd&\sbd&\sbd&\sbd&\sbd\\
        3:&\sbd&3&\sbd&\sbd&\sbd\\
        4:&\sbd&3&2&\sbd&\sbd\\
        5:&\sbd&3&6&\sbd&\sbd\\
        6:&\sbd&\bd&6&3&\sbd\\
        7:&\sbd&\bd&2&3&\sbd\\
        8:&\sbd&\sbd&\bd&3&\bd\\
        9:&\sbd&\sbd&\bd&\bd&\bd\\
        10:&\sbd&\sbd&\sbd&\bd&\bd\\
        11:&\sbd&\sbd&\sbd&\sbd&1\\
        12:&\sbd&\sbd&\sbd&\sbd&\sbd\\
}
\hspace{0.01\linewidth}%
\BettiTable{0.15\linewidth}{4}{\underline{1},\underline{1},2,2,3}{0,4,5,6,11}{rrrr}{
        &0&1&2&3\\ \hline
        \text{total:}&1&6&8&3\\ \hline
        0:&1&\sbd&\sbd&\sbd\\
        1:&\sbd&\sbd&\sbd&\sbd\\
        2:&\sbd&\sbd&\sbd&\sbd\\
        3:&\sbd&3&\sbd&\sbd\\
        4:&\sbd&2&2&\sbd\\
        5:&\sbd&1&4&\sbd\\
        6:&\sbd&\sbd&2&2\\
        7:&\sbd&\sbd&\sbd&1\\
        8:&\sbd&\sbd&\sbd&\sbd\\
}
\hspace{0.01\linewidth}%
\BettiTable{0.15\linewidth}{5}{\underline{1},1,\underline{1},2}{0,4,5,6}{rrr}{
        &0&1&2\\ \hline
        \text{total:}&1&3&2\\ \hline
        0:&1&\sbd&\sbd\\
        1:&\sbd&\bd&\sbd\\
        2:&\sbd&1&\bd\\
        3:&\sbd&2&1\\
        4:&\sbd&\bd&1\\
        5:&\sbd&\sbd&\sbd\\
}
\hspace{0.01\linewidth}%
\BettiTable{0.15\linewidth}{6}{\underline{1},1,1,\underline{1}}{0,4,5,6}{rrr}{
        &0&1&2\\ \hline
        \text{total:}&1&2&1\\ \hline
        0:&1&\sbd&\sbd\\
        1:&\sbd&1&\bd\\
        2:&\sbd&1&\bd\\
        3:&\sbd&\bd&1\\
        4:&\sbd&\sbd&\sbd\\
}
\hspace{0.01\linewidth}%
\BettiTable{0.20\linewidth}{7}{\underline{1},1,1,1,2,\underline{2},3}{0,4,5,6,13,14,21}{rrrrrr}{
        &0&1&2&3&4&5\\ \hline
        \text{total:}&1&15&40&45&24&5\\ \hline
        0:&1&\sbd&\sbd&\sbd&\sbd&\sbd\\
        1:&\sbd&1&\bd&\bd&\sbd&\sbd\\
        2:&\sbd&7&8&1&\bd&\sbd\\
        3:&\sbd&5&15&11&1&\sbd\\
        4:&\sbd&1&11&15&5&\bd\\
        5:&\sbd&1&5&14&11&1\\
        6:&\sbd&\sbd&1&4&7&4\\
        7:&\sbd&\sbd&\sbd&\sbd&\sbd&\sbd\\
}
\hspace{0.01\linewidth}%
\BettiTable{0.18\linewidth}{8}{\underline{1},1,1,1,\underline{1},2}{0,4,5,6,8,15}{rrrrr}{
        &0&1&2&3&4\\ \hline
        \text{total:}&1&8&16&13&4\\ \hline
        0:&1&\sbd&\sbd&\sbd&\sbd\\
        1:&\sbd&3&\bd&\bd&\sbd\\
        2:&\sbd&4&12&4&\bd\\
        3:&\sbd&1&4&9&4\\
        4:&\sbd&\sbd&\sbd&\sbd&\sbd\\
}
\hspace{0.01\linewidth}%
\BettiTable{0.18\linewidth}{9}{\underline{1},1,1,1,1,\underline{1}}{0,4,5,6,8,9}{rrrrr}{
        &0&1&2&3&4\\ \hline
        \text{total:}&1&7&14&12&4\\ \hline
        0:&1&\sbd&\sbd&\sbd&\sbd\\
        1:&\sbd&6&5&\bd&\bd\\
        2:&\sbd&1&9&12&4\\
        3:&\sbd&\sbd&\sbd&\sbd&\sbd\\
}
\hspace{0.01\linewidth}%
%%%%%%%%%%%%%%%%%%%%%
\end{adjustwidth}
\caption{$C = \Proj \mathbb F_3[x_0..x_5]/(x_2x_4-x_1x_5,x_1x_4+x_3x_4-x_2x_5,x_2^2-x_1x_3-x_0x_4,x_1x_2-x_0x_5,x_1^2-x_0x_4,x_0^2-x_0x_1-x_0x_2-x_3^2-x_3x_4+x_2x_5,x_3^2x_4-x_0x_4^2-x_2x_3x_5+x_0x_5^2,x_0x_3x_4+x_0x_4^2-x_3x_4^2-x_4^3-x_0x_2x_5+x_0x_4x_5+x_3x_5^2)$}
\label{fig:betti-genus-4-semigroup-4-5-6}
\end{figure}

\newpage

\begin{figure}[H]
\begin{adjustwidth}{-0.5in}{-0.3in}
\centering

$H(P) = \langle 5,6,7,8,9\rangle$

%%%%%%%%%%%%%%%%%%%%%
% kk = ZZ/3
% H  = {5, 6, 7, 8, 9} -- genus 4
% R  = kk[x_0..x_5]/(x_1*x_3+x_2*x_3+x_3^2-x_0*x_4+x_1*x_4-x_4^2+x_0*x_5-x_1*x_5+x_3*x_5,x_1*x_2+x_2^2-x_0*x_3+x_2*x_3-x_3^2-x_1*x_4-x_2*x_4-x_3*x_4+x_0*x_5-x_1*x_5+x_2*x_5,x_0*x_2+x_0*x_3-x_2*x_3+x_0*x_4+x_3*x_4-x_4^2-x_0*x_5-x_2*x_5+x_3*x_5,x_1^2+x_2^2+x_0*x_3+x_2*x_3+x_3^2+x_0*x_4+x_1*x_4-x_4^2+x_0*x_5-x_1*x_5+x_3*x_5,x_0*x_1+x_2^2+x_0*x_3+x_2*x_3+x_3^2+x_4^2+x_0*x_5-x_1*x_5-x_3*x_5,x_0^2-x_2^2+x_3^2-x_0*x_4+x_1*x_4+x_2*x_4+x_3*x_4+x_4^2+x_1*x_5-x_2*x_5-x_3*x_5,x_3^3+x_2*x_3*x_4+x_3^2*x_4-x_0*x_4^2-x_2*x_4^2+x_3*x_4^2-x_4^3+x_0*x_3*x_5-x_0*x_4*x_5+x_1*x_4*x_5-x_2*x_4*x_5+x_3*x_4*x_5+x_4^2*x_5-x_3*x_5^2,x_2*x_3^2-x_0*x_3*x_4+x_2*x_3*x_4-x_3^2*x_4-x_0*x_4^2-x_1*x_4^2+x_2*x_4^2+x_3*x_4^2-x_4^3+x_2^2*x_5-x_0*x_3*x_5+x_2*x_3*x_5+x_3^2*x_5-x_1*x_4*x_5-x_3*x_4*x_5+x_1*x_5^2-x_2*x_5^2,x_0*x_3^2-x_2^2*x_4-x_0*x_3*x_4+x_2*x_4^2+x_3*x_4^2+x_4^3+x_2^2*x_5-x_0*x_3*x_5+x_2*x_3*x_5+x_3^2*x_5+x_0*x_4*x_5-x_1*x_4*x_5+x_2*x_4*x_5+x_4^2*x_5+x_0*x_5^2-x_2*x_5^2-x_3*x_5^2,x_2^2*x_3-x_1*x_4^2-x_2*x_4^2+x_0*x_3*x_5+x_2*x_3*x_5-x_0*x_4*x_5-x_2*x_4*x_5-x_0*x_5^2+x_1*x_5^2,x_2^3-x_2^2*x_4+x_2*x_3*x_4+x_3^2*x_4+x_0*x_4^2+x_1*x_4^2+x_2*x_4^2+x_4^3-x_2^2*x_5-x_0*x_3*x_5-x_2*x_3*x_5+x_3^2*x_5+x_0*x_4*x_5+x_1*x_4*x_5-x_2*x_4*x_5+x_3*x_4*x_5-x_4^2*x_5+x_0*x_5^2+x_2*x_5^2+x_3*x_5^2)
% pt = ideal(-x_0,-x_1,-x_2,-x_3,-x_4)
\BettiTable{0.18\linewidth}{1}{\underline{1},\underline{5},6,7,8,9}{0,5,6,7,8,9}{rrrrr}{
        &0&1&2&3&4\\ \hline
        \text{total:}&1&10&20&15&4\\ \hline
        0:&1&\sbd&\sbd&\sbd&\sbd\\
        1\text{--}10:&\sbd&\sbd&\sbd&\sbd&\sbd\\
        % 1:&\sbd&\sbd&\sbd&\sbd&\sbd\\
        % 2:&\sbd&\sbd&\sbd&\sbd&\sbd\\
        % 3:&\sbd&\sbd&\sbd&\sbd&\sbd\\
        % 4:&\sbd&\sbd&\sbd&\sbd&\sbd\\
        % 5:&\sbd&\sbd&\sbd&\sbd&\sbd\\
        % 6:&\sbd&\sbd&\sbd&\sbd&\sbd\\
        % 7:&\sbd&\sbd&\sbd&\sbd&\sbd\\
        % 8:&\sbd&\sbd&\sbd&\sbd&\sbd\\
        % 9:&\sbd&\sbd&\sbd&\sbd&\sbd\\
        % 10:&\sbd&\sbd&\sbd&\sbd&\sbd\\
        11:&\sbd&1&\sbd&\sbd&\sbd\\
        12:&\sbd&1&\sbd&\sbd&\sbd\\
        13:&\sbd&2&\sbd&\sbd&\sbd\\
        14:&\sbd&2&\sbd&\sbd&\sbd\\
        15:&\sbd&2&\sbd&\sbd&\sbd\\
        16:&\sbd&1&\sbd&\sbd&\sbd\\
        17:&\sbd&1&1&\sbd&\sbd\\
        18:&\sbd&\sbd&2&\sbd&\sbd\\
        19:&\sbd&\sbd&3&\sbd&\sbd\\
        20:&\sbd&\sbd&4&\sbd&\sbd\\
        21:&\sbd&\sbd&4&\sbd&\sbd\\
        22:&\sbd&\sbd&3&\sbd&\sbd\\
        23:&\sbd&\sbd&2&\sbd&\sbd\\
        24:&\sbd&\sbd&1&1&\sbd\\
        25:&\sbd&\sbd&\sbd&2&\sbd\\
        26:&\sbd&\sbd&\sbd&3&\sbd\\
        27:&\sbd&\sbd&\sbd&3&\sbd\\
        28:&\sbd&\sbd&\sbd&3&\sbd\\
        29:&\sbd&\sbd&\sbd&2&\sbd\\
        30:&\sbd&\sbd&\sbd&1&\sbd\\
        31:&\sbd&\sbd&\sbd&\sbd&\sbd\\
        32:&\sbd&\sbd&\sbd&\sbd&1\\
        33:&\sbd&\sbd&\sbd&\sbd&1\\
        34:&\sbd&\sbd&\sbd&\sbd&1\\
        35:&\sbd&\sbd&\sbd&\sbd&1\\
        36:&\sbd&\sbd&\sbd&\sbd&\sbd\\
}
\hspace{0.01\linewidth}%
\BettiTable{0.20\linewidth}{2}{\underline{1},3,\underline{3},4,4,5,5}{0,5,6,7,8,9,10}{rrrrrr}{
        &0&1&2&3&4&5\\ \hline
        \text{total:}&1&15&40&45&24&5\\ \hline
        0:&1&\sbd&\sbd&\sbd&\sbd&\sbd\\
        1:&\sbd&\sbd&\sbd&\sbd&\sbd&\sbd\\
        2:&\sbd&\sbd&\sbd&\sbd&\sbd&\sbd\\
        3:&\sbd&\sbd&\sbd&\sbd&\sbd&\sbd\\
        4:&\sbd&\sbd&\sbd&\sbd&\sbd&\sbd\\
        5:&\sbd&1&\sbd&\sbd&\sbd&\sbd\\
        6:&\sbd&2&\sbd&\sbd&\sbd&\sbd\\
        7:&\sbd&5&\sbd&\sbd&\sbd&\sbd\\
        8:&\sbd&4&2&\sbd&\sbd&\sbd\\
        9:&\sbd&3&6&\sbd&\sbd&\sbd\\
        10:&\sbd&\sbd&10&\sbd&\sbd&\sbd\\
        11:&\sbd&\sbd&12&1&\sbd&\sbd\\
        12:&\sbd&\sbd&8&6&\sbd&\sbd\\
        13:&\sbd&\sbd&2&11&\sbd&\sbd\\
        14:&\sbd&\sbd&\sbd&14&\sbd&\sbd\\
        15:&\sbd&\sbd&\sbd&9&2&\sbd\\
        16:&\sbd&\sbd&\sbd&4&6&\sbd\\
        17:&\sbd&\sbd&\sbd&\sbd&8&\sbd\\
        18:&\sbd&\sbd&\sbd&\sbd&6&\sbd\\
        19:&\sbd&\sbd&\sbd&\sbd&2&1\\
        20:&\sbd&\sbd&\sbd&\sbd&\sbd&2\\
        21:&\sbd&\sbd&\sbd&\sbd&\sbd&2\\
        22:&\sbd&\sbd&\sbd&\sbd&\sbd&\sbd\\
}
\hspace{0.01\linewidth}%
\BettiTable{0.18\linewidth}{3}{\underline{1},2,\underline{2},3,3,3}{0,5,6,7,8,9}{rrrrr}{
        &0&1&2&3&4\\ \hline
        \text{total:}&1&10&20&15&4\\ \hline
        0:&1&\sbd&\sbd&\sbd&\sbd\\
        1:&\sbd&\sbd&\sbd&\sbd&\sbd\\
        2:&\sbd&\sbd&\sbd&\sbd&\sbd\\
        3:&\sbd&\bd&\sbd&\sbd&\sbd\\
        4:&\sbd&3&\sbd&\sbd&\sbd\\
        5:&\sbd&7&\bd&\sbd&\sbd\\
        6:&\sbd&\bd&9&\sbd&\sbd\\
        7:&\sbd&\sbd&11&\bd&\sbd\\
        8:&\sbd&\sbd&\bd&9&\sbd\\
        9:&\sbd&\sbd&\sbd&6&\bd\\
        10:&\sbd&\sbd&\sbd&\bd&3\\
        11:&\sbd&\sbd&\sbd&\sbd&1\\
        12:&\sbd&\sbd&\sbd&\sbd&\sbd\\
}
\hspace{0.01\linewidth}%
\BettiTable{0.25\linewidth}{4}{\underline{1},2,2,2,\underline{2},3,3,3,3}{0,5,6,7,8,9,10,11,12}{rrrrrrrr}{
        &0&1&2&3&4&5&6&7\\ \hline
        \text{total:}&1&28&112&210&224&140&48&7\\ \hline
        0:&1&\sbd&\sbd&\sbd&\sbd&\sbd&\sbd&\sbd\\
        1:&\sbd&\sbd&\sbd&\sbd&\sbd&\sbd&\sbd&\sbd\\
        2:&\sbd&\sbd&\sbd&\sbd&\sbd&\sbd&\sbd&\sbd\\
        3:&\sbd&6&\sbd&\sbd&\sbd&\sbd&\sbd&\sbd\\
        4:&\sbd&12&8&\sbd&\sbd&\sbd&\sbd&\sbd\\
        5:&\sbd&10&36&3&\sbd&\sbd&\sbd&\sbd\\
        6:&\sbd&\sbd&48&36&\sbd&\sbd&\sbd&\sbd\\
        7:&\sbd&\sbd&20&84&12&\sbd&\sbd&\sbd\\
        8:&\sbd&\sbd&\sbd&72&64&\sbd&\sbd&\sbd\\
        9:&\sbd&\sbd&\sbd&15&96&18&\sbd&\sbd\\
        10:&\sbd&\sbd&\sbd&\sbd&48&56&\sbd&\sbd\\
        11:&\sbd&\sbd&\sbd&\sbd&4&54&12&\sbd\\
        12:&\sbd&\sbd&\sbd&\sbd&\sbd&12&24&\sbd\\
        13:&\sbd&\sbd&\sbd&\sbd&\sbd&\sbd&12&3\\
        14:&\sbd&\sbd&\sbd&\sbd&\sbd&\sbd&\sbd&4\\
        15:&\sbd&\sbd&\sbd&\sbd&\sbd&\sbd&\sbd&\sbd\\
}
\hspace{0.01\linewidth}%
\BettiTable{0.18\linewidth}{5}{\underline{1},\underline{1},2,2,2,2}{0,5,6,7,8,9}{rrrrr}{
        &0&1&2&3&4\\ \hline
        \text{total:}&1&10&20&15&4\\ \hline
        0:&1&\sbd&\sbd&\sbd&\sbd\\
        1:&\sbd&\sbd&\sbd&\sbd&\sbd\\
        2:&\sbd&\sbd&\sbd&\sbd&\sbd\\
        3:&\sbd&10&\sbd&\sbd&\sbd\\
        4:&\sbd&\sbd&20&\sbd&\sbd\\
        5:&\sbd&\sbd&\sbd&15&\sbd\\
        6:&\sbd&\sbd&\sbd&\sbd&4\\
        7:&\sbd&\sbd&\sbd&\sbd&\sbd\\
}
\hspace{0.01\linewidth}%
\BettiTable{0.18\linewidth}{6}{\underline{1},1,\underline{1},2,2,2}{0,5,6,7,8,9}{rrrrr}{
        &0&1&2&3&4\\ \hline
        \text{total:}&1&10&20&15&4\\ \hline
        0:&1&\sbd&\sbd&\sbd&\sbd\\
        1:&\sbd&\bd&\sbd&\sbd&\sbd\\
        2:&\sbd&4&\bd&\sbd&\sbd\\
        3:&\sbd&6&12&\bd&\sbd\\
        4:&\sbd&\sbd&8&12&\bd\\
        5:&\sbd&\sbd&\sbd&3&4\\
        6:&\sbd&\sbd&\sbd&\sbd&\sbd\\
}
\hspace{0.01\linewidth}%
\BettiTable{0.15\linewidth}{7}{\underline{1},1,1,\underline{1},2}{0,5,6,7,9}{rrrr}{
        &0&1&2&3\\ \hline
        \text{total:}&1&7&10&4\\ \hline
        0:&1&\sbd&\sbd&\sbd\\
        1:&\sbd&\bd&\bd&\sbd\\
        2:&\sbd&6&4&\bd\\
        3:&\sbd&1&6&4\\
        4:&\sbd&\sbd&\sbd&\sbd\\
}
\hspace{0.01\linewidth}%
\BettiTable{0.15\linewidth}{8}{\underline{1},1,1,1,\underline{1}}{0,5,6,7,8}{rrrr}{
        &0&1&2&3\\ \hline
        \text{total:}&1&6&9&4\\ \hline
        0:&1&\sbd&\sbd&\sbd\\
        1:&\sbd&2&\bd&\bd\\
        2:&\sbd&4&9&4\\
        3:&\sbd&\sbd&\sbd&\sbd\\
}
\hspace{0.01\linewidth}%
\BettiTable{0.18\linewidth}{9}{\underline{1},1,1,1,1,\underline{1}}{0,5,6,7,8,9}{rrrrr}{
        &0&1&2&3&4\\ \hline
        \text{total:}&1&6&13&12&4\\ \hline
        0:&1&\sbd&\sbd&\sbd&\sbd\\
        1:&\sbd&6&4&\bd&\bd\\
        2:&\sbd&\bd&9&12&4\\
        3:&\sbd&\sbd&\sbd&\sbd&\sbd\\
}
\hspace{0.01\linewidth}%
%%%%%%%%%%%%%%%%%%%%%
\end{adjustwidth}
\caption{$C = \Proj \mathbb F_3[x_0..x_5]/(x_1x_3+x_2x_3+x_3^2-x_0x_4+x_1x_4-x_4^2+x_0x_5-x_1x_5+x_3x_5,x_1x_2+x_2^2-x_0x_3+x_2x_3-x_3^2-x_1x_4-x_2x_4-x_3x_4+x_0x_5-x_1x_5+x_2x_5,x_0x_2+x_0x_3-x_2x_3+x_0x_4+x_3x_4-x_4^2-x_0x_5-x_2x_5+x_3x_5,x_1^2+x_2^2+x_0x_3+x_2x_3+x_3^2+x_0x_4+x_1x_4-x_4^2+x_0x_5-x_1x_5+x_3x_5,x_0x_1+x_2^2+x_0x_3+x_2x_3+x_3^2+x_4^2+x_0x_5-x_1x_5-x_3x_5,x_0^2-x_2^2+x_3^2-x_0x_4+x_1x_4+x_2x_4+x_3x_4+x_4^2+x_1x_5-x_2x_5-x_3x_5,x_3^3+x_2x_3x_4+x_3^2x_4-x_0x_4^2-x_2x_4^2+x_3x_4^2-x_4^3+x_0x_3x_5-x_0x_4x_5+x_1x_4x_5-x_2x_4x_5+x_3x_4x_5+x_4^2x_5-x_3x_5^2,x_2x_3^2-x_0x_3x_4+x_2x_3x_4-x_3^2x_4-x_0x_4^2-x_1x_4^2+x_2x_4^2+x_3x_4^2-x_4^3+x_2^2x_5-x_0x_3x_5+x_2x_3x_5+x_3^2x_5-x_1x_4x_5-x_3x_4x_5+x_1x_5^2-x_2x_5^2,x_0x_3^2-x_2^2x_4-x_0x_3x_4+x_2x_4^2+x_3x_4^2+x_4^3+x_2^2x_5-x_0x_3x_5+x_2x_3x_5+x_3^2x_5+x_0x_4x_5-x_1x_4x_5+x_2x_4x_5+x_4^2x_5+x_0x_5^2-x_2x_5^2-x_3x_5^2,x_2^2x_3-x_1x_4^2-x_2x_4^2+x_0x_3x_5+x_2x_3x_5-x_0x_4x_5-x_2x_4x_5-x_0x_5^2+x_1x_5^2,x_2^3-x_2^2x_4+x_2x_3x_4+x_3^2x_4+x_0x_4^2+x_1x_4^2+x_2x_4^2+x_4^3-x_2^2x_5-x_0x_3x_5-x_2x_3x_5+x_3^2x_5+x_0x_4x_5+x_1x_4x_5-x_2x_4x_5+x_3x_4x_5-x_4^2x_5+x_0x_5^2+x_2x_5^2+x_3x_5^2)$}
\label{fig:betti-genus-4-semigroup-5-6-7-8-9}
\end{figure}

\newpage

\begin{figure}[H]
\begin{adjustwidth}{-0.5in}{-0.3in}
\centering

$H(P) = \langle 4,6,7,9\rangle$

%%%%%%%%%%%%%%%%%%%%%
% kk = ZZ/3
% H  = {4, 6, 7, 9} -- genus 4
% R  = kk[x_0..x_5]/(x_1*x_3-x_0*x_4-x_1*x_4-x_3*x_4-x_0*x_5+x_2*x_5,x_2^2-x_2*x_3+x_3^2-x_0*x_4-x_3*x_4+x_4^2-x_0*x_5+x_1*x_5+x_2*x_5-x_3*x_5,x_1*x_2-x_2*x_3+x_2*x_4+x_4^2+x_0*x_5+x_1*x_5-x_3*x_5,x_0*x_2+x_2*x_3+x_3^2-x_0*x_4-x_1*x_4+x_2*x_4-x_4^2+x_0*x_5-x_1*x_5+x_3*x_5,x_1^2-x_0*x_4,x_0*x_1-x_0*x_3+x_2*x_3-x_0*x_4+x_3*x_4+x_0*x_5-x_1*x_5-x_2*x_5,x_3^2*x_4+x_1*x_4^2-x_2*x_4^2+x_0^2*x_5-x_0*x_3*x_5-x_2*x_3*x_5+x_3^2*x_5+x_0*x_4*x_5+x_2*x_4*x_5-x_3*x_4*x_5-x_4^2*x_5+x_1*x_5^2+x_2*x_5^2+x_3*x_5^2,x_2*x_3*x_4-x_1*x_4^2-x_0*x_3*x_5-x_2*x_3*x_5-x_0*x_4*x_5-x_1*x_4*x_5-x_2*x_4*x_5+x_3*x_4*x_5-x_4^2*x_5+x_1*x_5^2-x_2*x_5^2+x_3*x_5^2,x_3^3-x_0*x_3*x_4-x_0*x_4^2+x_1*x_4^2+x_2*x_4^2-x_4^3+x_2*x_3*x_5-x_3^2*x_5+x_0*x_4*x_5+x_1*x_4*x_5+x_2*x_4*x_5-x_3*x_4*x_5-x_0*x_5^2-x_2*x_5^2,x_2*x_3^2-x_0*x_4^2-x_1*x_4^2-x_3*x_4^2-x_0^2*x_5+x_0*x_3*x_5+x_2*x_3*x_5+x_0*x_4*x_5-x_1*x_4*x_5+x_2*x_4*x_5+x_0*x_5^2-x_1*x_5^2,x_0*x_3^2-x_0^2*x_4-x_0*x_3*x_4+x_0*x_4^2+x_1*x_4^2-x_2*x_4^2-x_0^2*x_5+x_2*x_3*x_5+x_1*x_4*x_5+x_2*x_4*x_5-x_3*x_4*x_5+x_4^2*x_5+x_0*x_5^2+x_2*x_5^2-x_3*x_5^2,x_0*x_3*x_4^2-x_0*x_4^3-x_1*x_4^3-x_2*x_4^3+x_3*x_4^3+x_4^4+x_0^2*x_3*x_5-x_0^2*x_4*x_5-x_1*x_4^2*x_5+x_2*x_4^2*x_5+x_0^2*x_5^2+x_2*x_3*x_5^2-x_3^2*x_5^2+x_0*x_4*x_5^2+x_0*x_5^3,x_0^2*x_4^2+x_0*x_4^3-x_1*x_4^3-x_2*x_4^3+x_0^3*x_5+x_0^2*x_4*x_5+x_0*x_3*x_4*x_5-x_0*x_4^2*x_5-x_1*x_4^2*x_5-x_3*x_4^2*x_5-x_4^3*x_5-x_0^2*x_5^2-x_3^2*x_5^2-x_0*x_4*x_5^2-x_1*x_4*x_5^2+x_4^2*x_5^2-x_1*x_5^3+x_2*x_5^3-x_3*x_5^3)
% pt = ideal(-x_0,-x_1,-x_2,-x_3,-x_4)
\BettiTable{0.15\linewidth}{1}{\underline{1},\underline{4},6,7,9}{0,4,6,7,9}{rrrr}{
        &0&1&2&3\\ \hline
        \text{total:}&1&6&8&3\\ \hline
        0:&1&\sbd&\sbd&\sbd\\
        1\text{--}10:&\sbd&\sbd&\sbd&\sbd\\
        % 1:&\sbd&\sbd&\sbd&\sbd\\
        % 2:&\sbd&\sbd&\sbd&\sbd\\
        % 3:&\sbd&\sbd&\sbd&\sbd\\
        % 4:&\sbd&\sbd&\sbd&\sbd\\
        % 5:&\sbd&\sbd&\sbd&\sbd\\
        % 6:&\sbd&\sbd&\sbd&\sbd\\
        % 7:&\sbd&\sbd&\sbd&\sbd\\
        % 8:&\sbd&\sbd&\sbd&\sbd\\
        % 9:&\sbd&\sbd&\sbd&\sbd\\
        % 10:&\sbd&\sbd&\sbd&\sbd\\
        11:&\sbd&1&\sbd&\sbd\\
        12:&\sbd&1&\sbd&\sbd\\
        13:&\sbd&1&\sbd&\sbd\\
        14:&\sbd&1&\sbd&\sbd\\
        15:&\sbd&1&\sbd&\sbd\\
        16:&\sbd&\bd&\sbd&\sbd\\
        17:&\sbd&1&1&\sbd\\
        18:&\sbd&\sbd&1&\sbd\\
        19:&\sbd&\sbd&1&\sbd\\
        20:&\sbd&\sbd&2&\sbd\\
        21:&\sbd&\sbd&1&\sbd\\
        22:&\sbd&\sbd&1&\sbd\\
        23:&\sbd&\sbd&1&\sbd\\
        24:&\sbd&\sbd&\sbd&\sbd\\
        25:&\sbd&\sbd&\sbd&1\\
        26:&\sbd&\sbd&\sbd&1\\
        27:&\sbd&\sbd&\sbd&\bd\\
        28:&\sbd&\sbd&\sbd&1\\
        29:&\sbd&\sbd&\sbd&\sbd\\
}
\hspace{0.01\linewidth}%
\BettiTable{0.15\linewidth}{2}{\underline{1},\underline{2},3,4,5}{0,4,6,7,9}{rrrr}{
        &0&1&2&3\\ \hline
        \text{total:}&1&6&8&3\\ \hline
        0:&1&\sbd&\sbd&\sbd\\
        1:&\sbd&\sbd&\sbd&\sbd\\
        2:&\sbd&\sbd&\sbd&\sbd\\
        3:&\sbd&\sbd&\sbd&\sbd\\
        4:&\sbd&\sbd&\sbd&\sbd\\
        5:&\sbd&1&\sbd&\sbd\\
        6:&\sbd&1&\sbd&\sbd\\
        7:&\sbd&2&\sbd&\sbd\\
        8:&\sbd&1&1&\sbd\\
        9:&\sbd&1&2&\sbd\\
        10:&\sbd&\sbd&2&\sbd\\
        11:&\sbd&\sbd&2&\sbd\\
        12:&\sbd&\sbd&1&1\\
        13:&\sbd&\sbd&\sbd&1\\
        14:&\sbd&\sbd&\sbd&1\\
        15:&\sbd&\sbd&\sbd&\sbd\\
}
\hspace{0.01\linewidth}%
\BettiTable{0.20\linewidth}{3}{\underline{1},2,\underline{2},3,3,3,4}{0,4,6,7,8,9,11}{rrrrrr}{
        &0&1&2&3&4&5\\ \hline
        \text{total:}&1&15&40&45&24&5\\ \hline
        0:&1&\sbd&\sbd&\sbd&\sbd&\sbd\\
        1:&\sbd&\sbd&\sbd&\sbd&\sbd&\sbd\\
        2:&\sbd&\sbd&\sbd&\sbd&\sbd&\sbd\\
        3:&\sbd&1&\sbd&\sbd&\sbd&\sbd\\
        4:&\sbd&3&\sbd&\sbd&\sbd&\sbd\\
        5:&\sbd&7&3&\sbd&\sbd&\sbd\\
        6:&\sbd&3&10&\sbd&\sbd&\sbd\\
        7:&\sbd&1&14&3&\sbd&\sbd\\
        8:&\sbd&\sbd&10&12&\sbd&\sbd\\
        9:&\sbd&\sbd&3&15&1&\sbd\\
        10:&\sbd&\sbd&\sbd&12&6&\sbd\\
        11:&\sbd&\sbd&\sbd&3&10&\sbd\\
        12:&\sbd&\sbd&\sbd&\sbd&6&1\\
        13:&\sbd&\sbd&\sbd&\sbd&1&3\\
        14:&\sbd&\sbd&\sbd&\sbd&\sbd&1\\
        15:&\sbd&\sbd&\sbd&\sbd&\sbd&\sbd\\
}
\hspace{0.01\linewidth}%
\BettiTable{0.15\linewidth}{4}{\underline{1},\underline{1},2,2,3}{0,4,6,7,9}{rrrr}{
        &0&1&2&3\\ \hline
        \text{total:}&1&6&8&3\\ \hline
        0:&1&\sbd&\sbd&\sbd\\
        1:&\sbd&\sbd&\sbd&\sbd\\
        2:&\sbd&\sbd&\sbd&\sbd\\
        3:&\sbd&3&\sbd&\sbd\\
        4:&\sbd&2&2&\sbd\\
        5:&\sbd&1&4&\sbd\\
        6:&\sbd&\sbd&2&2\\
        7:&\sbd&\sbd&\sbd&1\\
        8:&\sbd&\sbd&\sbd&\sbd\\
}
\hspace{0.01\linewidth}%
\BettiTable{0.20\linewidth}{5}{\underline{1},1,2,2,2,\underline{2},3}{0,4,6,7,9,10,15}{rrrrrr}{
        &0&1&2&3&4&5\\ \hline
        \text{total:}&1&15&40&45&24&5\\ \hline
        0:&1&\sbd&\sbd&\sbd&\sbd&\sbd\\
        1:&\sbd&\bd&\sbd&\sbd&\sbd&\sbd\\
        2:&\sbd&1&\bd&\sbd&\sbd&\sbd\\
        3:&\sbd&10&3&\bd&\sbd&\sbd\\
        4:&\sbd&3&21&3&\bd&\sbd\\
        5:&\sbd&1&13&18&1&\sbd\\
        6:&\sbd&\sbd&3&21&7&\bd\\
        7:&\sbd&\sbd&\sbd&3&15&1\\
        8:&\sbd&\sbd&\sbd&\sbd&1&4\\
        9:&\sbd&\sbd&\sbd&\sbd&\sbd&\sbd\\
}
\hspace{0.01\linewidth}%
\BettiTable{0.18\linewidth}{6}{\underline{1},1,\underline{1},2,2,2}{0,4,6,7,9,11}{rrrrr}{
        &0&1&2&3&4\\ \hline
        \text{total:}&1&10&20&15&4\\ \hline
        0:&1&\sbd&\sbd&\sbd&\sbd\\
        1:&\sbd&\bd&\sbd&\sbd&\sbd\\
        2:&\sbd&4&\bd&\sbd&\sbd\\
        3:&\sbd&6&12&\bd&\sbd\\
        4:&\sbd&\sbd&8&12&\bd\\
        5:&\sbd&\sbd&\sbd&3&4\\
        6:&\sbd&\sbd&\sbd&\sbd&\sbd\\
}
\hspace{0.01\linewidth}%
\BettiTable{0.15\linewidth}{7}{\underline{1},1,1,\underline{1},2}{0,4,6,7,9}{rrrr}{
        &0&1&2&3\\ \hline
        \text{total:}&1&7&10&4\\ \hline
        0:&1&\sbd&\sbd&\sbd\\
        1:&\sbd&\bd&\bd&\sbd\\
        2:&\sbd&6&4&\bd\\
        3:&\sbd&1&6&4\\
        4:&\sbd&\sbd&\sbd&\sbd\\
}
\hspace{0.01\linewidth}%
\BettiTable{0.15\linewidth}{8}{\underline{1},1,1,1,\underline{1}}{0,4,6,7,8}{rrrr}{
        &0&1&2&3\\ \hline
        \text{total:}&1&6&9&4\\ \hline
        0:&1&\sbd&\sbd&\sbd\\
        1:&\sbd&2&\bd&\bd\\
        2:&\sbd&4&9&4\\
        3:&\sbd&\sbd&\sbd&\sbd\\
}
\hspace{0.01\linewidth}%
\BettiTable{0.18\linewidth}{9}{\underline{1},1,1,1,1,\underline{1}}{0,4,6,7,8,9}{rrrrr}{
        &0&1&2&3&4\\ \hline
        \text{total:}&1&6&13&12&4\\ \hline
        0:&1&\sbd&\sbd&\sbd&\sbd\\
        1:&\sbd&6&4&\bd&\bd\\
        2:&\sbd&\bd&9&12&4\\
        3:&\sbd&\sbd&\sbd&\sbd&\sbd\\
}
\hspace{0.01\linewidth}%
%%%%%%%%%%%%%%%%%%%%%
\end{adjustwidth}
\caption{$C = \Proj \mathbb F_3[x_0..x_5]/(x_1x_3-x_0x_4-x_1x_4-x_3x_4-x_0x_5+x_2x_5,x_2^2-x_2x_3+x_3^2-x_0x_4-x_3x_4+x_4^2-x_0x_5+x_1x_5+x_2x_5-x_3x_5,x_1x_2-x_2x_3+x_2x_4+x_4^2+x_0x_5+x_1x_5-x_3x_5,x_0x_2+x_2x_3+x_3^2-x_0x_4-x_1x_4+x_2x_4-x_4^2+x_0x_5-x_1x_5+x_3x_5,x_1^2-x_0x_4,x_0x_1-x_0x_3+x_2x_3-x_0x_4+x_3x_4+x_0x_5-x_1x_5-x_2x_5,x_3^2x_4+x_1x_4^2-x_2x_4^2+x_0^2x_5-x_0x_3x_5-x_2x_3x_5+x_3^2x_5+x_0x_4x_5+x_2x_4x_5-x_3x_4x_5-x_4^2x_5+x_1x_5^2+x_2x_5^2+x_3x_5^2,x_2x_3x_4-x_1x_4^2-x_0x_3x_5-x_2x_3x_5-x_0x_4x_5-x_1x_4x_5-x_2x_4x_5+x_3x_4x_5-x_4^2x_5+x_1x_5^2-x_2x_5^2+x_3x_5^2,x_3^3-x_0x_3x_4-x_0x_4^2+x_1x_4^2+x_2x_4^2-x_4^3+x_2x_3x_5-x_3^2x_5+x_0x_4x_5+x_1x_4x_5+x_2x_4x_5-x_3x_4x_5-x_0x_5^2-x_2x_5^2,x_2x_3^2-x_0x_4^2-x_1x_4^2-x_3x_4^2-x_0^2x_5+x_0x_3x_5+x_2x_3x_5+x_0x_4x_5-x_1x_4x_5+x_2x_4x_5+x_0x_5^2-x_1x_5^2,x_0x_3^2-x_0^2x_4-x_0x_3x_4+x_0x_4^2+x_1x_4^2-x_2x_4^2-x_0^2x_5+x_2x_3x_5+x_1x_4x_5+x_2x_4x_5-x_3x_4x_5+x_4^2x_5+x_0x_5^2+x_2x_5^2-x_3x_5^2,x_0x_3x_4^2-x_0x_4^3-x_1x_4^3-x_2x_4^3+x_3x_4^3+x_4^4+x_0^2x_3x_5-x_0^2x_4x_5-x_1x_4^2x_5+x_2x_4^2x_5+x_0^2x_5^2+x_2x_3x_5^2-x_3^2x_5^2+x_0x_4x_5^2+x_0x_5^3,x_0^2x_4^2+x_0x_4^3-x_1x_4^3-x_2x_4^3+x_0^3x_5+x_0^2x_4x_5+x_0x_3x_4x_5-x_0x_4^2x_5-x_1x_4^2x_5-x_3x_4^2x_5-x_4^3x_5-x_0^2x_5^2-x_3^2x_5^2-x_0x_4x_5^2-x_1x_4x_5^2+x_4^2x_5^2-x_1x_5^3+x_2x_5^3-x_3x_5^3)$}
\label{fig:betti-genus-4-semigroup-4-6-7-9}
\end{figure}

\newpage

\begin{figure}[H]
\begin{adjustwidth}{-0.5in}{-0.3in}
\centering

$H(P) = \langle 4,5,7 \rangle$

%%%%%%%%%%%%%%%%%%%%%
% kk = ZZ/3
% H  = {4, 5, 7} -- genus 4
% R  = kk[x_0..x_5]/(x_2*x_4-x_1*x_5,x_0*x_4+x_4^2+x_1*x_5-x_2*x_5-x_3*x_5,x_2^2+x_1*x_3-x_3^2-x_1*x_4+x_4^2+x_1*x_5-x_3*x_5,x_1*x_2-x_0*x_5,x_1^2+x_4^2+x_1*x_5-x_2*x_5-x_3*x_5,x_0*x_1+x_1*x_3-x_2*x_3-x_3^2+x_4^2+x_0*x_5+x_1*x_5-x_3*x_5,x_1*x_3*x_4-x_3^2*x_4-x_1*x_4^2+x_4^3+x_1*x_4*x_5-x_3*x_4*x_5+x_0*x_5^2,x_2*x_3^2+x_3^3+x_3^2*x_4+x_4^3-x_0^2*x_5+x_0*x_2*x_5-x_0*x_3*x_5+x_1*x_3*x_5-x_3*x_4*x_5+x_4^2*x_5+x_0*x_5^2-x_1*x_5^2-x_2*x_5^2-x_3*x_5^2,x_1*x_3^2-x_1*x_4^2+x_3*x_4^2-x_4^3-x_0*x_2*x_5-x_1*x_3*x_5-x_2*x_3*x_5-x_3^2*x_5-x_1*x_4*x_5+x_3*x_4*x_5+x_4^2*x_5-x_1*x_5^2-x_2*x_5^2-x_3*x_5^2,x_0^2*x_2-x_0^2*x_3+x_0*x_3^2-x_3^3+x_3*x_4^2-x_4^3-x_0^2*x_5+x_0*x_2*x_5+x_0*x_3*x_5-x_1*x_3*x_5+x_3^2*x_5+x_1*x_4*x_5+x_3*x_4*x_5+x_0*x_5^2+x_1*x_5^2+x_2*x_5^2,x_3^3*x_4+x_3^2*x_4^2+x_4^4+x_3^2*x_4*x_5-x_1*x_4^2*x_5-x_3*x_4^2*x_5-x_0*x_3*x_5^2-x_2*x_3*x_5^2-x_3^2*x_5^2+x_1*x_4*x_5^2-x_4^2*x_5^2+x_1*x_5^3-x_2*x_5^3+x_3*x_5^3,x_0^2*x_3^3+x_0*x_3^4-x_3^5+x_4^5+x_0^4*x_5-x_0*x_3^3*x_5+x_3^4*x_5-x_3^2*x_4^2*x_5-x_1*x_4^3*x_5-x_4^4*x_5+x_0^2*x_3*x_5^2+x_3^3*x_5^2-x_3*x_4^2*x_5^2+x_4^3*x_5^2-x_0^2*x_5^3+x_1*x_3*x_5^3+x_2*x_3*x_5^3+x_1*x_4*x_5^3-x_3*x_4*x_5^3+x_0*x_5^4+x_1*x_5^4-x_2*x_5^4)
% pt = ideal(-x_0,-x_1,-x_2,-x_3,-x_4)
\BettiTable{0.15\linewidth}{1}{\underline{1},\underline{4},5,7}{0,4,5,7}{rrr}{
        &0&1&2\\ \hline
        \text{total:}&1&3&2\\ \hline
        0:&1&\sbd&\sbd\\
        1:&\sbd&\sbd&\sbd\\
        2:&\sbd&\sbd&\sbd\\
        3:&\sbd&\sbd&\sbd\\
        4:&\sbd&\sbd&\sbd\\
        5:&\sbd&\sbd&\sbd\\
        6:&\sbd&\sbd&\sbd\\
        7:&\sbd&\sbd&\sbd\\
        8:&\sbd&\sbd&\sbd\\
        9:&\sbd&\bd&\sbd\\
        10:&\sbd&\bd&\sbd\\
        11:&\sbd&1&\sbd\\
        12:&\sbd&\bd&\sbd\\
        13:&\sbd&1&\sbd\\
        14:&\sbd&1&\sbd\\
        15:&\sbd&\bd&\bd\\
        16:&\sbd&\bd&\bd\\
        17:&\sbd&\sbd&1\\
        18:&\sbd&\sbd&\bd\\
        19:&\sbd&\sbd&\bd\\
        20:&\sbd&\sbd&1\\
        21:&\sbd&\sbd&\sbd\\
}
\hspace{0.01\linewidth}%
\BettiTable{0.15\linewidth}{2}{\underline{1},\underline{2},3,4,5}{0,4,5,7,10}{rrrr}{
        &0&1&2&3\\ \hline
        \text{total:}&1&6&8&3\\ \hline
        0:&1&\sbd&\sbd&\sbd\\
        1:&\sbd&\sbd&\sbd&\sbd\\
        2:&\sbd&\sbd&\sbd&\sbd\\
        3:&\sbd&\sbd&\sbd&\sbd\\
        4:&\sbd&\sbd&\sbd&\sbd\\
        5:&\sbd&1&\sbd&\sbd\\
        6:&\sbd&1&\sbd&\sbd\\
        7:&\sbd&2&\sbd&\sbd\\
        8:&\sbd&1&1&\sbd\\
        9:&\sbd&1&2&\sbd\\
        10:&\sbd&\sbd&2&\sbd\\
        11:&\sbd&\sbd&2&\sbd\\
        12:&\sbd&\sbd&1&1\\
        13:&\sbd&\sbd&\sbd&1\\
        14:&\sbd&\sbd&\sbd&1\\
        15:&\sbd&\sbd&\sbd&\sbd\\
}
\hspace{0.01\linewidth}%
\BettiTable{0.25\linewidth}{3}{\underline{1},2,2,3,3,\underline{3},4,4,5}{0,4,5,7,8,9,11,12,15}{rrrrrrrr}{
        &0&1&2&3&4&5&6&7\\ \hline
        \text{total:}&1&28&112&210&224&140&48&7\\ \hline
        0:&1&\sbd&\sbd&\sbd&\sbd&\sbd&\sbd&\sbd\\
        1:&\sbd&\sbd&\sbd&\sbd&\sbd&\sbd&\sbd&\sbd\\
        2:&\sbd&\sbd&\sbd&\sbd&\sbd&\sbd&\sbd&\sbd\\
        3:&\sbd&2&\sbd&\sbd&\sbd&\sbd&\sbd&\sbd\\
        4:&\sbd&4&1&\sbd&\sbd&\sbd&\sbd&\sbd\\
        5:&\sbd&8&6&\sbd&\sbd&\sbd&\sbd&\sbd\\
        6:&\sbd&6&15&2&\sbd&\sbd&\sbd&\sbd\\
        7:&\sbd&5&23&8&\sbd&\sbd&\sbd&\sbd\\
        8:&\sbd&2&26&23&1&\sbd&\sbd&\sbd\\
        9:&\sbd&1&21&37&6&\sbd&\sbd&\sbd\\
        10:&\sbd&\sbd&12&47&18&\sbd&\sbd&\sbd\\
        11:&\sbd&\sbd&6&41&35&2&\sbd&\sbd\\
        12:&\sbd&\sbd&2&30&47&7&\sbd&\sbd\\
        13:&\sbd&\sbd&\sbd&15&48&18&\sbd&\sbd\\
        14:&\sbd&\sbd&\sbd&6&38&28&1&\sbd\\
        15:&\sbd&\sbd&\sbd&1&21&34&4&\sbd\\
        16:&\sbd&\sbd&\sbd&\sbd&8&27&9&\sbd\\
        17:&\sbd&\sbd&\sbd&\sbd&2&17&13&\sbd\\
        18:&\sbd&\sbd&\sbd&\sbd&\sbd&6&12&1\\
        19:&\sbd&\sbd&\sbd&\sbd&\sbd&1&7&2\\
        20:&\sbd&\sbd&\sbd&\sbd&\sbd&\sbd&2&3\\
        21:&\sbd&\sbd&\sbd&\sbd&\sbd&\sbd&\sbd&1\\
        22:&\sbd&\sbd&\sbd&\sbd&\sbd&\sbd&\sbd&\sbd\\
}
\hspace{0.01\linewidth}%
\BettiTable{0.15\linewidth}{4}{\underline{1},\underline{1},2,2,3}{0,4,5,7,10}{rrrr}{
        &0&1&2&3\\ \hline
        \text{total:}&1&6&8&3\\ \hline
        0:&1&\sbd&\sbd&\sbd\\
        1:&\sbd&\sbd&\sbd&\sbd\\
        2:&\sbd&\sbd&\sbd&\sbd\\
        3:&\sbd&3&\sbd&\sbd\\
        4:&\sbd&2&2&\sbd\\
        5:&\sbd&1&4&\sbd\\
        6:&\sbd&\sbd&2&2\\
        7:&\sbd&\sbd&\sbd&1\\
        8:&\sbd&\sbd&\sbd&\sbd\\
}
\hspace{0.01\linewidth}%
\BettiTable{0.15\linewidth}{5}{\underline{1},1,\underline{1},2}{0,4,5,7}{rrr}{
        &0&1&2\\ \hline
        \text{total:}&1&3&2\\ \hline
        0:&1&\sbd&\sbd\\
        1:&\sbd&\bd&\sbd\\
        2:&\sbd&1&\bd\\
        3:&\sbd&2&1\\
        4:&\sbd&\bd&1\\
        5:&\sbd&\sbd&\sbd\\
}
\hspace{0.01\linewidth}%
\BettiTable{0.20\linewidth}{6}{\underline{1},1,1,2,2,\underline{2},3}{0,4,5,7,11,12,18}{rrrrrr}{
        &0&1&2&3&4&5\\ \hline
        \text{total:}&1&15&40&45&24&5\\ \hline
        0:&1&\sbd&\sbd&\sbd&\sbd&\sbd\\
        1:&\sbd&\bd&\bd&\sbd&\sbd&\sbd\\
        2:&\sbd&5&1&\bd&\sbd&\sbd\\
        3:&\sbd&7&14&2&\bd&\sbd\\
        4:&\sbd&2&15&14&1&\sbd\\
        5:&\sbd&1&8&18&6&\bd\\
        6:&\sbd&\sbd&2&10&13&1\\
        7:&\sbd&\sbd&\sbd&1&4&4\\
        8:&\sbd&\sbd&\sbd&\sbd&\sbd&\sbd\\
}
\hspace{0.01\linewidth}%
\BettiTable{0.15\linewidth}{7}{\underline{1},1,1,\underline{1},2}{0,4,5,7,13}{rrrr}{
        &0&1&2&3\\ \hline
        \text{total:}&1&7&10&4\\ \hline
        0:&1&\sbd&\sbd&\sbd\\
        1:&\sbd&\bd&\bd&\sbd\\
        2:&\sbd&6&4&\bd\\
        3:&\sbd&1&6&4\\
        4:&\sbd&\sbd&\sbd&\sbd\\
}
\hspace{0.01\linewidth}%
\BettiTable{0.15\linewidth}{8}{\underline{1},1,1,1,\underline{1}}{0,4,5,7,8}{rrrr}{
        &0&1&2&3\\ \hline
        \text{total:}&1&6&9&4\\ \hline
        0:&1&\sbd&\sbd&\sbd\\
        1:&\sbd&2&\bd&\bd\\
        2:&\sbd&4&9&4\\
        3:&\sbd&\sbd&\sbd&\sbd\\
}
\hspace{0.01\linewidth}%
\BettiTable{0.18\linewidth}{9}{\underline{1},1,1,1,1,\underline{1}}{0,4,5,7,8,9}{rrrrr}{
        &0&1&2&3&4\\ \hline
        \text{total:}&1&7&14&12&4\\ \hline
        0:&1&\sbd&\sbd&\sbd&\sbd\\
        1:&\sbd&6&5&\bd&\bd\\
        2:&\sbd&1&9&12&4\\
        3:&\sbd&\sbd&\sbd&\sbd&\sbd\\
}
\hspace{0.01\linewidth}%
%%%%%%%%%%%%%%%%%%%%%
\end{adjustwidth}
\caption{$C = \Proj \mathbb F_3[x_0..x_5]/(x_2x_4-x_1x_5,x_0x_4+x_4^2+x_1x_5-x_2x_5-x_3x_5,x_2^2+x_1x_3-x_3^2-x_1x_4+x_4^2+x_1x_5-x_3x_5,x_1x_2-x_0x_5,x_1^2+x_4^2+x_1x_5-x_2x_5-x_3x_5,x_0x_1+x_1x_3-x_2x_3-x_3^2+x_4^2+x_0x_5+x_1x_5-x_3x_5,x_1x_3x_4-x_3^2x_4-x_1x_4^2+x_4^3+x_1x_4x_5-x_3x_4x_5+x_0x_5^2,x_2x_3^2+x_3^3+x_3^2x_4+x_4^3-x_0^2x_5+x_0x_2x_5-x_0x_3x_5+x_1x_3x_5-x_3x_4x_5+x_4^2x_5+x_0x_5^2-x_1x_5^2-x_2x_5^2-x_3x_5^2,x_1x_3^2-x_1x_4^2+x_3x_4^2-x_4^3-x_0x_2x_5-x_1x_3x_5-x_2x_3x_5-x_3^2x_5-x_1x_4x_5+x_3x_4x_5+x_4^2x_5-x_1x_5^2-x_2x_5^2-x_3x_5^2,x_0^2x_2-x_0^2x_3+x_0x_3^2-x_3^3+x_3x_4^2-x_4^3-x_0^2x_5+x_0x_2x_5+x_0x_3x_5-x_1x_3x_5+x_3^2x_5+x_1x_4x_5+x_3x_4x_5+x_0x_5^2+x_1x_5^2+x_2x_5^2,x_3^3x_4+x_3^2x_4^2+x_4^4+x_3^2x_4x_5-x_1x_4^2x_5-x_3x_4^2x_5-x_0x_3x_5^2-x_2x_3x_5^2-x_3^2x_5^2+x_1x_4x_5^2-x_4^2x_5^2+x_1x_5^3-x_2x_5^3+x_3x_5^3,x_0^2x_3^3+x_0x_3^4-x_3^5+x_4^5+x_0^4x_5-x_0x_3^3x_5+x_3^4x_5-x_3^2x_4^2x_5-x_1x_4^3x_5-x_4^4x_5+x_0^2x_3x_5^2+x_3^3x_5^2-x_3x_4^2x_5^2+x_4^3x_5^2-x_0^2x_5^3+x_1x_3x_5^3+x_2x_3x_5^3+x_1x_4x_5^3-x_3x_4x_5^3+x_0x_5^4+x_1x_5^4-x_2x_5^4)$}
\caption{$C = \Proj R_0$, $H(P) = \langle 4, 5, 7\rangle$}
\label{fig:betti-genus-4-semigroup-4-5-7}
\end{figure}

\newpage

\begin{figure}[H]
\begin{adjustwidth}{-0.5in}{-0.3in}
\centering

$H(P) = \langle 3,7,8\rangle$

%%%%%%%%%%%%%%%%%%%%%
% kk = ZZ/3
% H  = {3, 7, 8} -- genus 4
% R  = kk[x_0..x_5]/(x_2^2-x_1*x_5,x_1*x_2-x_0*x_5,x_0*x_2-x_2*x_3-x_3^2-x_1*x_4+x_2*x_4+x_4^2-x_0*x_5-x_1*x_5+x_2*x_5-x_3*x_5,x_1^2-x_2*x_3-x_3^2-x_1*x_4+x_2*x_4+x_4^2-x_0*x_5-x_1*x_5+x_2*x_5-x_3*x_5,x_0*x_1-x_0*x_3+x_0*x_4-x_1*x_4+x_3*x_4-x_4^2+x_1*x_5+x_2*x_5+x_3*x_5,x_0^2+x_0*x_3+x_1*x_3+x_2*x_3-x_3^2-x_1*x_4+x_2*x_4+x_3*x_4-x_4^2+x_0*x_5+x_2*x_5+x_3*x_5,x_3^3+x_1*x_3*x_4-x_3^2*x_4-x_1*x_4^2-x_3*x_4^2+x_4^3+x_1*x_3*x_5-x_3^2*x_5+x_0*x_4*x_5-x_1*x_4*x_5-x_3*x_4*x_5+x_4^2*x_5+x_0*x_5^2+x_1*x_5^2+x_2*x_5^2-x_3*x_5^2,x_2*x_3^2-x_2*x_4^2-x_0*x_3*x_5+x_1*x_3*x_5-x_2*x_3*x_5+x_3^2*x_5-x_0*x_4*x_5-x_1*x_4*x_5-x_2*x_4*x_5+x_3*x_4*x_5+x_4^2*x_5-x_0*x_5^2+x_1*x_5^2-x_3*x_5^2,x_1*x_3^2+x_2*x_3*x_4+x_3^2*x_4-x_2*x_4^2-x_4^3-x_1*x_3*x_5-x_2*x_3*x_5-x_0*x_4*x_5-x_1*x_4*x_5-x_2*x_4*x_5+x_3*x_4*x_5-x_4^2*x_5+x_0*x_5^2-x_2*x_5^2+x_3*x_5^2,x_0*x_3^2+x_0*x_3*x_4+x_2*x_3*x_4+x_3^2*x_4+x_0*x_4^2-x_1*x_4^2+x_2*x_4^2-x_3*x_4^2-x_0*x_3*x_5-x_1*x_3*x_5-x_2*x_3*x_5+x_3^2*x_5-x_0*x_4*x_5+x_1*x_4*x_5-x_4^2*x_5+x_0*x_5^2-x_2*x_5^2+x_3*x_5^2)
% pt = ideal(-x_0,-x_1,-x_2,-x_3,-x_4)
\BettiTable{0.15\linewidth}{1}{\underline{1},\underline{3},7,8}{0,3,7,8}{rrr}{
        &0&1&2\\ \hline
        \text{total:}&1&3&2\\ \hline
        0:&1&\sbd&\sbd\\
        1:&\sbd&\sbd&\sbd\\
        2:&\sbd&\sbd&\sbd\\
        3:&\sbd&\sbd&\sbd\\
        4:&\sbd&\sbd&\sbd\\
        5:&\sbd&\sbd&\sbd\\
        6:&\sbd&\sbd&\sbd\\
        7:&\sbd&\sbd&\sbd\\
        8:&\sbd&\sbd&\sbd\\
        9:&\sbd&\sbd&\sbd\\
        10:&\sbd&\sbd&\sbd\\
        11:&\sbd&\sbd&\sbd\\
        12:&\sbd&\sbd&\sbd\\
        13:&\sbd&1&\sbd\\
        14:&\sbd&1&\sbd\\
        15:&\sbd&1&\sbd\\
        16:&\sbd&\sbd&\sbd\\
        17:&\sbd&\sbd&\sbd\\
        18:&\sbd&\sbd&\sbd\\
        19:&\sbd&\sbd&\sbd\\
        20:&\sbd&\sbd&1\\
        21:&\sbd&\sbd&1\\
        22:&\sbd&\sbd&\sbd\\
}
\hspace{0.01\linewidth}%
\BettiTable{0.18\linewidth}{2}{\underline{1},2,\underline{3},4,4,5}{0,3,6,7,8,10}{rrrrr}{
        &0&1&2&3&4\\ \hline
        \text{total:}&1&9&17&12&3\\ \hline
        0:&1&\sbd&\sbd&\sbd&\sbd\\
        1:&\sbd&\sbd&\sbd&\sbd&\sbd\\
        2:&\sbd&\sbd&\sbd&\sbd&\sbd\\
        3:&\sbd&1&\sbd&\sbd&\sbd\\
        4:&\sbd&\bd&\sbd&\sbd&\sbd\\
        5:&\sbd&1&\sbd&\sbd&\sbd\\
        6:&\sbd&1&1&\sbd&\sbd\\
        7:&\sbd&3&1&\sbd&\sbd\\
        8:&\sbd&2&2&\sbd&\sbd\\
        9:&\sbd&1&3&\bd&\sbd\\
        10:&\sbd&\bd&4&1&\sbd\\
        11:&\sbd&\sbd&4&1&\sbd\\
        12:&\sbd&\sbd&2&3&\sbd\\
        13:&\sbd&\sbd&\bd&3&\bd\\
        14:&\sbd&\sbd&\sbd&3&\bd\\
        15:&\sbd&\sbd&\sbd&1&1\\
        16:&\sbd&\sbd&\sbd&\bd&1\\
        17:&\sbd&\sbd&\sbd&\sbd&1\\
        18:&\sbd&\sbd&\sbd&\sbd&\sbd\\
}
\hspace{0.01\linewidth}%
\BettiTable{0.15\linewidth}{3}{\underline{1},\underline{1},3,3}{0,3,7,8}{rrr}{
        &0&1&2\\ \hline
        \text{total:}&1&3&2\\ \hline
        0:&1&\sbd&\sbd\\
        1:&\sbd&\sbd&\sbd\\
        2:&\sbd&\sbd&\sbd\\
        3:&\sbd&\sbd&\sbd\\
        4:&\sbd&\sbd&\sbd\\
        5:&\sbd&3&\sbd\\
        6:&\sbd&\sbd&\sbd\\
        7:&\sbd&\sbd&2\\
        8:&\sbd&\sbd&\sbd\\
}
\hspace{0.01\linewidth}%
\BettiTable{0.15\linewidth}{4}{\underline{1},1,2,\underline{2},3}{0,3,7,8,12}{rrrr}{
        &0&1&2&3\\ \hline
        \text{total:}&1&6&8&3\\ \hline
        0:&1&\sbd&\sbd&\sbd\\
        1:&\sbd&\bd&\sbd&\sbd\\
        2:&\sbd&\bd&\bd&\sbd\\
        3:&\sbd&3&\bd&\sbd\\
        4:&\sbd&2&2&\bd\\
        5:&\sbd&1&4&\bd\\
        6:&\sbd&\bd&2&2\\
        7:&\sbd&\sbd&\bd&1\\
        8:&\sbd&\sbd&\sbd&\sbd\\
}
\hspace{0.01\linewidth}%
\BettiTable{0.20\linewidth}{5}{\underline{1},1,2,2,2,\underline{2},3,3}{0,3,7,8,9,10,14,15}{rrrrrrr}{
        &0&1&2&3&4&5&6\\ \hline
        \text{total:}&1&21&70&105&84&35&6\\ \hline
        0:&1&\sbd&\sbd&\sbd&\sbd&\sbd&\sbd\\
        1:&\sbd&\bd&\sbd&\sbd&\sbd&\sbd&\sbd\\
        2:&\sbd&2&\bd&\sbd&\sbd&\sbd&\sbd\\
        3:&\sbd&10&6&\bd&\sbd&\sbd&\sbd\\
        4:&\sbd&6&24&6&\bd&\sbd&\sbd\\
        5:&\sbd&3&26&27&2&\sbd&\sbd\\
        6:&\sbd&\sbd&12&44&16&\bd&\sbd\\
        7:&\sbd&\sbd&2&22&36&4&\sbd\\
        8:&\sbd&\sbd&\sbd&6&24&14&\bd\\
        9:&\sbd&\sbd&\sbd&\sbd&6&15&2\\
        10:&\sbd&\sbd&\sbd&\sbd&\sbd&2&4\\
        11:&\sbd&\sbd&\sbd&\sbd&\sbd&\sbd&\sbd\\
}
\hspace{0.01\linewidth}%
\BettiTable{0.20\linewidth}{6}{\underline{1},1,\underline{1},2,2,2,2}{0,3,6,7,8,10,11}{rrrrrr}{
        &0&1&2&3&4&5\\ \hline
        \text{total:}&1&15&40&45&24&5\\ \hline
        0:&1&\sbd&\sbd&\sbd&\sbd&\sbd\\
        1:&\sbd&1&\sbd&\sbd&\sbd&\sbd\\
        2:&\sbd&4&4&\sbd&\sbd&\sbd\\
        3:&\sbd&10&16&6&\sbd&\sbd\\
        4:&\sbd&\sbd&20&24&4&\sbd\\
        5:&\sbd&\sbd&\sbd&15&16&1\\
        6:&\sbd&\sbd&\sbd&\sbd&4&4\\
        7:&\sbd&\sbd&\sbd&\sbd&\sbd&\sbd\\
}
\hspace{0.01\linewidth}%
\BettiTable{0.18\linewidth}{7}{\underline{1},1,1,\underline{1},2,2}{0,3,6,7,8,11}{rrrrr}{
        &0&1&2&3&4\\ \hline
        \text{total:}&1&10&20&15&4\\ \hline
        0:&1&\sbd&\sbd&\sbd&\sbd\\
        1:&\sbd&1&\bd&\sbd&\sbd\\
        2:&\sbd&6&6&\bd&\sbd\\
        3:&\sbd&3&12&9&\bd\\
        4:&\sbd&\sbd&2&6&4\\
        5:&\sbd&\sbd&\sbd&\sbd&\sbd\\
}
\hspace{0.01\linewidth}%
\BettiTable{0.15\linewidth}{8}{\underline{1},1,1,1,\underline{1}}{0,3,6,7,8}{rrrr}{
        &0&1&2&3\\ \hline
        \text{total:}&1&6&9&4\\ \hline
        0:&1&\sbd&\sbd&\sbd\\
        1:&\sbd&2&\bd&\bd\\
        2:&\sbd&4&9&4\\
        3:&\sbd&\sbd&\sbd&\sbd\\
}
\hspace{0.01\linewidth}%
\BettiTable{0.18\linewidth}{9}{\underline{1},1,1,1,1,\underline{1}}{0,3,6,7,8,9}{rrrrr}{
        &0&1&2&3&4\\ \hline
        \text{total:}&1&6&13&12&4\\ \hline
        0:&1&\sbd&\sbd&\sbd&\sbd\\
        1:&\sbd&6&4&\bd&\bd\\
        2:&\sbd&\bd&9&12&4\\
        3:&\sbd&\sbd&\sbd&\sbd&\sbd\\
}
\hspace{0.01\linewidth}%
%%%%%%%%%%%%%%%%%%%%%
\end{adjustwidth}
\caption{$C = \Proj \mathbb F_3[x_0..x_5]/(x_2^2-x_1x_5,x_1x_2-x_0x_5,x_0x_2-x_2x_3-x_3^2-x_1x_4+x_2x_4+x_4^2-x_0x_5-x_1x_5+x_2x_5-x_3x_5,x_1^2-x_2x_3-x_3^2-x_1x_4+x_2x_4+x_4^2-x_0x_5-x_1x_5+x_2x_5-x_3x_5,x_0x_1-x_0x_3+x_0x_4-x_1x_4+x_3x_4-x_4^2+x_1x_5+x_2x_5+x_3x_5,x_0^2+x_0x_3+x_1x_3+x_2x_3-x_3^2-x_1x_4+x_2x_4+x_3x_4-x_4^2+x_0x_5+x_2x_5+x_3x_5,x_3^3+x_1x_3x_4-x_3^2x_4-x_1x_4^2-x_3x_4^2+x_4^3+x_1x_3x_5-x_3^2x_5+x_0x_4x_5-x_1x_4x_5-x_3x_4x_5+x_4^2x_5+x_0x_5^2+x_1x_5^2+x_2x_5^2-x_3x_5^2,x_2x_3^2-x_2x_4^2-x_0x_3x_5+x_1x_3x_5-x_2x_3x_5+x_3^2x_5-x_0x_4x_5-x_1x_4x_5-x_2x_4x_5+x_3x_4x_5+x_4^2x_5-x_0x_5^2+x_1x_5^2-x_3x_5^2,x_1x_3^2+x_2x_3x_4+x_3^2x_4-x_2x_4^2-x_4^3-x_1x_3x_5-x_2x_3x_5-x_0x_4x_5-x_1x_4x_5-x_2x_4x_5+x_3x_4x_5-x_4^2x_5+x_0x_5^2-x_2x_5^2+x_3x_5^2,x_0x_3^2+x_0x_3x_4+x_2x_3x_4+x_3^2x_4+x_0x_4^2-x_1x_4^2+x_2x_4^2-x_3x_4^2-x_0x_3x_5-x_1x_3x_5-x_2x_3x_5+x_3^2x_5-x_0x_4x_5+x_1x_4x_5-x_4^2x_5+x_0x_5^2-x_2x_5^2+x_3x_5^2)$}
\caption{$C = \Proj R_0$, $H(P) = \langle 3, 7, 8\rangle$}
\label{fig:betti-genus-4-semigroup-3-7-8}
\end{figure}

\newpage

\begin{figure}[H]
\begin{adjustwidth}{-0.5in}{-0.3in}
\centering

$H(P) = \langle 3,5\rangle$

%%%%%%%%%%%%%%%%%%%%%
% kk = ZZ/3
% H  = {3, 5} -- genus 4
% R  = kk[x_0..x_5]/(x_3*x_4-x_2*x_5,x_3^2-x_1*x_5,x_2*x_3-x_1*x_4,x_1*x_3-x_0*x_5,x_1*x_2-x_0*x_4,x_1^2-x_0*x_3,x_4^3-x_0^2*x_5-x_2^2*x_5-x_1*x_4*x_5-x_2*x_4*x_5+x_2*x_5^2-x_3*x_5^2,x_1*x_4^2-x_2^2*x_5,x_0^2*x_3+x_0*x_4^2-x_2*x_4^2+x_2^2*x_5+x_0*x_4*x_5-x_1*x_4*x_5+x_1*x_5^2,x_0^2*x_1+x_0*x_2*x_4-x_2^2*x_4+x_0*x_4^2+x_0*x_2*x_5-x_0*x_4*x_5+x_0*x_5^2,x_0^3+x_0*x_2^2-x_2^3+x_0*x_1*x_4+x_0*x_2*x_4-x_0*x_2*x_5+x_0*x_3*x_5)
% pt = ideal(-x_0,-x_1,-x_2,-x_3,-x_4)
\BettiTable{0.15\linewidth}{1}{\underline{1},\underline{3},5}{0,3,5}{rr}{
        &0&1\\ \hline
        \text{total:}&1&1\\ \hline
        0:&1&\sbd\\
        1:&\sbd&\sbd\\
        2:&\sbd&\sbd\\
        3:&\sbd&\sbd\\
        4:&\sbd&\sbd\\
        5:&\sbd&\sbd\\
        6:&\sbd&\sbd\\
        7:&\sbd&\sbd\\
        8:&\sbd&\sbd\\
        9:&\sbd&\bd\\
        10:&\sbd&\bd\\
        11:&\sbd&\bd\\
        12:&\sbd&\bd\\
        13:&\sbd&\bd\\
        14:&\sbd&1\\
        15:&\sbd&\sbd\\
}
\hspace{0.01\linewidth}%
\BettiTable{0.18\linewidth}{2}{\underline{1},2,3,\underline{3},4,5}{0,3,5,6,8,10}{rrrrr}{
        &0&1&2&3&4\\ \hline
        \text{total:}&1&9&16&9&1\\ \hline
        0:&1&\sbd&\sbd&\sbd&\sbd\\
        1:&\sbd&\sbd&\sbd&\sbd&\sbd\\
        2:&\sbd&\sbd&\sbd&\sbd&\sbd\\
        3:&\sbd&1&\sbd&\sbd&\sbd\\
        4:&\sbd&1&\sbd&\sbd&\sbd\\
        5:&\sbd&2&1&\sbd&\sbd\\
        6:&\sbd&1&2&\sbd&\sbd\\
        7:&\sbd&2&2&\sbd&\sbd\\
        8:&\sbd&1&3&1&\sbd\\
        9:&\sbd&1&3&1&\sbd\\
        10:&\sbd&\bd&2&2&\sbd\\
        11:&\sbd&\bd&2&1&\sbd\\
        12:&\sbd&\sbd&1&2&\bd\\
        13:&\sbd&\sbd&\bd&1&\bd\\
        14:&\sbd&\sbd&\bd&1&\bd\\
        15:&\sbd&\sbd&\sbd&\bd&\bd\\
        16:&\sbd&\sbd&\sbd&\bd&\bd\\
        17:&\sbd&\sbd&\sbd&\sbd&1\\
        18:&\sbd&\sbd&\sbd&\sbd&\sbd\\
}
\hspace{0.01\linewidth}%
\BettiTable{0.15\linewidth}{3}{\underline{1},\underline{1},2}{0,3,5}{rr}{
        &0&1\\ \hline
        \text{total:}&1&1\\ \hline
        0:&1&\sbd\\
        1:&\sbd&\sbd\\
        2:&\sbd&\sbd\\
        3:&\sbd&\bd\\
        4:&\sbd&\bd\\
        5:&\sbd&1\\
        6:&\sbd&\sbd\\
}
\hspace{0.01\linewidth}%
\BettiTable{0.18\linewidth}{4}{\underline{1},1,2,\underline{2},3,3}{0,3,5,8,10,12}{rrrrr}{
        &0&1&2&3&4\\ \hline
        \text{total:}&1&10&20&15&4\\ \hline
        0:&1&\sbd&\sbd&\sbd&\sbd\\
        1:&\sbd&\bd&\sbd&\sbd&\sbd\\
        2:&\sbd&1&\bd&\sbd&\sbd\\
        3:&\sbd&3&1&\sbd&\sbd\\
        4:&\sbd&3&4&\bd&\sbd\\
        5:&\sbd&3&7&2&\sbd\\
        6:&\sbd&\bd&6&5&\bd\\
        7:&\sbd&\sbd&2&5&1\\
        8:&\sbd&\sbd&\bd&3&2\\
        9:&\sbd&\sbd&\sbd&\bd&1\\
        10:&\sbd&\sbd&\sbd&\sbd&\sbd\\
}
\hspace{0.01\linewidth}%
\BettiTable{0.15\linewidth}{5}{\underline{1},1,\underline{1},2}{0,3,5,9}{rrr}{
        &0&1&2\\ \hline
        \text{total:}&1&3&2\\ \hline
        0:&1&\sbd&\sbd\\
        1:&\sbd&\bd&\sbd\\
        2:&\sbd&1&\bd\\
        3:&\sbd&2&1\\
        4:&\sbd&\bd&1\\
        5:&\sbd&\sbd&\sbd\\
}
\hspace{0.01\linewidth}%
\BettiTable{0.15\linewidth}{6}{\underline{1},1,1,\underline{1}}{0,3,5,6}{rrr}{
        &0&1&2\\ \hline
        \text{total:}&1&2&1\\ \hline
        0:&1&\sbd&\sbd\\
        1:&\sbd&1&\bd\\
        2:&\sbd&1&\bd\\
        3:&\sbd&\bd&1\\
        4:&\sbd&\sbd&\sbd\\
}
\hspace{0.01\linewidth}%
\BettiTable{0.20\linewidth}{7}{\underline{1},1,1,1,2,\underline{2},3}{0,3,5,6,13,14,21}{rrrrrr}{
        &0&1&2&3&4&5\\ \hline
        \text{total:}&1&15&40&45&24&5\\ \hline
        0:&1&\sbd&\sbd&\sbd&\sbd&\sbd\\
        1:&\sbd&1&\bd&\bd&\sbd&\sbd\\
        2:&\sbd&7&8&1&\bd&\sbd\\
        3:&\sbd&5&15&11&1&\sbd\\
        4:&\sbd&1&11&15&5&\bd\\
        5:&\sbd&1&5&14&11&1\\
        6:&\sbd&\sbd&1&4&7&4\\
        7:&\sbd&\sbd&\sbd&\sbd&\sbd&\sbd\\
}
\hspace{0.01\linewidth}%
\BettiTable{0.18\linewidth}{8}{\underline{1},1,1,1,\underline{1},2}{0,3,5,6,8,15}{rrrrr}{
        &0&1&2&3&4\\ \hline
        \text{total:}&1&10&20&15&4\\ \hline
        0:&1&\sbd&\sbd&\sbd&\sbd\\
        1:&\sbd&3&2&\bd&\sbd\\
        2:&\sbd&6&12&6&\bd\\
        3:&\sbd&1&6&9&4\\
        4:&\sbd&\sbd&\sbd&\sbd&\sbd\\
}
\hspace{0.01\linewidth}%
\BettiTable{0.18\linewidth}{9}{\underline{1},1,1,1,1,\underline{1}}{0,3,5,6,8,9}{rrrrr}{
        &0&1&2&3&4\\ \hline
        \text{total:}&1&10&20&15&4\\ \hline
        0:&1&\sbd&\sbd&\sbd&\sbd\\
        1:&\sbd&6&8&3&\bd\\
        2:&\sbd&4&12&12&4\\
        3:&\sbd&\sbd&\sbd&\sbd&\sbd\\
}
\hspace{0.01\linewidth}%
%%%%%%%%%%%%%%%%%%%%%
\end{adjustwidth}
\caption{$C = \Proj \mathbb F_3[x_0..x_5]/(x_3x_4-x_2x_5,x_3^2-x_1x_5,x_2x_3-x_1x_4,x_1x_3-x_0x_5,x_1x_2-x_0x_4,x_1^2-x_0x_3,x_4^3-x_0^2x_5-x_2^2x_5-x_1x_4x_5-x_2x_4x_5+x_2x_5^2-x_3x_5^2,x_1x_4^2-x_2^2x_5,x_0^2x_3+x_0x_4^2-x_2x_4^2+x_2^2x_5+x_0x_4x_5-x_1x_4x_5+x_1x_5^2,x_0^2x_1+x_0x_2x_4-x_2^2x_4+x_0x_4^2+x_0x_2x_5-x_0x_4x_5+x_0x_5^2,x_0^3+x_0x_2^2-x_2^3+x_0x_1x_4+x_0x_2x_4-x_0x_2x_5+x_0x_3x_5)$ }
\label{fig:betti-genus-4-semigroup-3-5}
\end{figure}

\newpage

\begin{figure}[H]
\begin{adjustwidth}{-0.5in}{-0.3in}
\centering

$H(P) = \langle 2,9\rangle$

%%%%%%%%%%%%%%%%%%%%%
% kk = ZZ/3
% H  = {2, 9} -- genus 4
% R  = kk[x_0..x_5]/(x_3^2-x_2*x_4,x_2*x_3-x_1*x_4,x_1*x_3-x_0*x_4,x_2^2-x_0*x_4,x_1*x_2-x_0*x_3,x_1^2-x_0*x_2,x_0^2*x_4-x_0*x_1*x_4-x_1*x_4^2+x_2*x_4^2-x_4^3+x_3*x_5^2,x_0^2*x_3-x_0*x_1*x_4-x_0*x_4^2+x_2*x_4^2-x_3*x_4^2-x_4^3+x_2*x_5^2+x_3*x_5^2,x_0^2*x_2-x_0*x_1*x_4-x_0*x_3*x_4-x_3*x_4^2-x_4^3+x_1*x_5^2+x_2*x_5^2+x_3*x_5^2,x_0^2*x_1-x_0*x_1*x_4-x_0*x_2*x_4-x_1*x_4^2-x_3*x_4^2-x_4^3+x_0*x_5^2+x_1*x_5^2+x_2*x_5^2+x_3*x_5^2)
% pt = ideal(-x_0,-x_1,-x_2,-x_3,-x_4)
\BettiTable{0.15\linewidth}{1}{\underline{1},\underline{2},9}{0,2,9}{rr}{
        &0&1\\ \hline
        \text{total:}&1&1\\ \hline
        0:&1&\sbd\\
        1:&\sbd&\sbd\\
        2:&\sbd&\sbd\\
        3:&\sbd&\sbd\\
        4:&\sbd&\sbd\\
        5:&\sbd&\sbd\\
        6:&\sbd&\sbd\\
        7:&\sbd&\sbd\\
        8:&\sbd&\sbd\\
        9:&\sbd&\sbd\\
        10:&\sbd&\sbd\\
        11:&\sbd&\sbd\\
        12:&\sbd&\sbd\\
        13:&\sbd&\sbd\\
        14:&\sbd&\sbd\\
        15:&\sbd&\sbd\\
        16:&\sbd&\sbd\\
        17:&\sbd&1\\
        18:&\sbd&\sbd\\
}
\hspace{0.01\linewidth}%
\BettiTable{0.15\linewidth}{2}{\underline{1},\underline{1},5}{0,2,9}{rr}{
        &0&1\\ \hline
        \text{total:}&1&1\\ \hline
        0:&1&\sbd\\
        1:&\sbd&\sbd\\
        2:&\sbd&\sbd\\
        3:&\sbd&\sbd\\
        4:&\sbd&\sbd\\
        5:&\sbd&\sbd\\
        6:&\sbd&\sbd\\
        7:&\sbd&\sbd\\
        8:&\sbd&\sbd\\
        9:&\sbd&1\\
        10:&\sbd&\sbd\\
}
\hspace{0.01\linewidth}%
\BettiTable{0.15\linewidth}{3}{\underline{1},1,\underline{2},3}{0,2,6,9}{rrr}{
        &0&1&2\\ \hline
        \text{total:}&1&2&1\\ \hline
        0:&1&\sbd&\sbd\\
        1:&\sbd&\bd&\sbd\\
        2:&\sbd&1&\sbd\\
        3:&\sbd&\bd&\bd\\
        4:&\sbd&\bd&\bd\\
        5:&\sbd&1&\bd\\
        6:&\sbd&\bd&\bd\\
        7:&\sbd&\bd&1\\
        8:&\sbd&\sbd&\sbd\\
}
\hspace{0.01\linewidth}%
\BettiTable{0.15\linewidth}{4}{\underline{1},1,\underline{1},3,3}{0,2,4,9,11}{rrrr}{
        &0&1&2&3\\ \hline
        \text{total:}&1&6&8&3\\ \hline
        0:&1&\sbd&\sbd&\sbd\\
        1:&\sbd&1&\sbd&\sbd\\
        2:&\sbd&\bd&\sbd&\sbd\\
        3:&\sbd&2&2&\sbd\\
        4:&\sbd&\bd&\bd&\sbd\\
        5:&\sbd&3&4&1\\
        6:&\sbd&\sbd&\bd&\bd\\
        7:&\sbd&\sbd&2&2\\
        8:&\sbd&\sbd&\sbd&\sbd\\
}
\hspace{0.01\linewidth}%
\BettiTable{0.18\linewidth}{5}{\underline{1},1,1,2,\underline{2},3}{0,2,4,9,10,15}{rrrrr}{
        &0&1&2&3&4\\ \hline
        \text{total:}&1&9&17&12&3\\ \hline
        0:&1&\sbd&\sbd&\sbd&\sbd\\
        1:&\sbd&1&\bd&\sbd&\sbd\\
        2:&\sbd&2&2&\bd&\sbd\\
        3:&\sbd&4&4&1&\sbd\\
        4:&\sbd&1&6&3&\bd\\
        5:&\sbd&1&4&4&1\\
        6:&\sbd&\bd&1&4&1\\
        7:&\sbd&\sbd&\bd&\bd&1\\
        8:&\sbd&\sbd&\sbd&\sbd&\sbd\\
}
\hspace{0.01\linewidth}%
\BettiTable{0.18\linewidth}{6}{\underline{1},1,1,\underline{1},2,2}{0,2,4,6,9,11}{rrrrr}{
        &0&1&2&3&4\\ \hline
        \text{total:}&1&9&16&9&1\\ \hline
        0:&1&\sbd&\sbd&\sbd&\sbd\\
        1:&\sbd&3&2&\sbd&\sbd\\
        2:&\sbd&3&6&3&\sbd\\
        3:&\sbd&3&6&3&\bd\\
        4:&\sbd&\bd&2&3&\bd\\
        5:&\sbd&\sbd&\bd&\bd&1\\
        6:&\sbd&\sbd&\sbd&\sbd&\sbd\\
}
\hspace{0.01\linewidth}%
\BettiTable{0.25\linewidth}{7}{\underline{1},1,1,1,2,2,2,\underline{2},3}{0,2,4,6,9,11,13,14,21}{rrrrrrrr}{
        &0&1&2&3&4&5&6&7\\ \hline
        \text{total:}&1&28&112&210&224&140&48&7\\ \hline
        0:&1&\sbd&\sbd&\sbd&\sbd&\sbd&\sbd&\sbd\\
        1:&\sbd&3&2&\bd&\sbd&\sbd&\sbd&\sbd\\
        2:&\sbd&9&18&9&\bd&\sbd&\sbd&\sbd\\
        3:&\sbd&12&39&42&15&\bd&\sbd&\sbd\\
        4:&\sbd&3&35&72&51&11&\bd&\sbd\\
        5:&\sbd&1&15&57&76&36&3&\sbd\\
        6:&\sbd&\sbd&3&27&60&51&15&\bd\\
        7:&\sbd&\sbd&\sbd&3&21&36&21&3\\
        8:&\sbd&\sbd&\sbd&\sbd&1&6&9&4\\
        9:&\sbd&\sbd&\sbd&\sbd&\sbd&\sbd&\sbd&\sbd\\
}
\hspace{0.01\linewidth}%
\BettiTable{0.25\linewidth}{8}{\underline{1},1,1,1,\underline{1},2,2,2,2}{0,2,4,6,8,9,11,13,15}{rrrrrrrr}{
        &0&1&2&3&4&5&6&7\\ \hline
        \text{total:}&1&28&112&210&224&140&48&7\\ \hline
        0:&1&\sbd&\sbd&\sbd&\sbd&\sbd&\sbd&\sbd\\
        1:&\sbd&6&8&3&\sbd&\sbd&\sbd&\sbd\\
        2:&\sbd&12&36&36&12&\sbd&\sbd&\sbd\\
        3:&\sbd&10&48&84&64&18&\sbd&\sbd\\
        4:&\sbd&\sbd&20&72&96&56&12&\sbd\\
        5:&\sbd&\sbd&\sbd&15&48&54&24&3\\
        6:&\sbd&\sbd&\sbd&\sbd&4&12&12&4\\
        7:&\sbd&\sbd&\sbd&\sbd&\sbd&\sbd&\sbd&\sbd\\
}
\hspace{0.01\linewidth}%
\BettiTable{0.18\linewidth}{9}{\underline{1},1,1,1,1,\underline{1}}{0,2,4,6,8,9}{rrrrr}{
        &0&1&2&3&4\\ \hline
        \text{total:}&1&10&20&15&4\\ \hline
        0:&1&\sbd&\sbd&\sbd&\sbd\\
        1:&\sbd&6&8&3&\bd\\
        2:&\sbd&4&12&12&4\\
        3:&\sbd&\sbd&\sbd&\sbd&\sbd\\
}
\hspace{0.01\linewidth}%
%%%%%%%%%%%%%%%%%%%%%
\end{adjustwidth}
\caption{$C = \Proj \mathbb F_3[x_0..x_5]/(x_3^2-x_2x_4,x_2x_3-x_1x_4,x_1x_3-x_0x_4,x_2^2-x_0x_4,x_1x_2-x_0x_3,x_1^2-x_0x_2,x_0^2x_4-x_0x_1x_4-x_1x_4^2+x_2x_4^2-x_4^3+x_3x_5^2,x_0^2x_3-x_0x_1x_4-x_0x_4^2+x_2x_4^2-x_3x_4^2-x_4^3+x_2x_5^2+x_3x_5^2,x_0^2x_2-x_0x_1x_4-x_0x_3x_4-x_3x_4^2-x_4^3+x_1x_5^2+x_2x_5^2+x_3x_5^2,x_0^2x_1-x_0x_1x_4-x_0x_2x_4-x_1x_4^2-x_3x_4^2-x_4^3+x_0x_5^2+x_1x_5^2+x_2x_5^2+x_3x_5^2)$}
\label{fig:betti-genus-4-semigroup-2-9}
\end{figure}

  %% \section{Betti tables of weighted embeddings with $g = 5$}
  %% \input{figs/fig-genus-5-betti-tables.tex}
  % Examples that get very ample too soon,
  % then strictly ample, then very ample again
  % {5,6,7,9} is va at 10,11
  % {5,6,7,8} is va at 8,11
  % {5,6,8,9} is va at 9,10,11
  % {4,7,9,10} is va at 10,11
  % {4,6,7}    is va at 7,8,11 !!
\end{appendix}

\bibliography{bibliography}

\end{document}